\documentclass[a4paper, reqno, 11pt]{amsart}

\usepackage{accents}
\usepackage{mathtools,amssymb,stmaryrd}

\usepackage[colorlinks]{hyperref}
\hypersetup{
  pdftitle   = {Fukaya categories of orbifold surfaces in representation theory},
  pdfauthor  = {Severin Barmeier, Zhengfang Wang},
  pdfcreator = {},
  linkcolor  = blue,
  citecolor  = blue,
  urlcolor   = blue
}

\usepackage{xcolor}
\definecolor{stopcolour}{HTML}{007A99}
\definecolor{arccolour}{HTML}{00AB57}
\definecolor{ribboncolour}{HTML}{F7971C}

\newcommand{\mucirc}{\accentset{\circ}{\mu}}
\newcommand{\mutimes}{\accentset{\times}{\mu}}
\newcommand{\muotimes}{\accentset{\otimes}{\mu}}
\newcommand{\mubar}{\bar\mu}

\usepackage{tikz}
\usetikzlibrary{backgrounds,matrix,arrows.meta,decorations.markings,decorations.pathmorphing,shapes,positioning,calc}
\tikzset{>=stealth}
\newcommand{\W}{\mathcal W \mkern-2mu}
\renewcommand{\epsilon}{\varepsilon}
\renewcommand{\to}{\mathrel{\tikz[baseline]\draw[ ->,line width=.4pt] (0ex,0.65ex) -- (3ex,0.65ex);}}
\renewcommand{\mapsto}{\mathrel{\tikz[baseline]\draw[|->,line width=.4pt] (0pt,0.65ex) -- (3ex,0.65ex);}}

\renewcommand{\rightrightarrows}{\mathrel{\tikz[baseline]{\draw[->,line width=.4pt] (0ex,0.25ex) -- (3ex,0.25ex); \draw[->] (0ex,1.05ex) -- (3ex,1.05ex);}}}

\renewcommand{\leftrightarrow}{\mathrel{\tikz[baseline]\draw[<->,line width=.4pt] (0ex,0.65ex) -- (3.5ex,0.65ex);}}

\newcommand{\toarg}[1]{\mathrel{\tikz[baseline]\path[->,line width=.4pt] (0ex,0.65ex) edge node[above=-.4ex, overlay, font=\scriptsize] {$#1$} (3.5ex,.65ex);}}

\usepackage{makecell}

\newtheorem{theorem}[equation]{Theorem}
\newtheorem*{theorem*}{Theorem}
\newtheorem{proposition}[equation]{Proposition}
\newtheorem{lemma}[equation]{Lemma}
\newtheorem*{lemma*}{Lemma}

\newtheorem{conjecture}[equation]{Conjecture}
\newtheorem*{corollary*}{Corollary}
\theoremstyle{definition}
\newtheorem{definition}[equation]{Definition}
\newtheorem{example}[equation]{Example}

\newtheorem{notation}[equation]{Notation}
\theoremstyle{remark}
\newtheorem{remark}[equation]{Remark}

\DeclareMathOperator{\End}{End}

\renewcommand{\H}{\mathrm{H}}

\DeclareMathOperator{\HH}{HH}
\DeclareMathOperator{\hocolim}{hocolim}
\DeclareMathOperator{\Hom}{Hom}

\newcommand{\id}{\mathrm{id}}

\DeclareMathOperator{\per}{per}

\newcommand{\s}{\mathrm{s}}
\DeclareMathOperator{\Sing}{Sing}

\DeclareMathOperator{\tw}{tw}

\begin{document}

\title[Fukaya categories of orbifolds in representation theory]{Fukaya categories of orbifold surfaces in representation theory}

\author{Severin Barmeier}
\address{Mathematical Institute, University of Cologne, Weyertal 86-90, 50931 Köln, Germany}
\email{s.barmeier@gmail.com}

\author{Zhengfang Wang}
\address{Nanjing University, School of Mathematics, Nanjing 210093, Jiangsu, China}
\email{zhengfangw@nju.edu.cn}



\begin{abstract}
We give an introduction to partially wrapped Fukaya categories of surfaces with orbifold singularities. Dissecting an orbifold surface $\mathbf S$ into polygons, certain dissections give rise to formal generators, inducing a triangulated equivalence between the derived Fukaya category of $\mathbf S$ and the perfect derived category of a graded associative algebra. This provides a geometric means for obtaining associative algebras --- conjecturally all --- which are derived equivalent to skew-gentle algebras.

We include a new perspective on the partially wrapped Fukaya category of an orbifold disk which serves as a local model for the Fukaya categories of general orbifold surfaces. This perspective yields an equivalence between the perfect derived category of a quiver of type $\mathrm D_{n+1}$ and the perfect derived category of a graded quiver of type $\widetilde{\mathrm A}_{n-1}$, the latter being equipped with quadratic zero relations and a nontrivial A$_\infty$ structure. This equivalence elucidates the relationship between skew-gentle algebras and orbifold surfaces, and the role of deformation theory in this relationship.
\end{abstract}

\maketitle

\section{Introduction}

\subsection{Gentle algebras and Fukaya categories of surfaces}

The study of gentle and skew-gentle algebras has a long and rich history in representation theory of algebras. Gentle algebras were introduced by Assem and Skowroński \cite{assemskowronski} for their tractable representation theory, with a generalization to skew-gentle algebras given by Geiss and de la Peña \cite{geisspena}. Gentle and skew-gentle algebras belong to the few known classes of algebras which are not only tame, but also derived tame. This makes it possible to classify the indecomposable objects in the bounded derived category of a gentle or skew-gentle algebra \cite{bekkertmerklen,burbandrozd}.

The trivial extension of a gentle algebra is a symmetric special biserial algebra and hence a Brauer graph algebra. Schroll \cite{schroll} showed how the corresponding Brauer graph can be constructed directly from the gentle algebra: the maximal paths in a gentle algebra $A$ (together with certain idempotents) give rise to vertices of its associated Brauer graph $G$, vertices of the underlying quiver of $A$ give rise to edges of $G$, and arrows of the quiver of $A$ give rise to a linear ordering of the half-edges at each vertex of $G$. Extending this linear order to a cyclic order, one obtains a ribbon graph which can be viewed as a deformation retract of a compact oriented surface with boundary. It turns out that much of the representation theory of a gentle algebra can be encoded geometrically on this surface. In particular, the surface serves as a geometric model of the bounded derived category of a gentle algebra \cite{opperplamondonschroll}. 

The study of gentle algebras received a new impetus through the work of Haiden, Katzarkov and Kontsevich \cite{haidenkatzarkovkontsevich} who showed that the partially wrapped Fukaya category of a {\it graded} surface admits a formal generator whose endomorphism algebra is a {\it graded} gentle algebra. The surface, including the boundary stops, coincides with that obtained via the Brauer graph \cite{schroll} but is additionally equipped with a {\it line field} which endows the Fukaya category -- and any associated gentle algebra -- with a $\mathbb Z$-grading. Conversely, for an arbitrary grading of a gentle algebra, there always exists a line field on the surface inducing the given grading \cite{lekilipolishchuk2}. We usually denote the data of the surface $S$, the set $\Sigma \subset \partial S$ of stops and line field $\eta$ by $\mathbf S = (S, \Sigma, \eta)$.

The general theory of Fukaya categories drastically simplifies in dimension $2$ and the symplectic viewpoint thus provides an accessible arsenal of geometric techniques that inform and illuminate the representation theory of gentle algebras. For example, Lekili and Polishchuk \cite{lekilipolishchuk2} show how the classical Avella-Alaminos--Geiss invariant of a gentle algebra \cite{avellaalaminosgeiss} can be obtained from the winding numbers of the line field along the boundary components of the surface. In fact, this invariant can be enhanced to a complete derived invariant \cite{lekilipolishchuk2,amiotplamondonschroll,jinschrollwang}. 

The perspective of Fukaya categories naturally leads to a paradigm shift in the study of gentle algebras and their derived categories. The relationship
\begin{equation}
\label{eq:correspondencegentle}
\begin{tikzpicture}[x=1em,y=1em,baseline=-2.25ex]
\node[font=\footnotesize, align=center, anchor=north] (L) at (0, 0) {\normalsize $\mathbf S$ \\ smooth graded surface \\ with stops in the boundary};
\node[font=\footnotesize, align=center, anchor=north] (R) at (10, 0) {\normalsize $A$ \\ graded gentle \\ algebra};
\path[<->, line width=.5pt] (4,-.75) edge (6.5,-.75);
\end{tikzpicture}
\end{equation}
is then witnessed by an equivalence
\begin{align}\label{align:equivalencegentle}
\W (\mathbf S) \simeq \tw (A)^\natural
\end{align}
of pretriangulated A$_\infty$ categories. Here $\W (\mathbf S)$ is the partially wrapped Fukaya category of $\mathbf S$ and $\tw (A)$ is the category of twisted complexes over $A$ (viewed as a DG category) and $^\natural$ denotes the idempotent completion. Twisted complexes play a central role in the study of DG enhancements of triangulated categories \cite{bondalkapranov,keller0} and in symplectic geometry \cite{kontsevich1,seidel2}. The equivalence \eqref{align:equivalencegentle} thus places the theory of gentle algebras in the framework of DG and A$_\infty$ categories and their homotopy theory.

Passing to homotopy categories, the equivalence \eqref{align:equivalencegentle} induces a triangulated equivalence
\begin{equation}
\label{eq:triangulatedequivalence}
\mathrm D \W (\mathbf S) \simeq \per (A)
\end{equation}
between the ``derived'' partially wrapped Fukaya category $\mathrm D \W (\mathbf S) := \H^0 \W (\mathbf S)$ of $\mathbf S$ and the perfect derived category $\per (A) \simeq \H^0 \tw (A)^\natural$ of $A$. Fukaya and Fukaya--Seidel categories are often constructed as twisted completions of a collection of generating Lagrangians \cite{kontsevich1,seidel1}. Extending the notion of (bounded) twisted complexes to unbounded twisted complexes \cite{annologvinenko}, it is possible to classify indecomposable objects and morphisms in the bounded derived category of any finite-dimensional graded gentle algebra \cite{opperplamondonschroll}.

There is a vast and rapidly growing literature on gentle (and skew-gentle) algebras and we apologize for all sins of omission that are bound to occur in our cursory account. We refer to the excellent recent survey \cite{schrollicm} for further references.

\subsection{Gluing properties of gentle algebras, ribbon graphs and Fukaya categories}
\label{subsection:gluingproperties}

It turns out that the equivalence $\W (\mathbf S) \simeq \tw (A)^\natural$ \eqref{align:equivalencegentle} can be assembled from simple pieces via ``gluing'' which manifests itself in several equivalent ways: 
\begin{description}
\item[{\bf algebraic}] gluing the quiver with relations for the gentle algebra $A$ by identifying pairs of vertices in the building blocks $\mathrm A_n$ and $\widetilde{\mathrm A}_{n-1}$ (for $n \geq 1$), both considered with full quadratic monomial relations
\item[{\bf topological}] gluing the ribbon graph of $A$ by identifying pairs of edges in the building blocks consisting of one vertex of valency $n$ and $n$ edges with either a linear order (corresponding to $\mathrm A_n$) or a cyclic order (corresponding to $\widetilde{\mathrm A}_{n-1}$)
\item[{\bf symplectic}] gluing the graded surface $\mathbf S$ out of graded ``sectors'' which are polygons with exactly one stop in the boundary or a ``puncture'' modelled by a boundary component with an entire boundary stop.
\end{description}
We now give some more details for each of these perspectives.

\paragraph{Gluing properties of gentle algebras}

Throughout we work over a field $\Bbbk$ of characteristic $0$. Recall that a (possibly infinite-dimensional) {\it gentle algebra} $A = \Bbbk Q / I$ is the path algebra of a finite quiver $Q$ modulo an ideal $I$ generated by quadratic monomial relations such that at each vertex of $Q$ there are at most two incoming and two outgoing arrows, and the relations are such that for every arrow $q \in Q_1$, there is at most one precomposable arrow $p$ such that $q p \in I$ and at most one precomposable arrow $p'$ such that $q p' \not\in I$, with an analogous condition for postcomposable arrows \cite{assemskowronski}. A {\it grading} on $A$ is simply given by a function $\lvert - \rvert \colon Q_1 \to \mathbb Z$ which assigns to an arrow $p \in Q_1$ its degree $\lvert p \rvert \in \mathbb Z$ and we may view $A$ as DG algebra with trivial differential.

It is a telling observation that any graded gentle algebra $A = \Bbbk Q / I$ can itself be viewed as the homotopy colimit of a diagram of graded quivers of types $\mathrm A_n$ and $\widetilde{\mathrm A}_{n-1}$ (for $n \geq 1$), linearly resp.\ cyclically oriented and both considered with full quadratic monomial relations \cite{haidenkatzarkovkontsevich,barmeierschrollwang}. In other words, any graded gentle algebra can be ``built'' out of the following building blocks:
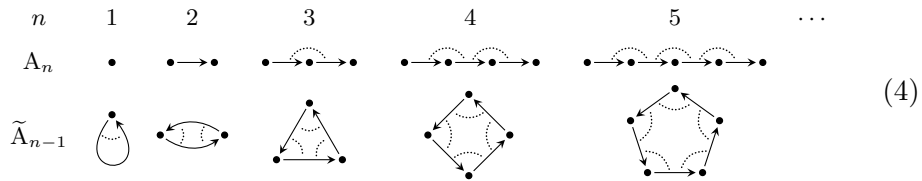
\begin{equation}
\label{eq:buildingblocks}
\begin{tikzpicture}[x=2em,y=2em,decoration={markings,mark=at position 0.99 with {\arrow[black]{Stealth[length=4.8pt]}}}, baseline=-1.3em]
\begin{scope}[yshift=1.5em]
\node[font=\footnotesize] at (-1.25,0) {$n$\strut};
\end{scope}
\node[font=\footnotesize] at (-1.25,0) {$\mathrm A_n$};
\begin{scope} 
\begin{scope}[yshift=1.5em]
\node[font=\footnotesize] at (0,0) {$1$\strut};
\end{scope}
\draw[line width=.5pt, fill=black] (0,0) circle(.1em);
\end{scope}
\begin{scope}[xshift=2em] 
\begin{scope}[yshift=1.5em]
\node[font=\footnotesize] at (.375,0) {$2$\strut};
\end{scope}
\draw[line width=.5pt, fill=black] (0,0) circle(.1em);
\draw[line width=.5pt, fill=black] (.75,0) circle(.1em);
\node[circle, minimum size=.1em, outer sep=2pt, inner sep=0] (0) at (0,0) {};
\node[circle, minimum size=.1em, outer sep=2pt, inner sep=0] (1) at (.75,0) {};
\path[-stealth] (0) edge (1);
\end{scope}
\begin{scope}[xshift=5.25em] 
\begin{scope}[yshift=1.5em]
\node[font=\footnotesize] at (.75,0) {$3$\strut};
\end{scope}
\draw[line width=.5pt, fill=black] (0,0) circle(.1em);
\draw[line width=.5pt, fill=black] (.75,0) circle(.1em);
\draw[line width=.5pt, fill=black] (1.5,0) circle(.1em);
\node[circle, minimum size=.1em, outer sep=2pt, inner sep=0] (0) at (0,0) {};
\node[circle, minimum size=.1em, outer sep=2pt, inner sep=0] (1) at (.75,0) {};
\node[circle, minimum size=.1em, outer sep=2pt, inner sep=0] (2) at (1.5,0) {};
\path[-stealth] (0) edge (1) (1) edge (2);
\draw[line cap=round, dash pattern=on 0pt off 1.2pt, line width=.6pt] (.75,0) +(-.65em,.12em) to[bend left=62, looseness=1.22] +(.55em,.12em);
\end{scope}
\begin{scope}[xshift=10em] 
\begin{scope}[yshift=1.5em]
\node[font=\footnotesize] at (.75+.375,0) {$4$\strut};
\end{scope}
\draw[line width=.5pt, fill=black] (0,0) circle(.1em);
\draw[line width=.5pt, fill=black] (.75,0) circle(.1em);
\draw[line width=.5pt, fill=black] (1.5,0) circle(.1em);
\draw[line width=.5pt, fill=black] (2.25,0) circle(.1em);
\node[circle, minimum size=.1em, outer sep=2pt, inner sep=0] (0) at (0,0) {};
\node[circle, minimum size=.1em, outer sep=2pt, inner sep=0] (1) at (.75,0) {};
\node[circle, minimum size=.1em, outer sep=2pt, inner sep=0] (2) at (1.5,0) {};
\node[circle, minimum size=.1em, outer sep=2pt, inner sep=0] (3) at (2.25,0) {};
\path[-stealth] (0) edge (1) (1) edge (2) (2) edge (3);
\draw[line cap=round, dash pattern=on 0pt off 1.2pt, line width=.6pt] (.75,0) +(-.65em,.12em) to[bend left=62, looseness=1.22] +(.55em,.12em) (1.5,0) +(-.65em,.12em) to[bend left=62, looseness=1.22] +(.55em,.12em);
\end{scope}
\begin{scope}[xshift=16.25em] 
\begin{scope}[yshift=1.5em]
\node[font=\footnotesize] at (1.5,0) {$5$\strut};
\end{scope}
\draw[line width=.5pt, fill=black] (0,0) circle(.1em);
\draw[line width=.5pt, fill=black] (.75,0) circle(.1em);
\draw[line width=.5pt, fill=black] (1.5,0) circle(.1em);
\draw[line width=.5pt, fill=black] (2.25,0) circle(.1em);
\draw[line width=.5pt, fill=black] (3,0) circle(.1em);
\node[circle, minimum size=.1em, outer sep=2pt, inner sep=0] (0) at (0,0) {};
\node[circle, minimum size=.1em, outer sep=2pt, inner sep=0] (1) at (.75,0) {};
\node[circle, minimum size=.1em, outer sep=2pt, inner sep=0] (2) at (1.5,0) {};
\node[circle, minimum size=.1em, outer sep=2pt, inner sep=0] (3) at (2.25,0) {};
\node[circle, minimum size=.1em, outer sep=2pt, inner sep=0] (4) at (3,0) {};
\path[-stealth] (0) edge (1) (1) edge (2) (2) edge (3) (3) edge (4);
\draw[line cap=round, dash pattern=on 0pt off 1.2pt, line width=.6pt] (.75,0) +(-.65em,.12em) to[bend left=62, looseness=1.22] +(.55em,.12em) (1.5,0) +(-.65em,.12em) to[bend left=62, looseness=1.22] +(.55em,.12em) (2.25,0) +(-.65em,.12em) to[bend left=62, looseness=1.22] +(.55em,.12em);
\end{scope}
\begin{scope}[xshift=21.5em] 
\begin{scope}[yshift=1.5em]
\node[font=\footnotesize] at (1.25,0) {$\dotsb$\strut};
\end{scope}
\end{scope}
\begin{scope}[yshift=-2.5em]
\node[font=\footnotesize] at (-1.25,0) {$\widetilde{\mathrm A}_{n-1}$};
\begin{scope}
\draw[line width=.5pt, fill=black] (0,.35) circle(.1em);
\node[circle, minimum size=.1em, outer sep=2pt, inner sep=0] (0) at (0,.35) {};
\path[-stealth] (0) edge[out=235, in=305, looseness=24] (0);
\draw[line cap=round, dash pattern=on 0pt off 1.2pt, line width=.6pt] (0,.35) +(246:.8em) to[bend right=30, looseness=1.22] +(294:.8em);
\end{scope}
\begin{scope}[xshift=2.75em, rotate=90] 
\draw[line width=.5pt, fill=black] (90:.55) circle(.1em);
\draw[line width=.5pt, fill=black] (270:.55) circle(.1em);
\node[circle, minimum size=.1em, outer sep=2pt, inner sep=0] (0) at (270:.55) {};
\node[circle, minimum size=.1em, outer sep=2pt, inner sep=0] (1) at (90:.55) {};
\path[-stealth] (0) edge[bend right=35] (1) (1) edge[bend right=35] (0);
\draw[line cap=round, dash pattern=on 0pt off 1.2pt, line width=.6pt] (270:.55) +(69:.76em) to[bend right=30, looseness=1.1] +(111:.76em) (90:.55) +(69+180:.76em) to[bend right=30, looseness=1.1] +(111+180:.76em);
\end{scope}
\begin{scope}[xshift=6.75em, yshift=-.2em] 
\draw[line width=.5pt, fill=black] (90:.65) circle(.1em);
\draw[line width=.5pt, fill=black] (210:.65) circle(.1em);
\draw[line width=.5pt, fill=black] (330:.65) circle(.1em);
\node[circle, minimum size=.1em, outer sep=2pt, inner sep=0] (0) at (210:.65) {};
\node[circle, minimum size=.1em, outer sep=2pt, inner sep=0] (1) at (330:.65) {};
\node[circle, minimum size=.1em, outer sep=2pt, inner sep=0] (2) at (90:.65) {};
\path[-stealth] (0) edge (1) (1) edge (2) (2) edge (0);
\draw[line cap=round, dash pattern=on 0pt off 1.2pt, line width=.6pt] (90:.65) +(-69:.9em) to[bend left=30, looseness=1.2] +(-111:.9em) (90+120:.65) +(-69+120:.9em) to[bend left=30, looseness=1.2] +(-111+120:.9em) (90-120:.65) +(-69-120:.9em) to[bend left=30, looseness=1.2] +(-111-120:.9em); 
\end{scope}
\begin{scope}[xshift=10.8em] 
\draw[line width=.5pt, fill=black] (0,0) circle(.1em);
\draw[line width=.5pt, fill=black] (.7,.7) circle(.1em);
\draw[line width=.5pt, fill=black] (.7,-.7) circle(.1em);
\draw[line width=.5pt, fill=black] (1.4,0) circle(.1em);
\node[circle, minimum size=.1em, outer sep=2pt, inner sep=0] (0) at (0,0) {};
\node[circle, minimum size=.1em, outer sep=2pt, inner sep=0] (1) at (.7,-.7) {};
\node[circle, minimum size=.1em, outer sep=2pt, inner sep=0] (2) at (1.4,0) {};
\node[circle, minimum size=.1em, outer sep=2pt, inner sep=0] (3) at (.7,.7) {};
\path[-stealth] (0) edge (1) (1) edge (2) (2) edge (3) (3) edge (0);
\draw[line cap=round, dash pattern=on 0pt off 1.2pt, line width=.6pt] (0,0) +(37:.8em) arc[start angle=37, end angle=-37, radius=.8em] (.7,-.7) +(127:.8em) arc[start angle=127, end angle=53, radius=.8em] (1.4,0) +(217:.8em) arc[start angle=217, end angle=143, radius=.8em] (.7,.7) +(-53:.8em) arc[start angle=-53, end angle=-127, radius=.8em];
\end{scope}
\begin{scope}[xshift=19.25em] 
\draw[line width=.5pt, fill=black] (90:.8) circle(.1em);
\draw[line width=.5pt, fill=black] (162:.8) circle(.1em);
\draw[line width=.5pt, fill=black] (234:.8) circle(.1em);
\draw[line width=.5pt, fill=black] (306:.8) circle(.1em);
\draw[line width=.5pt, fill=black] (18:.8) circle(.1em);
\node[circle, minimum size=.1em, outer sep=2pt, inner sep=0] (0) at (162:.8) {};
\node[circle, minimum size=.1em, outer sep=2pt, inner sep=0] (1) at (234:.8) {};
\node[circle, minimum size=.1em, outer sep=2pt, inner sep=0] (2) at (306:.8) {};
\node[circle, minimum size=.1em, outer sep=2pt, inner sep=0] (3) at (18:.8) {};
\node[circle, minimum size=.1em, outer sep=2pt, inner sep=0] (4) at (90:.8) {};
\path[-stealth] (0) edge (1) (1) edge (2) (2) edge (3) (3) edge (4) (4) edge (0);
\draw[line cap=round, dash pattern=on 0pt off 1.2pt, line width=.6pt] (90:.8) +(-135:.75em) to[bend right=32, looseness=1.2] +(-45:.75em) (90+72:.8) +(-135+72:.75em) to[bend right=32, looseness=1.2] +(-45+72:.75em) (90+144:.8) +(-135+144:.75em) to[bend right=32, looseness=1.2] +(-45+144:.75em) (90+216:.8) +(-135+216:.75em) to[bend right=32, looseness=1.2] +(-45+216:.75em) (90-72:.8) +(-135-72:.75em) to[bend right=32, looseness=1.2] +(-45-72:.75em);
\end{scope}
\end{scope}
\end{tikzpicture}
\end{equation}
For each of these quivers $Q$, the dotted lines connecting two composable arrows indicate that the composition of these arrows belongs to the ideal $I$, i.e.\ the corresponding composition is zero in the quotient algebra $\Bbbk Q / I$. Each building block may be graded by assigning a degree to each arrow.

\begin{example}
To give an example of a gluing of the building blocks described in \eqref{eq:buildingblocks}, consider the algebras $A$ and $A'$ given by the building blocks of type $\mathrm A_4$ and $\widetilde{\mathrm A}_3$, respectively, as in \eqref{eq:buildingblocks}. For any chosen grading, the algebras $A$ and $A'$ can be glued along a vertex as follows:
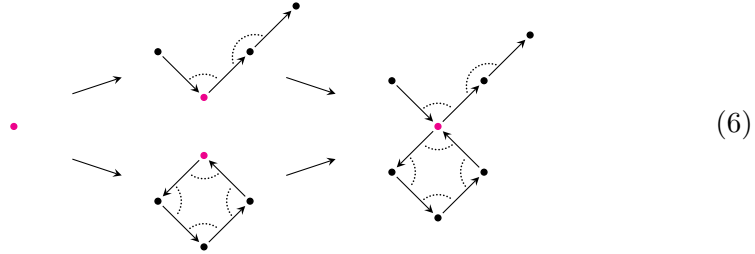
\begin{equation}
\label{eq:colimitgentle}
\begin{tikzpicture}[x=2.25em,y=2.25em,decoration={markings,mark=at position 0.99 with {\arrow[black]{Stealth[length=4.8pt]}}}, baseline=-.2em]
\begin{scope}
\draw[-stealth] (-2,.5) to (-1.25,.75);
\draw[-stealth] (-2,-.5) to (-1.25,-.75);
\draw[stealth-] (2,.5) to (1.25,.75);
\draw[stealth-] (2,-.5) to (1.25,-.75);
\end{scope}
\begin{scope}[xshift=-6.5em] 
\draw[line width=.5pt, color=magenta, fill=magenta] (0,0) circle(.1em);
\end{scope}
\begin{scope}[yshift=1em] 
\draw[line width=.5pt, color=magenta, fill=magenta] (0,0) circle(.1em);
\draw[line width=.5pt, fill=black] (-.7,.7) circle(.1em);
\draw[line width=.5pt, fill=black] (.7,.7) circle(.1em);
\draw[line width=.5pt, fill=black, overlay] (1.4,1.4) circle(.1em);
\node[circle, minimum size=.1em, outer sep=2pt, inner sep=0] (C) at (0,0) {};
\node[circle, outer sep=2pt, inner sep=0] (L) at (-.7,.7) {};
\node[circle, outer sep=2pt, inner sep=0] (R1) at (.7,.7) {};
\node[circle, outer sep=2pt, inner sep=0, overlay] (R2) at (1.4,1.4) {};
\draw[-stealth] (L) to (C);
\draw[-stealth] (C) to (R1);
\draw[-stealth, overlay] (R1) to (R2);
\draw[line cap=round, dash pattern=on 0pt off 1.2pt, line width=.6pt] (0,0) +(127:.8em) arc[start angle=127, end angle=53, radius=.8em] (.7,.7) +(214:.7em) to[bend left=60, looseness=1.22] +(51:.7em);
\end{scope}
\begin{scope}[xshift=-1.575em, yshift=-2.575em] 
\draw[line width=.5pt, fill=black] (0,0) circle(.1em);
\draw[line width=.5pt, color=magenta, fill=magenta] (.7,.7) circle(.1em);
\draw[line width=.5pt, fill=black] (.7,-.7) circle(.1em);
\draw[line width=.5pt, fill=black] (1.4,0) circle(.1em);
\node[circle, minimum size=.1em, outer sep=2pt, inner sep=0] (0) at (0,0) {};
\node[circle, minimum size=.1em, outer sep=2pt, inner sep=0] (1) at (.7,-.7) {};
\node[circle, minimum size=.1em, outer sep=2pt, inner sep=0] (2) at (1.4,0) {};
\node[circle, minimum size=.1em, outer sep=2pt, inner sep=0] (3) at (.7,.7) {};
\path[-stealth] (0) edge (1) (1) edge (2) (2) edge (3) (3) edge (0);
\draw[line cap=round, dash pattern=on 0pt off 1.2pt, line width=.6pt] (0,0) +(37:.8em) arc[start angle=37, end angle=-37, radius=.8em] (.7,-.7) +(127:.8em) arc[start angle=127, end angle=53, radius=.8em] (1.4,0) +(217:.8em) arc[start angle=217, end angle=143, radius=.8em] (.7,.7) +(-53:.8em) arc[start angle=-53, end angle=-127, radius=.8em];
\end{scope}
\begin{scope}[xshift=8em] 
\begin{scope}[xshift=-1.575em, yshift=-1.575em]
\draw[line width=.5pt, fill=black] (0,0) circle(.1em);
\draw[line width=.5pt, color=magenta, fill=magenta] (.7,.7) circle(.1em);
\draw[line width=.5pt, fill=black] (.7,-.7) circle(.1em);
\draw[line width=.5pt, fill=black] (1.4,0) circle(.1em);
\node[circle, minimum size=.1em, outer sep=2pt, inner sep=0] (0) at (0,0) {};
\node[circle, minimum size=.1em, outer sep=2pt, inner sep=0] (1) at (.7,-.7) {};
\node[circle, minimum size=.1em, outer sep=2pt, inner sep=0] (2) at (1.4,0) {};
\node[circle, minimum size=.1em, outer sep=2pt, inner sep=0] (3) at (.7,.7) {};
\path[-stealth] (0) edge (1) (1) edge (2) (2) edge (3) (3) edge (0);
\draw[line cap=round, dash pattern=on 0pt off 1.2pt, line width=.6pt] (0,0) +(37:.8em) arc[start angle=37, end angle=-37, radius=.8em] (.7,-.7) +(127:.8em) arc[start angle=127, end angle=53, radius=.8em] (1.4,0) +(217:.8em) arc[start angle=217, end angle=143, radius=.8em] (.7,.7) +(-53:.8em) arc[start angle=-53, end angle=-127, radius=.8em];
\end{scope}
\begin{scope}
\draw[line width=.5pt, fill=black] (-.7,.7) circle(.1em);
\draw[line width=.5pt, fill=black] (.7,.7) circle(.1em);
\draw[line width=.5pt, fill=black] (1.4,1.4) circle(.1em);
\node[circle, minimum size=.1em, outer sep=2pt, inner sep=0] (C) at (0,0) {};
\node[circle, outer sep=2pt, inner sep=0] (L) at (-.7,.7) {};
\node[circle, outer sep=2pt, inner sep=0] (R1) at (.7,.7) {};
\node[circle, outer sep=2pt, inner sep=0] (R2) at (1.4,1.4) {};
\draw[-stealth] (L) to (C);
\draw[-stealth] (C) to (R1);
\draw[-stealth] (R1) to (R2);
\draw[line cap=round, dash pattern=on 0pt off 1.2pt, line width=.6pt] (0,0) +(127:.8em) arc[start angle=127, end angle=53, radius=.8em] (.7,.7) +(214:.7em) to[bend left=60, looseness=1.22] +(51:.7em);
\end{scope}
\end{scope}
\end{tikzpicture}
\end{equation}
This diagram is a pushout diagram in the model category of $\Bbbk$-linear DG and A$_\infty$ categories localized at Morita equivalences. In other words, the gentle algebra $B$ associated to the quiver on the right represents the homotopy colimit of the diagram
\begin{equation*}
\begin{tikzpicture}[baseline=-2.6pt,description/.style={fill=white,inner sep=1.75pt}]
\matrix (m) [matrix of math nodes, row sep={.8em,between origins}, text height=1.5ex, column sep={3em,between origins}, text depth=0.25ex, ampersand replacement=\&, overlay]
{
\& A \\
\Bbbk \& \\
\& A\mathrlap{'} \\
};
\path[->,line width=.4pt]
(m-2-1.25) edge (m-1-2.180)
(m-2-1.-25) edge (m-3-2.180)
;
\end{tikzpicture}
\end{equation*}
which is the left part of \eqref{eq:colimitgentle}. This can be shown by cofibrantly replacing the algebras $A, A'$ by their Bardzell resolutions (see the proof of \cite[Theorem 6.10]{barmeierschrollwang}). This gluing can also be considered as a colimit in the $\infty$-category of $\Bbbk$-linear DG categories. 
\end{example}


\paragraph{Gluing of ribbon graphs}

The quiver with relations of a (graded) gentle algebra $A$ can be encoded by a (graded) ribbon graph \cite{schroll,haidenkatzarkovkontsevich,lekilipolishchuk2,opperplamondonschroll}. In terms of the presentation by building blocks \eqref{eq:buildingblocks}, these ribbon graphs can likewise be glued out of the following building blocks:
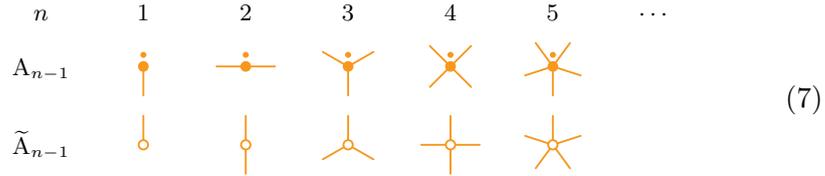
\begin{equation}
\label{eq:ribbonbuildingblocks}
\begin{tikzpicture}[x=2em,y=2em,decoration={markings,mark=at position 0.99 with {\arrow[black]{Stealth[length=4.8pt]}}}, baseline=-1.3em]
\begin{scope}[yshift=1.9em]
\node[font=\footnotesize] at (-1.75,0) {$n$\strut};
\end{scope}
\node[font=\footnotesize] at (-1.75,0) {$\mathrm A_{n-1}$};
\begin{scope} 
\begin{scope}[yshift=1.9em]
\node[font=\footnotesize] at (0,0) {$1$\strut};
\end{scope}
\draw[line width=.5pt, fill=ribboncolour, color=ribboncolour] (0,.05) circle(.16em);
\draw[line width=.5pt, fill=ribboncolour, color=ribboncolour] (0,.25) circle(.08em);
\draw[line cap=round, line width=.7pt, color=ribboncolour, overlay] (0,.05) -- +(-90:.5);
\end{scope}
\begin{scope}[xshift=3.5em] 
\begin{scope}[yshift=1.9em]
\node[font=\footnotesize] at (0,0) {$2$\strut};
\end{scope}
\draw[line width=.5pt, fill=ribboncolour, color=ribboncolour] (0,.05) circle(.16em);
\draw[line width=.5pt, fill=ribboncolour, color=ribboncolour] (0,.25) circle(.08em);
\draw[line cap=round, line width=.7pt, color=ribboncolour, overlay] (0,.05) -- +(0:.5);
\draw[line cap=round, line width=.7pt, color=ribboncolour, overlay] (0,.05) -- +(180:.5);
\end{scope}
\begin{scope}[xshift=7em] 
\begin{scope}[yshift=1.9em]
\node[font=\footnotesize] at (0,0) {$3$\strut};
\end{scope}
\draw[line width=.5pt, fill=ribboncolour, color=ribboncolour] (0,.05) circle(.16em);
\draw[line width=.5pt, fill=ribboncolour, color=ribboncolour] (0,.25) circle(.08em);
\draw[line cap=round, line width=.7pt, color=ribboncolour, overlay] (0,.05) -- +(150:.5);
\draw[line cap=round, line width=.7pt, color=ribboncolour, overlay] (0,.05) -- +(-90:.5);
\draw[line cap=round, line width=.7pt, color=ribboncolour, overlay] (0,.05) -- +(30:.5);
\end{scope}
\begin{scope}[xshift=10.5em] 
\begin{scope}[yshift=1.9em]
\node[font=\footnotesize] at (0,0) {$4$\strut};
\end{scope}
\draw[line width=.5pt, fill=ribboncolour, color=ribboncolour] (0,.05) circle(.16em);
\draw[line width=.5pt, fill=ribboncolour, color=ribboncolour] (0,.25) circle(.08em);
\draw[line cap=round, line width=.7pt, color=ribboncolour, overlay] (0,.05) -- +(135:.5);
\draw[line cap=round, line width=.7pt, color=ribboncolour, overlay] (0,.05) -- +(45:.5);
\draw[line cap=round, line width=.7pt, color=ribboncolour, overlay] (0,.05) -- +(-45:.5);
\draw[line cap=round, line width=.7pt, color=ribboncolour, overlay] (0,.05) -- +(-135:.5);
\end{scope}
\begin{scope}[xshift=14em] 
\begin{scope}[yshift=1.9em]
\node[font=\footnotesize] at (0,0) {$5$\strut};
\end{scope}
\draw[line width=.5pt, fill=ribboncolour, color=ribboncolour] (0,.05) circle(.16em);
\draw[line width=.5pt, fill=ribboncolour, color=ribboncolour] (0,.25) circle(.08em);
\draw[line cap=round, line width=.7pt, color=ribboncolour, overlay] (0,.05) -- +(72-18:.5);
\draw[line cap=round, line width=.7pt, color=ribboncolour, overlay] (0,.05) -- +(144-18:.5);
\draw[line cap=round, line width=.7pt, color=ribboncolour, overlay] (0,.05) -- +(216-18:.5);
\draw[line cap=round, line width=.7pt, color=ribboncolour, overlay] (0,.05) -- +(288-18:.5);
\draw[line cap=round, line width=.7pt, color=ribboncolour, overlay] (0,.05) -- +(-18:.5);
\end{scope}
\begin{scope}[xshift=17.5em] 
\begin{scope}[yshift=1.9em]
\node[font=\footnotesize] at (0,0) {$\dotsb$\strut};
\end{scope}
\end{scope}
\begin{scope}[yshift=-2.6em]
\node[font=\footnotesize] at (-1.75,0) {$\widetilde{\mathrm A}_{n-1}$};
\begin{scope}
\draw[line cap=round, line width=.7pt, color=ribboncolour] (0,0) -- +(90:.5);
\draw[line width=.7pt, color=ribboncolour, fill=white] (0,0) circle(.16em);
\end{scope}
\begin{scope}[xshift=3.5em] 
\draw[line cap=round, line width=.7pt, color=ribboncolour, overlay] (0,0) -- +(-90:.5);
\draw[line cap=round, line width=.7pt, color=ribboncolour, overlay] (0,0) -- +(90:.5);
\draw[line width=.7pt, color=ribboncolour, fill=white] (0,0) circle(.16em);
\end{scope}
\begin{scope}[xshift=7em] 
\draw[line cap=round, line width=.7pt, color=ribboncolour, overlay] (0,0) -- +(90:.5);
\draw[line cap=round, line width=.7pt, color=ribboncolour, overlay] (0,0) -- +(-30:.5);
\draw[line cap=round, line width=.7pt, color=ribboncolour, overlay] (0,0) -- +(210:.5);
\draw[line width=.7pt, color=ribboncolour, fill=white] (0,0) circle(.16em);
\end{scope}
\begin{scope}[xshift=10.5em] 
\draw[line cap=round, line width=.7pt, color=ribboncolour, overlay] (0,0) -- +(90:.5);
\draw[line cap=round, line width=.7pt, color=ribboncolour, overlay] (0,0) -- +(0:.5);
\draw[line cap=round, line width=.7pt, color=ribboncolour, overlay] (0,0) -- +(-90:.5);
\draw[line cap=round, line width=.7pt, color=ribboncolour, overlay] (0,0) -- +(180:.5);
\draw[line width=.7pt, color=ribboncolour, fill=white] (0,0) circle(.16em);
\end{scope}
\begin{scope}[xshift=14em] 
\draw[line cap=round, line width=.7pt, color=ribboncolour, overlay] (0,0) -- +(90:.5);
\draw[line cap=round, line width=.7pt, color=ribboncolour, overlay] (0,0) -- +(90-72:.5);
\draw[line cap=round, line width=.7pt, color=ribboncolour, overlay] (0,0) -- +(90+72:.5);
\draw[line cap=round, line width=.7pt, color=ribboncolour, overlay] (0,0) -- +(90-144:.5);
\draw[line cap=round, line width=.7pt, color=ribboncolour, overlay] (0,0) -- +(90+144:.5);
\draw[line width=.7pt, color=ribboncolour, fill=white] (0,0) circle(.16em);
\end{scope}
\end{scope}
\end{tikzpicture}
\end{equation}
In the upper row of ribbon graphs, the marking at one of the angles indicates where the natural cyclic ordering of the half-edges at the vertex is ``broken'' and restricted to a linear ordering. In the lower row we represent the vertex by an unfilled vertex, indicating that the corresponding polygon is ``punctured'', given by an extra boundary component with an ``entire boundary stop'' (cf.\ \eqref{eq:surfacebuildingblocks} and see Definition \ref{definition:orbifoldsurface}). Again, each building block may be assigned an arbitrary grading (see \cite[Section 1.3]{opperplamondonschroll} for more details).

The gluing of these building blocks is along edges\hspace{.2em}
\begin{tikzpicture}[x=1em, y=1em, baseline=-.25em]
\draw[line width=.7pt, line cap=round, color=ribboncolour] (0,0) -- (1,0);
\end{tikzpicture}
, a colimit construction in the category of topological spaces. Note that for any such gluing, the building blocks naturally form an open cover of the resulting ribbon graph.

\paragraph{Gluing properties of surfaces and their partially wrapped Fukaya categories}

The gluing properties of gentle algebras or ribbon graphs have a symplectic counterpart: the (graded) surface $\mathbf S$ with stops can be glued from building blocks
\begin{equation}
\label{eq:surfacebuildingblocks}
\begin{tikzpicture}[x=2em,y=2em,decoration={markings,mark=at position 0.99 with {\arrow[black]{Stealth[length=4.8pt]}}}, baseline=-1.3em]
\begin{scope}[yshift=2em]
\node[font=\footnotesize] at (-1.75,0) {$n$\strut};
\end{scope}
\node[font=\footnotesize] at (-1.75,0) {$\mathrm A_n$};
\begin{scope} 
\begin{scope}[yshift=2em]
\node[font=\footnotesize] at (0,0) {$1$\strut};
\end{scope}
\draw[line width=0pt, color=black!7, fill=black!7] (-.5,.5) to[out=-90, in=-90, looseness=1.95] (.5, .5) to cycle;
\draw[line width=.5pt, line cap=round] (-.5,.5) to (.5,.5);
\draw[color=stopcolour, fill=stopcolour] (0,.5) circle(.15em);
\draw[line cap=round, line width=.7pt, color=arccolour] (-.4,.5) to[out=-90, in=-90, looseness=2] (.4,.5);
\end{scope}
\begin{scope}[xshift=3.5em] 
\begin{scope}[yshift=2em]
\node[font=\footnotesize] at (0,0) {$2$\strut};
\end{scope}
\draw[line width=0pt, color=black!7, fill=black!7] (-.5,.5) to (-.4,-.5) to (.4,-.5) to (.5,.5) to cycle;
\draw[line width=.5pt, line cap=round] (-.5,.5) to (.5,.5);
\draw[line width=.5pt, line cap=round] (-.4,-.5) to (.4,-.5);
\draw[color=stopcolour, fill=stopcolour] (0,.5) circle(.15em);
\draw[line cap=round, line width=.7pt, color=arccolour] (-.4,.5) to (-.3,-.5) (.4,.5) to (.3,-.5);
\end{scope}
\begin{scope}[xshift=7em] 
\begin{scope}[yshift=2em]
\node[font=\footnotesize] at (0,0) {$3$\strut};
\end{scope}
\draw[line width=0pt, color=black!7, fill=black!7] ($(90:.5)+(90+90:.5)$) to[out=90-180, in=210-180] ($(210:.5)+(210-90:.4)$) to ($(210:.5)+(210+90:.4)$) to[out=210-180, in=330-180] ($(330:.5)+(330-90:.4)$) to ($(330:.5)+(330+90:.4)$) to[out=330-180, in=90-180] ($(90:.5)+(90-90:.5)$) to cycle;
\draw[line width=.5pt, line cap=round] (-.5,.5) to (.5,.5);
\draw[line width=.7pt, line cap=round, color=arccolour] ($(90:.5)+(90+90:.4)$) to[out=90-180, in=210-180] ($(210:.5)+(210-90:.3)$);
\draw[line width=.5pt, line cap=round] ($(210:.5)+(210-90:.4)$) to ($(210:.5)+(210+90:.4)$);
\draw[line width=.7pt, line cap=round, color=arccolour] ($(210:.5)+(210+90:.3)$) to[out=210-180, in=330-180] ($(330:.5)+(330-90:.3)$);
\draw[line width=.5pt, line cap=round] ($(330:.5)+(330-90:.4)$) to ($(330:.5)+(330+90:.4)$);
\draw[line width=.7pt, line cap=round, color=arccolour] ($(330:.5)+(330+90:.3)$) to[out=330-180, in=90-180] ($(90:.5)+(90-90:.4)$);
\draw[color=stopcolour, fill=stopcolour] (0,.5) circle(.15em);
\end{scope}
\begin{scope}[xshift=10.5em] 
\begin{scope}[yshift=2em]
\node[font=\footnotesize] at (0,0) {$4$\strut};
\end{scope}
\draw[line width=0pt, color=black!7, fill=black!7] ($(90:.5)+(90+90:.4)$) to[out=90-180, in=180-180] ($(180:.55)+(180-90:.3)$) to ($(180:.55)+(180+90:.3)$) to[out=180-180, in=270-180] ($(270:.55)+(270-90:.3)$) to ($(270:.55)+(270+90:.3)$) to[out=270-180, in=0-180] ($(0:.55)+(0-90:.3)$) to ($(0:.55)+(0+90:.3)$) to[out=0-180, in=90-180] ($(90:.5)+(90-90:.4)$) to cycle;
\draw[line width=.5pt, line cap=round] (-.4,.5) to (.4,.5);
\draw[line width=.7pt, line cap=round, color=arccolour] ($(90:.5)+(90+90:.3)$) to[out=90-180, in=180-180] ($(180:.55)+(180-90:.2)$);
\draw[line width=.5pt, line cap=round] ($(180:.55)+(180-90:.3)$) to ($(180:.55)+(180+90:.3)$);
\draw[line width=.7pt, line cap=round, color=arccolour] ($(180:.55)+(180+90:.2)$) to[out=180-180, in=270-180] ($(270:.55)+(270-90:.2)$);
\draw[line width=.5pt, line cap=round] ($(270:.55)+(270-90:.3)$) to ($(270:.55)+(270+90:.3)$);
\draw[line width=.7pt, line cap=round, color=arccolour] ($(270:.55)+(270+90:.2)$) to[out=270-180, in=0-180] ($(0:.55)+(0-90:.2)$);
\draw[line width=.5pt, line cap=round] ($(0:.55)+(0-90:.3)$) to ($(0:.55)+(0+90:.3)$);
\draw[line width=.7pt, line cap=round, color=arccolour] ($(0:.55)+(0+90:.2)$) to[out=0-180, in=90-180] ($(90:.5)+(90-90:.3)$);
\draw[fill=stopcolour, color=stopcolour] (0,.5) circle(.15em);
\end{scope}
\begin{scope}[xshift=14em] 
\begin{scope}[yshift=2em]
\node[font=\footnotesize] at (0,0) {$5$\strut};
\end{scope}
\draw[line width=0pt, color=black!7, fill=black!7] ($(90:.5)+(90+90:.35)$) to[out=90-180, in=162-180] ($(162:.55)+(162-90:.25)$) to ($(162:.55)+(162+90:.25)$) to[out=162-180, in=234-180] ($(234:.55)+(234-90:.25)$) to ($(234:.55)+(234+90:.25)$) to[out=234-180, in=306-180] ($(306:.55)+(306-90:.25)$) to ($(306:.55)+(306+90:.25)$) to[out=306-180, in=18-180] ($(18:.55)+(18-90:.25)$) to ($(18:.55)+(18+90:.25)$) to[out=18-180, in=90-180] ($(90:.5)+(90-90:.35)$) to cycle;
\draw[line width=.5pt, line cap=round] (-.35,.5) to (.35,.5);
\draw[line width=.5pt, line cap=round] ($(162:.55)+(162-90:.25)$) to ($(162:.55)+(162+90:.25)$);
\draw[line width=.5pt, line cap=round] ($(234:.55)+(234-90:.25)$) to ($(234:.55)+(234+90:.25)$);
\draw[line width=.5pt, line cap=round] ($(306:.55)+(306-90:.25)$) to ($(306:.55)+(306+90:.25)$);
\draw[line width=.5pt, line cap=round] ($(18:.55)+(18-90:.25)$) to ($(18:.55)+(18+90:.25)$);
\draw[line width=.7pt, line cap=round, color=arccolour] ($(90:.5)+(90+90:.25)$) to[out=90-180, in=162-180] ($(162:.55)+(162-90:.15)$);
\draw[line width=.7pt, line cap=round, color=arccolour] ($(162:.55)+(162+90:.15)$) to[out=162-180, in=234-180] ($(234:.55)+(234-90:.15)$);
\draw[line width=.7pt, line cap=round, color=arccolour] ($(234:.55)+(234+90:.15)$) to[out=234-180, in=306-180] ($(306:.55)+(306-90:.15)$);
\draw[line width=.7pt, line cap=round, color=arccolour] ($(306:.55)+(306+90:.15)$) to[out=306-180, in=18-180] ($(18:.55)+(18-90:.15)$);
\draw[line width=.7pt, line cap=round, color=arccolour] ($(18:.55)+(18+90:.15)$) to[out=18-180, in=90-180] ($(90:.5)+(90-90:.25)$);

\draw[fill=stopcolour, color=stopcolour] (0,.5) circle(.15em);
\end{scope}
\begin{scope}[xshift=17.5em] 
\begin{scope}[yshift=2em]
\node[font=\footnotesize] at (0,0) {$\dotsb$\strut};
\end{scope}
\end{scope}
\begin{scope}[yshift=-3em]
\node[font=\footnotesize] at (-1.75,0) {$\widetilde{\mathrm A}_{n-1}$};
\begin{scope} 
\draw[line width=0pt, color=black!7, fill=black!7] (-.5,-.5) to (-.5,-.1) to[out=90, in=90, looseness=1.95] (.5,-.1) to (.5,-.5) to cycle;
\draw[line width=.5pt, line cap=round] (-.5,-.5) to (.5,-.5);
\draw[line width=.7pt, line cap=round, color=arccolour] (-.4,-.5) to (-.4,-.1) to[out=90, in=90, looseness=2] (.4,-.1) to (.4,-.5);
\draw[line width=.15em, color=stopcolour, fill=white] (0,0) circle(.3em);
\end{scope}
\begin{scope}[xshift=3.5em] 
\draw[line width=0pt, color=black!7, fill=black!7] ($(180:.5)+(180-90:.4)$) to ($(180:.5)+(180+90:.4)$) to ($(0:.5)+(0-90:.4)$) to ($(0:.5)+(0+90:.4)$) to ($(180:.5)+(180-90:.4)$) to cycle;
\draw[line width=.7pt, line cap=round, color=arccolour] ($(180:.5)+(180+90:.3)$) to ($(0:.5)+(0-90:.3)$);
\draw[line width=.5pt, line cap=round] ($(180:.5)+(180-90:.4)$) to ($(180:.5)+(180+90:.4)$);
\draw[line width=.7pt, line cap=round, color=arccolour] ($(0:.5)+(0+90:.3)$) to ($(180:.5)+(180-90:.3)$);
\draw[line width=.5pt, line cap=round] ($(0:.5)+(0-90:.4)$) to ($(0:.5)+(0+90:.4)$);
\draw[line width=.15em, color=stopcolour, fill=white] (0,0) circle(.3em);
\end{scope}
\begin{scope}[xshift=7em] 
\draw[line width=0pt, color=black!7, fill=black!7] ($(150:.5)+(150+90:.4)$) to[out=150-180, in=270-180] ($(270:.5)+(270-90:.4)$) to ($(270:.5)+(270+90:.4)$) to[out=270-180, in=30-180] ($(30:.5)+(30-90:.4)$) to ($(30:.5)+(30+90:.4)$) to[out=30-180, in=150-180] ($(150:.5)+(150-90:.4)$) to cycle;
\draw[line width=.7pt, line cap=round, color=arccolour] ($(150:.5)+(150+90:.3)$) to[out=150-180, in=270-180] ($(270:.5)+(270-90:.3)$);
\draw[line width=.5pt, line cap=round] ($(150:.5)+(150-90:.4)$) to ($(150:.5)+(150+90:.4)$);
\draw[line width=.7pt, line cap=round, color=arccolour] ($(270:.5)+(270+90:.3)$) to[out=270-180, in=30-180] ($(30:.5)+(30-90:.3)$);
\draw[line width=.5pt, line cap=round] ($(270:.5)+(270-90:.4)$) to ($(270:.5)+(270+90:.4)$);
\draw[line width=.7pt, line cap=round, color=arccolour] ($(30:.5)+(30+90:.3)$) to[out=30-180, in=150-180] ($(150:.5)+(150-90:.3)$);
\draw[line width=.5pt, line cap=round] ($(30:.5)+(30-90:.4)$) to ($(30:.5)+(30+90:.4)$);
\draw[line width=.15em, color=stopcolour, fill=white] (0,0) circle(.3em);
\end{scope}
\begin{scope}[xshift=10.5em] 
\draw[line width=0pt, color=black!7, fill=black!7] ($(135:.55)+(135-90:.3)$) to ($(135:.55)+(135+90:.3)$) to[out=135-180, in=225-180] ($(225:.55)+(225-90:.3)$) to ($(225:.55)+(225+90:.3)$) to[out=225-180, in=315-180] ($(315:.55)+(315-90:.3)$) to ($(315:.55)+(315+90:.3)$) to[out=315-180, in=45-180] ($(45:.55)+(45-90:.3)$) to ($(45:.55)+(45+90:.3)$) to[out=45-180, in=135-180] cycle;
\draw[line width=.7pt, line cap=round, color=arccolour] ($(135:.55)+(135+90:.2)$) to[out=135-180, in=225-180] ($(225:.55)+(225-90:.2)$);
\draw[line width=.5pt, line cap=round] ($(135:.55)+(135-90:.3)$) to ($(135:.55)+(135+90:.3)$);
\draw[line width=.7pt, line cap=round, color=arccolour] ($(225:.55)+(225+90:.2)$) to[out=225-180, in=315-180] ($(315:.55)+(315-90:.2)$);
\draw[line width=.5pt, line cap=round] ($(225:.55)+(225-90:.3)$) to ($(225:.55)+(225+90:.3)$);
\draw[line width=.7pt, line cap=round, color=arccolour] ($(315:.55)+(315+90:.2)$) to[out=315-180, in=45-180] ($(45:.55)+(45-90:.2)$);
\draw[line width=.5pt, line cap=round] ($(315:.55)+(315-90:.3)$) to ($(315:.55)+(315+90:.3)$);
\draw[line width=.7pt, line cap=round, color=arccolour] ($(45:.55)+(45+90:.2)$) to[out=45-180, in=135-180] ($(135:.55)+(135-90:.2)$);
\draw[line width=.5pt, line cap=round] ($(45:.55)+(45-90:.3)$) to ($(45:.55)+(45+90:.3)$);
\draw[line width=.15em, color=stopcolour, fill=white] (0,0) circle(.3em);
\end{scope}
\begin{scope}[xshift=14em] 
\draw[line width=0pt, color=black!7, fill=black!7] ($(54:.55)+(54-90:.25)$) to ($(54:.55)+(54+90:.25)$) to[out=54-180, in=126-180] ($(126:.55)+(126-90:.25)$) to ($(126:.55)+(126+90:.25)$) to[out=126-180, in=198-180] ($(198:.55)+(198-90:.25)$) to ($(198:.55)+(198+90:.25)$) to[out=198-180, in=270-180] ($(270:.55)+(270-90:.25)$) to ($(270:.55)+(270+90:.25)$) to[out=270-180, in=342-180] ($(342:.55)+(342-90:.25)$) to ($(342:.55)+(342+90:.25)$) to[out=342-180, in=54-180] cycle;
\draw[line width=.5pt, line cap=round] ($(54:.55)+(54-90:.25)$) to ($(54:.55)+(54+90:.25)$);
\draw[line width=.7pt, line cap=round, color=arccolour] ($(54:.55)+(54+90:.15)$) to[out=54-180, in=126-180] ($(126:.55)+(126-90:.15)$);
\draw[line width=.5pt, line cap=round] ($(126:.55)+(126-90:.25)$) to ($(126:.55)+(126+90:.25)$);
\draw[line width=.7pt, line cap=round, color=arccolour] ($(126:.55)+(126+90:.15)$) to[out=126-180, in=198-180] ($(198:.55)+(198-90:.15)$);
\draw[line width=.5pt, line cap=round] ($(198:.55)+(198-90:.25)$) to ($(198:.55)+(198+90:.25)$);
\draw[line width=.7pt, line cap=round, color=arccolour] ($(198:.55)+(198+90:.15)$) to[out=198-180, in=270-180] ($(270:.55)+(270-90:.15)$);
\draw[line width=.5pt, line cap=round] ($(270:.55)+(270-90:.25)$) to ($(270:.55)+(270+90:.25)$);
\draw[line width=.7pt, line cap=round, color=arccolour] ($(270:.55)+(270+90:.15)$) to[out=270-180, in=342-180] ($(342:.55)+(342-90:.15)$);
\draw[line width=.5pt, line cap=round] ($(342:.55)+(342-90:.25)$) to ($(342:.55)+(342+90:.25)$);
\draw[line width=.7pt, line cap=round, color=arccolour] ($(342:.55)+(342+90:.15)$) to[out=342-180, in=54-180] ($(54:.55)+(54-90:.15)$);
\draw[line width=.15em, color=stopcolour, fill=white] (0,0) circle(.3em);
\end{scope}
\end{scope}
\end{tikzpicture}
\end{equation}
Here the black line represents a part of the boundary, the green line an {\it arc} inside the surface, and the blue markings the {\it stops}. Note that each building block contains exactly one stop (one connected closed subset of the boundary). Each building block may be graded by equipping it with a line field, which up to homotopy may be assumed to be parallel to the arcs. We also call the type $\mathrm A_n$ an {\it $n$-gon} and type $\widetilde{\mathrm A}_{n-1}$ a {\it punctured $n$-gon}. See Section \ref{section:orbifoldsurfaces} for more details.

\begin{remark}
The building blocks in the upper row of \eqref{eq:surfacebuildingblocks} are {\it type $\mathrm A_{n-1}$ Liouville/\hspace{0pt}Weinstein sectors} in the sense of \cite{ganatrapardonshende1,ganatrapardonshende2}. As the building blocks in the lower row play a similar role in the local-to-global properties of gentle algebras and finitely generated partially wrapped Fukaya categories, we sometimes refer to these as ``type $\widetilde{\mathrm A}_{n-1}$ sectors'', even if the underlying exact symplectic surface is not a Liouville sector.
\end{remark}

The (graded) building blocks \eqref{eq:surfacebuildingblocks} may be glued along the arcs, i.e.\ along pieces $\mathbf S_e = \hspace{.2em}
\begin{tikzpicture}[x=.5em, y=1em, baseline=-.25em]
\draw[line width=0, fill=black!7, color=black!7] (-1,.2) rectangle (1,-.2);
\draw[line width=.5pt, line cap=round] (-1,.2) -- (-1,-.2);
\draw[line width=.5pt, line cap=round] (1,.2) -- (1,-.2);
\draw[line width=.7pt, line cap=round, color=arccolour] (-1,0) -- (1,0);
\end{tikzpicture}
$
\hspace{.2em}
to obtain a (graded) surface. Indeed, any graded smooth surface $\mathbf S$ (cf.\ Definition \ref{definition:orbifoldsurface}) may be obtained as such a gluing, i.e.\ we may write $\mathbf S = \bigcup_{v \in \mathrm V (\mathrm G)} \mathbf S_v$, where each $\mathbf S_v$ is one of the building blocks \eqref{eq:surfacebuildingblocks} and the indexing set may be taken as the set $\mathrm V (\mathrm G)$ of vertices of the corresponding graded ribbon graph $\mathrm G$. Note that $\mathbf S_v \cap \mathbf S_w$ is either $\varnothing$, $\mathbf S_e$ or $\mathbf S_e \sqcup \mathbf S_{e'}$ depending on whether $v$ and $w$ share zero, one or two edges.

The green arcs in \eqref{eq:surfacebuildingblocks} represent objects in the partially wrapped Fukaya category and the flows along the boundary (avoiding the stop) span the graded vector space of morphisms, yielding precisely the algebraic building blocks \eqref{eq:buildingblocks}. The marvel of (partially) wrapped Fukaya categories is that the categories themselves can also be glued exactly in the same way, i.e.\ we have an equivalence
\[
\W (\mathbf S) \simeq \hocolim \bigl( \textstyle\prod_{e \in \mathrm E (\mathrm G)} \W (\mathbf S_e) \rightrightarrows \prod_{v \in \mathrm V (\mathrm G)} \W (\mathbf S_v) \bigr).
\]
This gluing property is called ``sectorial descent'' in \cite{ganatrapardonshende2} and is a cornerstone of the structure theory of partially wrapped Fukaya categories.

This local-to-global viewpoint can be formalized in the language of co\-sheaves of pretriangulated A$_\infty$ categories \cite{kontsevich2,dyckerhoffkapranov,haidenkatzarkovkontsevich}, where for Fukaya categories of surfaces one may take the ribbon graph as the underlying topological space on which this cosheaf of categories is defined. The local-to-global properties of $\W (\mathbf S)$ shown for smooth surfaces in \cite{dyckerhoffkapranov,haidenkatzarkovkontsevich} establish a partially wrapped version of a conjecture of Kontsevich on the local-to-global properties of wrapped Fukaya categories \cite{kontsevich2}. 
Algebraically, the passage from the surface to its partially wrapped Fukaya category can be realized by passing from the corresponding graded gentle algebra $A$ to its category of twisted complexes $\tw (A)^\natural$, giving the equivalence $\W (\mathbf S) \simeq \tw (A)^\natural$ \eqref{align:equivalencegentle}. 

\begin{remark}[Homological smoothness and higher structures]
\label{remark:homologicalsmoothness}
In \cite{haidenkatzarkovkontsevich}, the case of {\it homologically smooth} Fukaya categories is considered. In this case, the corresponding gentle algebras are homologically smooth and thus are built out of building blocks of type $\mathrm A_n$ only, as the building blocks of type $\widetilde{\mathrm A}_{n-1}$ are homologically non-smooth. In order to study ``flips'' of arcs it is useful to have extra flexibility by adding extra arcs and removing other arcs. In terms of building blocks, this may be achieved by admitting also building blocks without any stops:
\begin{equation}
\label{eq:surfacebuildingblocksnostop}
\begin{tikzpicture}[x=2em,y=2em,decoration={markings,mark=at position 0.99 with {\arrow[black]{Stealth[length=4.8pt]}}}, baseline=-.25em]
\begin{scope} 
\draw[line width=.5pt, fill=black] (90:.6) circle(.08em);
\draw[line width=.5pt, fill=black] (162:.6) circle(.08em);
\draw[line width=.5pt, fill=black] (234:.6) circle(.1em);
\draw[line width=.5pt, fill=black] (306:.6) circle(.1em);
\draw[line width=.5pt, fill=black] (18:.6) circle(.1em);
\node[circle, minimum size=.1em, outer sep=2pt, inner sep=0] (0) at (162:.6) {};
\node[circle, minimum size=.1em, outer sep=2pt, inner sep=0] (1) at (234:.6) {};
\node[circle, minimum size=.1em, outer sep=2pt, inner sep=0] (2) at (306:.6) {};
\node[circle, minimum size=.1em, outer sep=2pt, inner sep=0] (3) at (18:.6) {};
\node[circle, minimum size=.1em, outer sep=2pt, inner sep=0] (4) at (90:.6) {};
\path[-stealth] (0) edge (1) (1) edge (2) (2) edge (3) (3) edge (4) (4) edge (0);
\draw[line cap=round, dash pattern=on 0pt off 1.2pt, line width=.6pt] (90:.6) +(-135:.55em) to[bend right=32, looseness=1.2] +(-45:.65em) (90+72:.6) +(-135+72:.55em) to[bend right=32, looseness=1.2] +(-45+72:.65em) (90+144:.6) +(-135+144:.55em) to[bend right=32, looseness=1.2] +(-45+144:.65em) (90+216:.6) +(-135+216:.55em) to[bend right=32, looseness=1.2] +(-45+216:.65em) (90-72:.6) +(-135-72:.55em) to[bend right=32, looseness=1.2] +(-45-72:.65em);
\node[font=\footnotesize, left] at (180:.6) {$\Biggl($};
\node[font=\footnotesize, right] at (0:.6) {$, \mucirc_n \Biggr)$};
\node[font=\footnotesize] at (-3, 0) {$(\widetilde{\mathrm A}_{n-1}, \mucirc_n)$};
\end{scope}
\begin{scope}[xshift=7em] 
\draw[line cap=round, line width=.7pt, color=ribboncolour, overlay] (0,0) -- +(90:.5);
\draw[line cap=round, line width=.7pt, color=ribboncolour, overlay] (0,0) -- +(90-72:.5);
\draw[line cap=round, line width=.7pt, color=ribboncolour, overlay] (0,0) -- +(90+72:.5);
\draw[line cap=round, line width=.7pt, color=ribboncolour, overlay] (0,0) -- +(90-144:.5);
\draw[line cap=round, line width=.7pt, color=ribboncolour, overlay] (0,0) -- +(90+144:.5);
\draw[line width=.7pt, color=ribboncolour, fill=ribboncolour] (0,0) circle(.16em);
\end{scope}
\begin{scope}[xshift=12em] 
\draw[line width=0pt, color=black!7, fill=black!7] ($(54:.55)+(54-90:.25)$) to ($(54:.55)+(54+90:.25)$) to[out=54-180, in=126-180] ($(126:.55)+(126-90:.25)$) to ($(126:.55)+(126+90:.25)$) to[out=126-180, in=198-180] ($(198:.55)+(198-90:.25)$) to ($(198:.55)+(198+90:.25)$) to[out=198-180, in=270-180] ($(270:.55)+(270-90:.25)$) to ($(270:.55)+(270+90:.25)$) to[out=270-180, in=342-180] ($(342:.55)+(342-90:.25)$) to ($(342:.55)+(342+90:.25)$) to[out=342-180, in=54-180] cycle;
\draw[line width=.5pt, line cap=round] ($(54:.55)+(54-90:.25)$) to ($(54:.55)+(54+90:.25)$);
\draw[line width=.7pt, line cap=round, color=arccolour] ($(54:.55)+(54+90:.15)$) to[out=54-180, in=126-180] ($(126:.55)+(126-90:.15)$);
\draw[line width=.5pt, line cap=round] ($(126:.55)+(126-90:.25)$) to ($(126:.55)+(126+90:.25)$);
\draw[line width=.7pt, line cap=round, color=arccolour] ($(126:.55)+(126+90:.15)$) to[out=126-180, in=198-180] ($(198:.55)+(198-90:.15)$);
\draw[line width=.5pt, line cap=round] ($(198:.55)+(198-90:.25)$) to ($(198:.55)+(198+90:.25)$);
\draw[line width=.7pt, line cap=round, color=arccolour] ($(198:.55)+(198+90:.15)$) to[out=198-180, in=270-180] ($(270:.55)+(270-90:.15)$);
\draw[line width=.5pt, line cap=round] ($(270:.55)+(270-90:.25)$) to ($(270:.55)+(270+90:.25)$);
\draw[line width=.7pt, line cap=round, color=arccolour] ($(270:.55)+(270+90:.15)$) to[out=270-180, in=342-180] ($(342:.55)+(342-90:.15)$);
\draw[line width=.5pt, line cap=round] ($(342:.55)+(342-90:.25)$) to ($(342:.55)+(342+90:.25)$);
\draw[line width=.7pt, line cap=round, color=arccolour] ($(342:.55)+(342+90:.15)$) to[out=342-180, in=54-180] ($(54:.55)+(54-90:.15)$);
\node[font=\footnotesize] at (3,0) {$n \geq 2$};
\end{scope}
\end{tikzpicture}
\end{equation}
Although the underlying quiver is of type $\widetilde{\mathrm A}_{n-1}$ (with full quadratic monomial relations, as always), the grading is now required to satisfy $\sum_{p \in Q_1} \lvert p \rvert = n - 2$ and $\mucirc_n$ denotes a new nontrivial higher multiplication, equipping the path algebra with the structure of an A$_\infty$ algebra (see \eqref{align:smoothdiskproduct} for the formula). Geometrically, this higher structure can be seeing as counting a pseudo-holomorphic disk in the sense of partially wrapped Floer theory. Algebraically, it results in $(\widetilde{\mathrm A}_{n-1}, \mucirc_n)$ being homologically smooth and Morita equivalent to type $\mathrm A_{n-1}$. Moreover, $(\widetilde{\mathrm A}_{n-1}, \mucirc_n)$ can be viewed as an A$_\infty$ deformation of $\widetilde{\mathrm A}_{n-1}$ in the sense of Definition \ref{definition:ainfinitydeformation}.

This higher structure is also essential for fully wrapped Fukaya categories of surfaces with boundary (see e.g.\ \cite{bocklandt}), as such surfaces can only be glued out of \eqref{eq:surfacebuildingblocksnostop}.
\end{remark}

\subsection{From smooth surfaces to orbifold surfaces}

In \cite{barmeierschrollwang} Schroll and the authors of this article generalize the main constructions of partially wrapped Fukaya categories to graded surfaces with isolated orbifold singularities. We show the equivalence of different natural generalizations of the smooth case. One of the main results of \cite{barmeierschrollwang} can be summarized in the following equivalences
\begin{equation}
\label{eq:viewpoints}
\W (\mathbf S) := \mathbf \Gamma (\mathcal T_{\mathrm G (\Delta)}) \simeq \tw (\mathbf A_\Delta) \simeq (\W (\widetilde{\mathbf S}) / \mathbb Z_2)^\natural.
\end{equation}
Here $\W (\mathbf S)$ denotes the partially wrapped Fukaya category of a graded orbifold surface $\mathbf S$ (see Definition \ref{definition:orbifoldsurface} below) and
\begin{enumerate}
\item $\mathcal T_{\mathrm G (\Delta)}$ is a cosheaf of pretriangulated A$_\infty$ categories on a ribbon graph (or ribbon complex) dual to a dissection $\Delta$ of $\mathbf S$ and $\mathbf \Gamma (\mathcal T_{\mathrm G (\Delta)})$ denotes its category of global sections
\item $\mathbf A_\Delta$ is an explicit A$_\infty$ category associated to an admissible dissection $\Delta$ of $\mathbf S$, whose higher multiplications correspond to orbifold disk sequences
\item $\W (\widetilde{\mathbf S}) / \mathbb Z_2$ is the A$_\infty$ orbit category of $\W (\widetilde{\mathbf S})$ where $\widetilde{\mathbf S}$ is a smooth double cover of $\mathbf S$.
\end{enumerate}
We note that the second perspective via explicit A$_\infty$ structures is also studied in independent work by Cho and Kim \cite{chokim,kim} (from a different but ultimately equivalent perspective), and the third perspective on orbit categories is studied also independently by Cho and Kim \cite{chokim} (using the formalism of invariant categories) and by Amiot and Plamondon \cite{amiotplamondon2} (under the name of {\it skew-group categories}).

In \cite{barmeierschrollwang} we develop the flexible, but also more technical, formalism of (graded) ribbon {\it complexes} which are ribbon graphs to which certain orbifold $2$-cells are attached. In this survey, we focus on interpreting the results of \cite{barmeierschrollwang} purely in terms of ribbon graphs, analogous to the constructions outlined in Section \ref{subsection:gluingproperties}. 

Our main result of this survey (Theorem \ref{theorem:newdissection}) shows that the partially wrapped Fukaya category of an arbitary graded orbifold surface with stops can be glued out of Fukaya categories of the sectors of type $\mathrm A_n$ and $\widetilde{\mathrm A}_{n-1}$ \eqref{eq:surfacebuildingblocks}, together with a new family of ``sectors''
\begin{equation}
\label{eq:orbifoldbuildingblocks}
\begin{tikzpicture}[x=2em,y=2em,decoration={markings,mark=at position 0.99 with {\arrow[black]{Stealth[length=4.8pt]}}}, baseline=-.3em]
\begin{scope} 
\draw[line width=.5pt, fill=black] (90:.6) circle(.08em);
\draw[line width=.5pt, fill=black] (162:.6) circle(.08em);
\draw[line width=.5pt, fill=black] (234:.6) circle(.1em);
\draw[line width=.5pt, fill=black] (306:.6) circle(.1em);
\draw[line width=.5pt, fill=black] (18:.6) circle(.1em);
\node[circle, minimum size=.1em, outer sep=2pt, inner sep=0] (0) at (162:.6) {};
\node[circle, minimum size=.1em, outer sep=2pt, inner sep=0] (1) at (234:.6) {};
\node[circle, minimum size=.1em, outer sep=2pt, inner sep=0] (2) at (306:.6) {};
\node[circle, minimum size=.1em, outer sep=2pt, inner sep=0] (3) at (18:.6) {};
\node[circle, minimum size=.1em, outer sep=2pt, inner sep=0] (4) at (90:.6) {};
\path[-stealth] (0) edge (1) (1) edge (2) (2) edge (3) (3) edge (4) (4) edge (0);
\draw[line cap=round, dash pattern=on 0pt off 1.2pt, line width=.6pt] (90:.6) +(-135:.55em) to[bend right=32, looseness=1.2] +(-45:.65em) (90+72:.6) +(-135+72:.55em) to[bend right=32, looseness=1.2] +(-45+72:.65em) (90+144:.6) +(-135+144:.55em) to[bend right=32, looseness=1.2] +(-45+144:.65em) (90+216:.6) +(-135+216:.55em) to[bend right=32, looseness=1.2] +(-45+216:.65em) (90-72:.6) +(-135-72:.55em) to[bend right=32, looseness=1.2] +(-45-72:.65em);
\node[font=\footnotesize, left] at (180:.6) {$\Biggl($};
\node[font=\footnotesize, right] at (0:.6) {$, \muotimes_{2n} \Biggr)$};
\node[font=\footnotesize] at (-3, 0) {$(\widetilde{\mathrm A}_{n-1}, \muotimes_{2n})$};
\end{scope}
\begin{scope}[xshift=7em] 
\draw[line cap=round, line width=.7pt, color=ribboncolour, overlay] (0,0) -- +(90:.5);
\draw[line cap=round, line width=.7pt, color=ribboncolour, overlay] (0,0) -- +(90-72:.5);
\draw[line cap=round, line width=.7pt, color=ribboncolour, overlay] (0,0) -- +(90+72:.5);
\draw[line cap=round, line width=.7pt, color=ribboncolour, overlay] (0,0) -- +(90-144:.5);
\draw[line cap=round, line width=.7pt, color=ribboncolour, overlay] (0,0) -- +(90+144:.5);
\draw[line width=.3pt, color=ribboncolour, fill=white] (0,0) circle(.2em);
\draw[line width=.3pt, color=ribboncolour] (45:.2em) -- (-135:.2em) (-45:.2em) -- (135:.2em);
\end{scope}
\begin{scope}[xshift=12em] 
\draw[line width=0pt, color=black!7, fill=black!7] ($(54:.55)+(54-90:.25)$) to ($(54:.55)+(54+90:.25)$) to[out=54-180, in=126-180] ($(126:.55)+(126-90:.25)$) to ($(126:.55)+(126+90:.25)$) to[out=126-180, in=198-180] ($(198:.55)+(198-90:.25)$) to ($(198:.55)+(198+90:.25)$) to[out=198-180, in=270-180] ($(270:.55)+(270-90:.25)$) to ($(270:.55)+(270+90:.25)$) to[out=270-180, in=342-180] ($(342:.55)+(342-90:.25)$) to ($(342:.55)+(342+90:.25)$) to[out=342-180, in=54-180] cycle;
\draw[line width=.5pt, line cap=round] ($(54:.55)+(54-90:.25)$) to ($(54:.55)+(54+90:.25)$);
\draw[line width=.7pt, line cap=round, color=arccolour] ($(54:.55)+(54+90:.15)$) to[out=54-180, in=126-180] ($(126:.55)+(126-90:.15)$);
\draw[line width=.5pt, line cap=round] ($(126:.55)+(126-90:.25)$) to ($(126:.55)+(126+90:.25)$);
\draw[line width=.7pt, line cap=round, color=arccolour] ($(126:.55)+(126+90:.15)$) to[out=126-180, in=198-180] ($(198:.55)+(198-90:.15)$);
\draw[line width=.5pt, line cap=round] ($(198:.55)+(198-90:.25)$) to ($(198:.55)+(198+90:.25)$);
\draw[line width=.7pt, line cap=round, color=arccolour] ($(198:.55)+(198+90:.15)$) to[out=198-180, in=270-180] ($(270:.55)+(270-90:.15)$);
\draw[line width=.5pt, line cap=round] ($(270:.55)+(270-90:.25)$) to ($(270:.55)+(270+90:.25)$);
\draw[line width=.7pt, line cap=round, color=arccolour] ($(270:.55)+(270+90:.15)$) to[out=270-180, in=342-180] ($(342:.55)+(342-90:.15)$);
\draw[line width=.5pt, line cap=round] ($(342:.55)+(342-90:.25)$) to ($(342:.55)+(342+90:.25)$);
\draw[line width=.7pt, line cap=round, color=arccolour] ($(342:.55)+(342+90:.15)$) to[out=342-180, in=54-180] ($(54:.55)+(54-90:.15)$);
\node[font=\footnotesize] at (0,0) {$\times$};
\node[font=\footnotesize] at (3,0) {$n \geq 1$};
\end{scope}
\end{tikzpicture}
\end{equation}
where the grading is required to satisfy $\sum_{p \in Q_1} \lvert p \rvert = n - 1$ and $\muotimes_{2n}$ is a $2n$-ary operation defining an A$_\infty$ algebra structure on $\widetilde{\mathrm A}_{n-1}$ which for $n \geq 2$ is considered with full quadratic monomial relations. Note that the case $n = 1$ is special, as $\muotimes_2$ is a binary operation, thus changing the underlying composition of the loop $p$ at the vertex to $p^2 = 1$.

In the right picture, $\times$ marks an order $2$ orbifold on the surface. In fact, $\muotimes_{2n}$ is induced by $\mucirc_{2n}$ on $\widetilde{\mathrm A}_{2n-1}$ which can be viewed as a two-to-one cover of $\widetilde{\mathrm A}_{n-1}$. Since the A$_\infty$ algebra $(\widetilde{\mathrm A}_{n-1}, \muotimes_{2n})$ is Morita equivalent to the path algebra of a quiver of type $\mathrm D_{n+1}$, the orbifold disks in \eqref{eq:orbifoldbuildingblocks} can be viewed as ``type $\mathrm D_{n+1}$ sectors''. As in Remark \ref{remark:homologicalsmoothness}, $(\widetilde{\mathrm A}_{n-1}, \muotimes_{2n})$ can be viewed as an A$_\infty$ deformation of $\widetilde{\mathrm A}_{n-1}$. (However, these two types of A$_\infty$ deformations are qualitatively different!)

\paragraph{Gluing properties of skew-gentle algebras}

Graded skew-gentle algebras naturally arise a special case, namely by combining the building blocks for gentle algebras \eqref{eq:buildingblocks} with the simplest ($n = 1$) building block of \eqref{eq:orbifoldbuildingblocks}
\begin{equation}
\label{eq:skewgentlebuildingblock}
\begin{tikzpicture}[x=2em,y=2em,decoration={markings,mark=at position 0.99 with {\arrow[black]{Stealth[length=4.8pt]}}}, baseline=-.3em]
\begin{scope} 
\draw[line width=.5pt, fill=black] (0,.15) circle(.1em);
\node[circle, minimum size=.1em, outer sep=2pt, inner sep=0] (0) at (0,.15) {};
\path[-stealth] (0) edge[out=235, in=305, looseness=15] node[font=\scriptsize, pos=.6, overlay, right] {$\epsilon$} (0);
\node[font=\footnotesize, left] at (180:.15) {$\biggl($};
\node[font=\footnotesize, right] at (0:.4) {$, \muotimes_{2} \biggr)$};
\node[font=\footnotesize] at (-3, 0) {$(\widetilde{\mathrm A}_{0}, \muotimes_{2})$};
\end{scope}
\begin{scope}[xshift=7em] 
\draw[line cap=round, line width=.7pt, color=ribboncolour, overlay] (0,0) -- +(90:.5);
\draw[line width=.3pt, color=ribboncolour, fill=white] (0,0) circle(.2em);
\draw[line width=.3pt, color=ribboncolour] (45:.2em) -- (-135:.2em) (-45:.2em) -- (135:.2em);
\end{scope}
\begin{scope}[xshift=12em] 
\draw[line width=0pt, color=black!7, fill=black!7] (-.5,-.5) to (-.5,-.1) to[out=90, in=90, looseness=1.95] (.5,-.1) to (.5,-.5) to cycle;
\draw[line width=.5pt, line cap=round] (-.5,-.5) to (.5,-.5);
\draw[line width=.7pt, line cap=round, color=arccolour] (-.4,-.5) to (-.4,-.1) to[out=90, in=90, looseness=2] (.4,-.1) to (.4,-.5);
\node[font=\footnotesize] at (0,0) {$\times$};
\end{scope}
\end{tikzpicture}
\end{equation}
where $\muotimes_2$ is a binary product making $(\widetilde{\mathrm A}_0, \muotimes_2) \simeq \Bbbk [\epsilon] / (\epsilon^2 - 1) \simeq \Bbbk \times \Bbbk$ with $| \epsilon | = 0$.

In particular, we have that (graded) skew-gentle algebras arise as homotopy colimits of diagrams consisting of (graded) building blocks of types $\mathrm A_n$, $\widetilde{\mathrm A}_{n-1}$ and $(\widetilde{\mathrm A}_0, \muotimes_2)$.

\paragraph{From orbifold surfaces to derived skew-gentle algebras}

The equivalences \eqref{eq:viewpoints} give a symplectic interpretation of the orbifold surfaces that have been considered in the cluster theory of ``punctured surfaces'' \cite{schiffler,labardini,geisslabardinifragososchroeer,amiotplamondon,qiuzhou} and the representation theory of skew-gentle algebras \cite{qiuzhou,amiotbruestle,labardinifragososchrollvaldivieso,qiuzhangzhou}. (Note that the ``punctures'' that appear in the cluster-theoretic perspective can be viewed as order $2$ orbifold points. We use punctures with a different meaning, namely by removing an open disk from a surface, resulting in a new boundary component and in the absence of a higher multiplication.)

From the perspective of partially wrapped Fukaya categories and derived categories, a key difference to the smooth/gentle case is the observation that for an orbifold surface $\mathbf S$ there are generally many formal generators of $\W (\mathbf S)$ whose corresponding associative algebra $A$ is not skew-gentle. In other words, the class of skew-gentle algebras is not closed under derived equivalence. We conjecture \cite[Conjecture 8.11]{barmeierschrollwang} that every graded associative algebra which is derived equivalent to a graded skew-gentle algebra arises as the endomorphism algebra of a formal generator of $\W (\mathbf S)$ given by a {\it formal dissection} (see \cite[Definition 8.1]{barmeierschrollwang} and Section \ref{section:formal}) of the associated orbifold surface. This conjecture is known to hold for trivially graded gentle algebras by \cite{schroeerzimmermann} and is currently still open, even for the case of graded gentle algebras.

A classification of the formal generators of $\W (\mathbf S)$ arising from dissections of a graded orbifold surface $\mathbf S$ is given in \cite{barmeierschrollwang,kim}. The dissections giving rise to formal generators are called {\it formal dissections} in \cite{barmeierschrollwang} and they can be characterized geometrically (see Section \ref{section:formal}).

\subsection{Orbifold surfaces via deformations of Fukaya categories}
\label{subsection:seidel}

A main motivation for the study of Fukaya categories of orbifold surfaces stems from our work on deformations of path algebras \cite{barmeierwang}. We show in \cite{barmeierschrollwang2} that besides the three equivalent viewpoints in \eqref{eq:viewpoints}, which are intrinsic to the orbifold surface, the partially wrapped Fukaya category of an orbifold surface has yet another natural interpretation: it can be obtained as a {\it deformation} of the partially wrapped Fukaya category of a smooth surface. 

More precisely, it is shown in \cite{barmeierschrollwang2} that the partially wrapped Fukaya category of a smooth graded surface $\mathbf S_0$ admits a semi-universal family of A$_\infty$ deformations. This family may be viewed as an algebraization of the formal solution to the formal deformation problem. The family is defined over $\mathbb A^d$ where $d = \dim \HH^2 (\W (\mathbf S_0), \W (\mathbf S_0))$. Its fiber over $0 \in \mathbb A^d$ is $\W (\mathbf S_0)$, but being an algebraic family, it is possible to consider individual fibers of this family for all (closed) points of $\mathbb A^d$, corresponding to elements of $\Bbbk^d$. It turns out that every such fiber of the family is given by the partially wrapped Fukaya category of an orbifold surface $\mathbf S_\lambda$, where $\lambda \in \Bbbk^d$ and the orbifold singularities of $\mathbf S_\lambda$ are obtained by ``partially compactifying'' certain boundary components of $\mathbf S_0$. For the precise statement of this result we refer to \cite[Theorem 6.1]{barmeierschrollwang2}.

This deformation-theoretic viewpoint gives a geometric interpretation of the A$_\infty$ deformations of graded gentle algebras. Moreover, it can be understood as an instance of a general programme outlined in Seidel's ICM 2002 address \cite{seidel1} relating partial compactifications of symplectic manifolds and deformations of Fukaya categories. From this point of view, the results of \cite{barmeierschrollwang2} show several novel features:
\begin{enumerate}
\item In order to relate the full second Hochschild cohomology of partially wrapped Fukaya categories of surfaces to partial compactifications of the underlying surface, one is naturally led to orbifold singularities, as they arise through deformations of the Fukaya category.
\item In the case of partially wrapped Fukaya categories, partial compactifications may appear not only at ``fully wrapped'' boundary components (without stops) as considered in \cite{seidel1}, but also at partially wrapped boundary components or boundary components with entire boundary stops.
\item For a more complete view on the interplay between formal generators, deformations of Fukaya categories and partial compactifications, it is useful to consider also partially wrapped Fukaya categories that are not homologically smooth by allowing stops which are not Legendrian (cf.\ Remark \ref{remark:stops}).
\end{enumerate}

\begin{figure}
\[
\begin{tikzpicture}[x=1.3em,y=1.3em,decoration={markings,mark=at position 0.99 with {\arrow[black]{Stealth[length=4.8pt]}}}, scale=.6]
\draw[line width=.5pt] (0,0) ++ (180:4 and 1) arc[start angle=180,end angle=360,x radius=4,y radius=1, overlay] to[out=90,in=190,looseness=1] ++(2.5,4) arc[start angle=-90,end angle=90,x radius=.25,y radius=1] to[out=170,in=290,looseness=1, overlay] ++(-1,1) arc[start angle=20,end angle=135,radius=5, overlay] to[out=225,in=0,looseness=1] (-11,6) ++ (0,-2) to[out=0,in=90] (-9,0) arc[start angle=180,end angle=360,x radius=1,y radius=.25] arc[start angle=180,end angle=0,radius=1.5];
\begin{scope}
\node[font=\small, overlay, align=center] at (-5,8.5) {$\mathbf S_0$};
\draw[line width=.5pt,line cap=round] (5.5,7) ++(200:5) ++(0:2.5) arc[start angle=0,end angle=90,radius=2.5];
\draw[line width=.5pt,line cap=round] (5.5,7) ++(200:5) ++(0:2.5) ++(90:2.5) ++(280:2.5) arc[start angle=280,end angle=170,radius=2.5];
\end{scope}
\begin{scope}[shift={(-2.35,-2.35)}]
\draw[line width=.5pt,line cap=round] (5.5,7) ++(200:5) ++(0:2.5) arc[start angle=0,end angle=90,radius=2.5];
\draw[line width=.5pt,line cap=round] (5.5,7) ++(200:5) ++(0:2.5) ++(90:2.5) ++(280:2.5) arc[start angle=280,end angle=170,radius=2.5];
\end{scope}
\draw[line width=.6pt,dash pattern=on 0pt off 1.5pt,line cap=round, overlay] (0,0) ++ (180:4 and 1) arc[start angle=180,end angle=0,x radius=4,y radius=1, overlay] (6.5,4) arc[start angle=270,end angle=90,x radius=0.25,y radius=1, overlay] (-9,0) arc[start angle=180,end angle=0,x radius=1,y radius=.25, overlay];
\draw[line width=.2em,color=stopcolour] (-11,6) arc[start angle=90,end angle=450,x radius=.3, y radius=1];
\draw[fill=stopcolour, color=stopcolour, overlay] (85:4 and 1) circle(.25em);
\draw[fill=stopcolour, color=stopcolour, overlay] (235:4 and 1) circle(.25em);
\draw[fill=stopcolour, color=stopcolour, overlay] (325:4 and 1) circle(.25em);
\draw[fill=stopcolour, color=stopcolour] (-8,0) ++(260:1 and .25) circle(.25em);
\begin{scope}[xshift=32em]
\node[font=\small, overlay, align=center] at (-5.5,8.5) {$\mathbf S_\lambda$};
\draw[line width=.5pt, overlay] (0,0) ++ (180:4 and 1) arc[start angle=180,end angle=360,x radius=4,y radius=1] to[out=90,in=270,looseness=.9] ++(2,5) arc[start angle=0,end angle=135,radius=5.5] to[out=225,in=0,looseness=1] (-11,5) to[out=0,in=90] (-8,0) arc[start angle=180,end angle=0,radius=2];
\node[font=\tiny] at (-11,5) {$\times$};
\node[font=\tiny] at (-8,0) {$\times$};
\begin{scope}
\draw[line width=.5pt,line cap=round] (5.5,7) ++(200:5) ++(0:2.5) arc[start angle=0,end angle=90,radius=2.5];
\draw[line width=.5pt,line cap=round] (5.5,7) ++(200:5) ++(0:2.5) ++(90:2.5) ++(280:2.5) arc[start angle=280,end angle=170,radius=2.5];
\end{scope}
\begin{scope}[shift={(-2.35,-2.35)}]
\draw[line width=.5pt,line cap=round] (5.5,7) ++(200:5) ++(0:2.5) arc[start angle=0,end angle=90,radius=2.5];
\draw[line width=.5pt,line cap=round] (5.5,7) ++(200:5) ++(0:2.5) ++(90:2.5) ++(280:2.5) arc[start angle=280,end angle=170,radius=2.5];
\end{scope}
\draw[line width=.6pt,dash pattern=on 0pt off 1.5pt,line cap=round] (0,0) ++ (180:4 and 1) arc[start angle=180,end angle=0,x radius=4,y radius=1]; 
\draw[fill=stopcolour, color=stopcolour, overlay] (85:4 and 1) circle(.25em);
\draw[fill=stopcolour, color=stopcolour, overlay] (235:4 and 1) circle(.25em);
\draw[fill=stopcolour, color=stopcolour, overlay] (325:4 and 1) circle(.25em);
\end{scope}
\end{tikzpicture}
\]
\caption{An orbifold surface $\mathbf S_\lambda$ (right) arising through deformation of a partially wrapped Fukaya category of a smooth surface $\mathbf S_0$ (left).}
\label{fig:surface}
\end{figure}
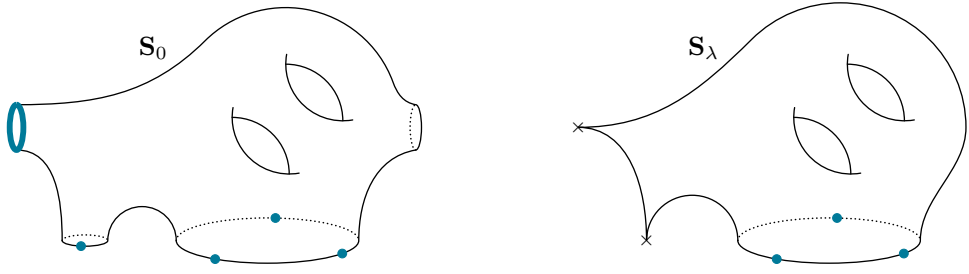

\section{Orbifold surfaces}
\label{section:orbifoldsurfaces}

In order to consider partially wrapped Fukaya categories of orbifold surfaces as $\mathbb Z$-graded A$_\infty$ categories, one should specify stop data and a grading structure, as in the following definition. 

\begin{definition}[{\cite[Definition 3.2]{barmeierschrollwang}}]
\label{definition:orbifoldsurface}
A {\it graded orbifold surface with stops} $\mathbf S = (S, \Sigma, \eta)$ consists of
\begin{itemize}
\item a compact connected oriented real orbifold surface $S$ (equipped with a smooth orbifold atlas) with nonempty smooth boundary $\partial S$ and isolated orbifold singularities
\item a closed nonempty subset $\Sigma \subset \partial S$ called the {\it stop}
\item a {\it grading structure} given by a smooth line field $\eta \in \Gamma (S, \mathbb P \mathrm T_S)$, i.e.\ a global section of the projectivized tangent (orbi)bundle of $S$.
\end{itemize}
The boundary $\partial S$ is homeomorphic to a disjoint union of circles. The connected components of $\Sigma \subset \partial S$ are thus either contractible or homeomorphic to a circle, and we call the former {\it boundary stops} and the latter {\it entire boundary stops}. We require that $\Sigma$ contain at least one boundary stop.
\end{definition}

We denote by $\Sing (S) \subset S \smallsetminus \partial S$ the set of orbifold points of the orbifold surface $S$, i.e.\ the set of points with a nontrivial stabilizer group. To reiterate this point, even though $S$ is smooth {\it as an orbifold}, we view the orbifold points as ``orbifold singularities'' to distinguish them from the points in a smooth manifold. In particular, by a {\it graded smooth surface} we mean a graded surface $\mathbf S = (S, \Sigma, \eta)$ such that $\Sing (S) = \varnothing$. Note that when $\Sing (S) = \varnothing$ and $\Sigma$ contains no entire boundary stops, we recover the graded smooth surfaces considered in \cite{haidenkatzarkovkontsevich,lekilipolishchuk2}.

We close the discussion of general orbifold surfaces with a few remarks providing further context for Definition \ref{definition:orbifoldsurface}.

\begin{remark}[Smooth structure on orbifold surfaces]
Any orbifold surface admits a unique smooth orbifold structure (up to orbifold diffeomorphism) \cite[Theorem 5.1.1]{choi} and this smooth orbifold structure is already determined by its topological data. As we only consider isolated orbifold singularities, the orbifold surface $S$ in Definition \ref{definition:orbifoldsurface} is completely determined by the genus and the number of boundary components of the underlying topological surface, and by the (necessarily finite) number of orbifold singularities.
\end{remark}


\begin{remark}[Symplectic structure and order of orbifold points]
\label{remark:orbifoldpoints}
Smooth real orbifold surfaces are topological spaces which are locally diffeomorphic to smooth disks $\mathbb D$ or to quotients of the form $\mathbb D / G$ for some finite group $G$. In the neighbourhood of an orbifold singularity with stabilizer group $G$, a section of $\mathbb P \mathrm T_S$ is given by a $G$-invariant section of the tangent bundle of $\mathbb D$ which can only exist for $G \simeq \mathbb Z_2$. In particular, it is an immediate consequence that all orbifold singularities of a {\it graded} orbifold surface are necessarily of order $2$ \cite[Lemma 3.1]{barmeierschrollwang}. Since the orbifold points are isolated, the orbifold charts are diffeomorphic to $\mathbb D / \mathbb Z_2$, where $\mathbb Z_2$ acts on the unit disk $\mathbb D \subset \mathbb R^2$ by rotation through an angle of $\pi$, i.e.\ $(x, y) \mapsto (-x, -y)$. This action has a unique orbifold point corresponding to the fixed point $0 \in \mathbb D \subset \mathbb R^2$ of the action. Note that this action is orientation-preserving and thus the symplectic (area) form $\mathrm d x \wedge \mathrm d y$ descends to a symplectic form on the orbifold.
\end{remark}

\begin{remark}[Stops]
\label{remark:stops}
The results of Ganatra, Pardon and Shende \cite{ganatrapardonshende1,ganatrapardonshende2} allow us to work in the generality of $\Sigma$ being any closed subset. Without loss of generality we may assume that the (contractible) boundary stops are singletons, rather than closed intervals. In particular, if $\Sigma$ contains no entire boundary stops, we may assume that $\Sigma$ is discrete in which case $\Sigma \subset \partial S$ is $0$-dimensional and hence a {\it Legendrian} subset of the $1$-dimensional boundary. The entire boundary stops are useful to give a geometric meaning to arbitrary (not necessarily homologically smooth) gentle algebras.
\end{remark}

\begin{remark}[Double cover]
Since the orbifold singularities are isolated and of order $2$, a graded orbifold surface $\mathbf S = (S, \Sigma, \eta)$ always admits a smooth double cover $\widetilde{\mathbf S} = (\widetilde S, \widetilde \Sigma, \widetilde \eta\,)$ (cf.\ \cite{amiotplamondon}), i.e.\ we can view $\mathbf S$ as a (global) $\mathbb Z_2$-quotient of a graded smooth surface. Note, however, that such a double cover need not be unique. For example, a cylinder with two orbifold points and $m$ resp.\ $n$ stops in its two boundary components admits the following two double covers: (1) a genus $0$ surface with two boundary components with $m$ stops and two with $n$ stops, (2) a genus $1$ surface with one boundary components with $2m$ stops and one with $2n$ stops.
\end{remark}

\subsection{Orbifold disks}
\label{subsection:orbifolddisks}

Orbifold disks are local models of general orbifold surfaces. In the remainder of this article our main focus lies on recalling the construction of the partially wrapped Fukaya category of an orbifold disk given in \cite{barmeierschrollwang} and then giving a new construction via a new type of dissection. Both constructions can therefore be used in the local-to-global construction of the partially wrapped Fukaya category of a general orbifold surface.

An orbifold disk with $n$ stops can be viewed as a $\mathbb Z_2$-quotient of a smooth disk with $2n$ stops as follows.

Let $\mathbf D_{2n} = (\mathbb D, \widetilde \Sigma, \widetilde \eta\,)$ be a graded smooth disk with $2n$ stops. Concretely, we may view $\mathbb D \subset \mathbb R^2$ as the (closed) unit disk and let \[\widetilde \Sigma = \{ (\cos (2 \pi k / 2n), \sin (2 \pi k / 2n) \}_{0 \leq k < 2n}\] be the set of $2n$ boundary stops. Up to homotopy, there is a unique line field on $\mathbb D$ and we let $\widetilde \eta$ be the horizontal line field. The $\mathbb Z_2$ action on $\mathbb D$, given by rotation around the origin through an angle $\pi$, leaves both $\widetilde \Sigma$ and $\widetilde \eta$ invariant, whence both descend to the quotient and we let $\Sigma = \widetilde \Sigma / \mathbb Z_2$ and let $\eta$ be the induced line field on $\mathbb D / \mathbb Z_2$. Then
\[
\mathbf D_{n}^\times := (\mathbb D / \mathbb Z_2, \Sigma, \eta)
\]
is a graded orbifold surface in the sense of Definition \ref{definition:orbifoldsurface} and we call $\mathbf D_n^\times$ a {\it graded orbifold disk with $n$ stops}. Since the origin is the unique fixpoint of the $\mathbb Z_2$ action, $\mathbf D_n^\times$ has exactly one orbifold point. See Fig.~\ref{fig:disk} for an illustration, where the orbifold point is marked by the symbol $\times$.

\begin{figure}
\centering
\begin{tikzpicture}[x=1em,y=1em]
\begin{scope}
\node[font=\scriptsize,left] at (-3.5,3.5) {$\mathbf D_{2n}$};
\foreach \k in {0,...,9} {%
\draw[fill=stopcolour, color=stopcolour] (36*\k:4em) circle(.16em);
}
\draw[clip] (0,0) circle(4em);
\foreach \r in {-4.25,-4,...,4.25} {%
\draw[line width=.2pt] (-4,\r) -- (4,\r);
}
\end{scope}
\begin{scope}
\draw[line width=.5pt] (0,0) circle(4em);
\foreach \k in {0,...,9} {%
\draw[fill=stopcolour, color=stopcolour] (36*\k:4em) circle(.16em);
}
\end{scope}
\begin{scope}[shift={(13em,0)}]
\draw[<->,line width=.5pt] (0.2,-2) arc[start angle=-90,end angle=90,radius=2em];
\draw[fill=black] circle(.07em);
\node[font=\scriptsize,right] at (2,0) {identify};
\draw[clip] (0,4) -- (0,-4) arc[start angle=270,end angle=90,radius=4em] -- cycle;
\draw[line width=.5pt] (0,4) -- (0,-4) arc[start angle=270,end angle=90,radius=4em] -- cycle;
\foreach \r in {-3.75,-3.5,...,3.75} {%
\draw[line width=.2pt] (-4,\r) -- (4,\r);
}
\end{scope}
\begin{scope}[shift={(13em,0)}]
\draw[line width=.5pt] (0,4) -- (0,-4) arc[start angle=270,end angle=90,radius=4em] -- cycle;
\foreach \k in {3,...,7} {%
\draw[fill=stopcolour, color=stopcolour] (36*\k:4em) circle(.16em);
}
\end{scope}
\begin{scope}[shift={(23em,0)}]
\node[font=\scriptsize,left] at (-2.8,2.8) {$\mathbf D^\times_n$};
\draw[line width=.5pt] (0,0) circle(3.186em);
\node[font=\scriptsize] at (0,0) {$\times$};
\draw[line width=.2pt] (180:3.18) to (0,0);
\draw[line width=.2pt] (2*19:3.18) to[bend left=26] (-2*19:3.18);
\draw[line width=.2pt] (2*27:3.18) to[bend left=40] (-2*27:3.18);
\draw[line width=.2pt] (2*33.3:3.18) to[bend left=50,looseness=1.07] (-2*33.3:3.18);
\draw[line width=.2pt] (2*38.9:3.18) to[bend left=57,looseness=1.19] (-2*38.9:3.18);
\draw[line width=.2pt] (2*43.6:3.18) to[bend left=60,looseness=1.34] (-2*43.6:3.18);
\draw[line width=.2pt] (2*48:3.18) to[bend left=62,looseness=1.51] (-2*48:3.18);
\draw[line width=.2pt] (2*52.2:3.18) to[bend left=64,looseness=1.69] (-2*52.2:3.18);
\draw[line width=.2pt] (2*56.2:3.18) to[bend left=66,looseness=1.9] (-2*56.2:3.18);
\draw[line width=.2pt] (2*60:3.18) to[bend left=68,looseness=2.13] (-2*60:3.18);
\draw[line width=.2pt] (2*63.5:3.18) to[bend left=70,looseness=2.39] (-2*63.5:3.18);
\draw[line width=.2pt] (2*67.1:3.18) to[bend left=72,looseness=2.73] (-2*67.1:3.18);
\draw[line width=.2pt] (2*70.5:3.18) to[bend left=74,looseness=3.14] (-2*70.5:3.18);
\draw[line width=.2pt] (2*73.9:3.18) to[bend left=76,looseness=3.72] (-2*73.9:3.18);
\draw[line width=.2pt] (2*77.1:3.18) to[bend left=78,looseness=4.52] (-2*77.1:3.18);
\draw[line width=.2pt] (2*80.3:3.18) to[bend left=82,looseness=5.78] (-2*80.3:3.18);
\draw[line width=.2pt] (2*83.5:3.18) to[bend left=86,looseness=8.3] (-2*83.5:3.18);
\draw[line width=.2pt] (2*86.8:3.18) to[bend left=89,looseness=16.15] (-2*86.8:3.18);
\end{scope}
\begin{scope}[shift={(23em,0)}]
\foreach \k in {0,...,4} {%
\draw[fill=stopcolour, color=stopcolour] (36+72*\k:3.186em) circle(.16em);
}
\end{scope}
\end{tikzpicture}
\caption{A smooth disk with $2n$ stops graded by the horizontal line field (left) and the graded orbifold disk with $n$ stops obtained as its $\mathbb Z_2$ quotient (right).}
\label{fig:disk}
\end{figure}
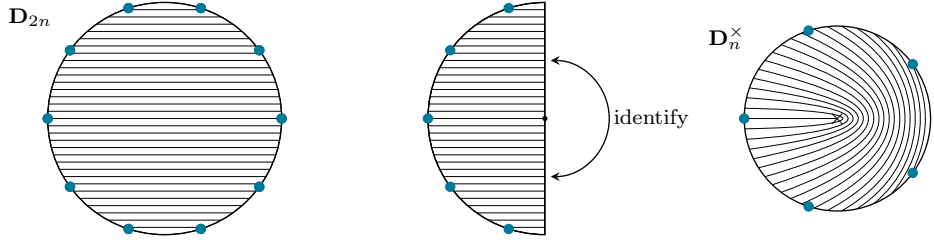

\subsection{Arcs on an orbifold disk}
\label{subsection:arcs}

We now introduce arcs on a graded orbifold surface which in later sections will serve as a set of ``generating Lagrangians'' for the partially wrapped Fukaya category.

\begin{definition}
An {\it arc} on a graded orbifold surface $\mathbf S = (S, \Sigma, \eta)$ is given by an embedding
\[
\gamma \colon [0, 1] \to S
\]
with $\gamma (0), \gamma (1) \in \Sing (S) \cup \partial S \smallsetminus \Sigma$ such that $\gamma ([0,1])$ is not nullhomotopic and $\gamma ((0,1)) \subset S \smallsetminus (\Sing (S) \cup \partial S)$. Here and unless stated otherwise, arcs are considered up to reparametrization and up to (orbifold) homotopy, where the homotopy is allowed to move the endpoints of $\gamma$ within $\partial S \smallsetminus \Sigma$. We also assume $\gamma$ to meet the boundary as well as other arcs transversally.
\end{definition}

Arcs on an orbifold disk $\mathbf D_n^\times = (\mathbb D / \mathbb Z_2, \Sigma, \eta)$ as in Section \ref{subsection:orbifolddisks} can be separated into the following three types:
\begin{enumerate}
\item arcs connecting two points in $\partial S \smallsetminus \Sigma$ whose images are continuously homotopic to the boundary
\item arcs connecting two points in $\partial S \smallsetminus \Sigma$ whose images are {\it not} continuously homotopic to the boundary
\item arcs connecting a point in a boundary segment in $\partial S \smallsetminus \Sigma$ and the orbifold point $0 \in \mathbb D / \mathbb Z_2$.
\end{enumerate}
Whereas we usually work with (smooth, orbifold) homotopies respecting the orbifold structure of $S$, by ``continuously homotopic'' we mean a homotopy in the underlying topological space, ignoring the orbifold structure.

All of these types of arcs lift to the double cover $\mathbf D_{2n}$ of $\mathbf D_n^\times$, but with different behaviour, illustrated in Fig.~\ref{fig:arcs}. An arc $\alpha$ of the first type lifts to a pair of arcs $\alpha^+, \alpha^-$ which are homotopic in the double cover. An arc $\beta$ of the second type lifts to a pair of arcs $\beta^+, \beta^-$ which are not homotopic in the double cover, but are interchanged by the $\mathbb Z_2$-action. An arc $\gamma$ of the third type lifts to a $\mathbb Z_2$-invariant arc $\gamma$ in $\mathbf D_{2n}$. (Note that in a general orbifold surface, there is a fourth type of arc, both of whose endpoints are orbifold points. This fourth type of arc lifts to a closed curve passing through the corresponding fixed points in the double cover.)

\begin{figure}
\centering
\begin{tikzpicture}[x=1em,y=1em]
\node[font=\scriptsize] at (159:2.2em) {$\alpha$};
\node[font=\scriptsize] at (275:2.5em) {$\beta$};
\node[font=\scriptsize] at (-12:2em) {$\gamma$};
\draw[line width=.5pt] (0,0) circle(3.2em);
\draw[fill=stopcolour, draw=stopcolour] (60:3.2em) circle(.16em);
\draw[fill=stopcolour, draw=stopcolour] (180:3.2em) circle(.16em);
\draw[fill=stopcolour, draw=stopcolour] (300:3.2em) circle(.16em);
\node[font=\scriptsize,shape=circle,scale=.6,fill=white] (X) at (0,0) {};
\node[font=\scriptsize] at (0,0) {$\times$};
\path[line width=.75pt,color=arccolour] (X) edge (0:3.2em);
\path[line width=.75pt,looseness=1.5, arccolour] (280:3.2em) edge[bend left=60] (320:3.2em);
\path[line width=.75pt,arccolour] (100:3.2em) edge[bend left=85,looseness=6] (140:3.2em); 
\begin{scope}[xshift=15em]
\draw[line width=.5pt, overlay] (0,0) circle(4em);
\foreach \a in {0,60,...,300} {%
\draw[fill=stopcolour, draw=stopcolour] (\a:4em) circle(.16em);
}
\path[line width=.75pt,looseness=1.5, arccolour] (284.5-60:4em) edge[bend left=60] (315.5-60:4em);
\path[line width=.75pt,looseness=1.5, arccolour, overlay] (284.5-240:4em) edge[bend left=60] (315.5-240:4em);
\path[line width=.75pt,arccolour, overlay] (90:4em) edge (270:4em);
\path[line width=.75pt,arccolour] (140:4em) edge (-20:4em);
\path[line width=.75pt,arccolour] (160:4em) edge (-40:4em);
\node[font=\scriptsize] at (120:2.6em) {$\alpha^-$};
\node[font=\scriptsize] at (180:2.6em) {$\alpha^+$};
\node[font=\scriptsize] at (221:3.3em) {$\beta^+$};
\node[font=\scriptsize] at (220-180:3.2em) {$\beta^-$};
\node[font=\scriptsize] at (261:2.6em) {$\gamma$};
\end{scope}
\end{tikzpicture}
\caption{Three types of arcs on an orbifold disk $\mathbf D_3^\times$ and their lifts to the double cover $\mathbf D_6$.}
\label{fig:arcs}
\end{figure}
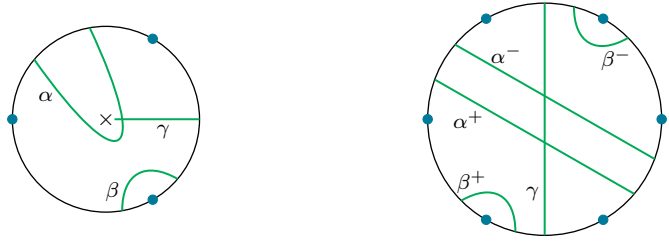

The Lagrangian intersection theory of curves in $\mathbf D_{2n}$, encoded by its partially wrapped Fukaya category $\W (\mathbf D_{2n})$, can now be restricted to the $\mathbb Z_2$-invariant structure which then descends to $\mathbf D^\times_n$. In general, there is a close relationship between the $\mathbb Z_2$-invariant Lagrangian intersection theory in $\widetilde{\mathbf S}$ and the Lagrangian intersection theory in $\mathbf S$ made precise by the equivalences \eqref{eq:viewpoints}. We now give an overview for the case of the orbifold disk.

\subsection{Admissible dissections}
\label{subsection:ribbongraph}

Similar to the smooth surfaces in \cite{haidenkatzarkovkontsevich}, it is useful to define the partially wrapped Fukaya category of an orbifold surface with stops via a notion of admissible dissections. Roughly speaking, such a dissection should contain enough arcs (Lagrangians) to generate all expected objects of the partially wrapped Fukaya category, but not too many arcs in order to avoid unnecessarily many higher operations. Such a dissection decomposes the surface into ``simple pieces''. Here we give the definition for orbifold disks only, as the general definition becomes more involved when the surface contains more than one orbifold. (See \cite[Sections 5 and 6]{barmeierschrollwang} for the general case.)

\begin{definition}
\label{definition:dissection}
Let $\Gamma$ be a finite collection of arcs on a graded orbifold disk $\mathbf D_n^\times = (\mathbb D / \mathbb Z_2, \Sigma, \eta)$ such that any two arcs intersect either at the orbifold point $0$ or not at all. The complement $(\mathbb D / \mathbb Z_2) \smallsetminus \Gamma$ is decomposed into its connected components as follows
\[
(\mathbb D / \mathbb Z_2) \smallsetminus \Gamma = \bigsqcup_{1 \leq i \leq k} P_{v_i}.
\]
Then $\Gamma$ is called a {\it dissection} of $\mathbf D_n^\times$ if the following hold:
\begin{enumerate}
\item Each $P_{v_i}$ is a smooth disk containing at most one boundary stop. In particular, $0 \notin  P_{v_i}$ for all $1 \leq i \leq k$.

\item If $P_{v_{i_1}}, \dotsc, P_{v_{i_m}}$ are the connected components around the orbifold point $0$, we have that $m \geq 2$ and at least one of the $P_{v_{i_l}}$'s does not contain any boundary stop. This condition allows us to introduce the orbifold stop at the orbifold point, see Definition \ref{definition:orbifoldstop}. 
\end{enumerate}
\end{definition}

We call the $P_{v_i}$'s {\it polygons}, an $m$-gon being a polygon whose closure contains $m$ curves belonging to $\Gamma$ (possibly counted twice). Since the $P_{v_i}$'s are connected components of the complement $(\mathbb D / \mathbb Z_2) \smallsetminus \Gamma$, they are open in the subspace topology of $\mathbb D / \mathbb Z_2$, but they will contain parts of the boundary of $\mathbb D / \mathbb Z_2$, thus we may ask whether or not some $P_{v_i}$ contains a boundary stop. Definition \ref{definition:dissection} requires there to be at least one arc connecting to the orbifold point $0$. In particular, $0$ does not belong to any $P_{v_i}$. 

We now introduce the notion of an {\it admissible dissection}, which allows us to give a concrete construction of the partially wrapped Fukaya category $\W (\mathbf S)$.

\begin{definition}\label{definition:orbifoldstop}
Let $\Gamma$ be a dissection. An {\it orbifold stop} at the orbifold point is an angle between two arcs in $\Gamma$ incident to the orbifold point such that no other stops lie in the same polygon.

An {\it admissible dissection} $\Delta$ consists of a dissection $\Gamma$ together with a choice of orbifold stop at the orbifold point.
\end{definition}

\begin{remark}[]
Arcs connecting to the orbifold point correspond to $\mathbb Z_2$-invariant arcs in the double cover. Such $\mathbb Z_2$-invariant arcs have the semi-simple ring $\Bbbk [\mathbb Z_2] \simeq \Bbbk \times \Bbbk$ as endomorphism ring and hence split into two direct summands in the idempotent completion of the orbit category $\W (\mathbf S) / \mathbb Z_2$. The choice of an orbifold stop is essentially equivalent to the choice of suitable direct summands in the orbit category of the $\mathbb Z_2$-invariant arcs in the double cover of $S$. This choice simplifies the morphism spaces between arcs connecting to orbifold points: it forces there to be exactly one maximal path around an orbifold point as in Section \ref{section:disk}. The orbifold stop indicates the missing morphism in one of the polygons around the orbifold point which in the figures we illustrate by {\color{stopcolour} $\bullet$} near the orbifold point $0$ (see e.g.\ Fig.~\ref{fig:orbifoldannulus}).
\end{remark}

\begin{figure}[t]
\centering
\begin{tikzpicture}[x=1em,y=1em,decoration={markings,mark=at position 0.55 with {\arrow[black]{Stealth[length=4.8pt]}}}]
\begin{scope}
\node[font=\small] at (-3.8em,2.5em) {$\mathbf D_3^\times$\strut};
\node[font=\scriptsize,shape=circle,scale=.6] (X) at (0,0) {};
\draw[line width=0,postaction={decorate}] (210:3.5em) arc[start angle=210, end angle=250, radius=3.5em];
\draw[line width=0,postaction={decorate}, overlay] (90:3.5em) arc[start angle=90, end angle=135, radius=3.5em];
\draw[line width=0,postaction={decorate}] (170:3.5em) arc[start angle=170, end angle=215, radius=3.5em];
\draw[line width=0,postaction={decorate}] (295:3.5em) arc[start angle=295, end angle=335, radius=3.5em];
\draw[line width=.5pt, overlay] (0,0) circle(3.5em);
\draw[fill=stopcolour,color=stopcolour] (30:3.5em) circle(.15em);
\draw[fill=stopcolour,color=stopcolour] (150:3.5em) circle(.15em);
\draw[fill=stopcolour,color=stopcolour, overlay] (270:3.5em) circle(.15em);
\path[line width=.75pt,color=arccolour] (X) edge (330:3.5em);
\path[line width=.75pt,color=arccolour] (X) edge (210:3.5em);
\path[line width=.75pt,color=arccolour, overlay] (X) edge (90:3.5em);
\path[line width=.75pt,color=arccolour, overlay] (245:3.5em) edge[bend left=55, looseness=1.1] (295:3.5em);
\path[line width=.75pt,color=arccolour, overlay] (130:3.5em) edge[bend left=55, looseness=1.1] (170:3.5em);
\draw[->, line width=.5pt] (207:.7em) arc[start angle=207, end angle=95, radius=.7em];
\draw[->, line width=.5pt] (87:.7em) arc[start angle=87, end angle=-27, radius=.7em];
\node[font=\scriptsize] at (150:1.15em) {$q_1$};
\node[font=\scriptsize] at (30:1.2em) {$q_2$};
\node[font=\scriptsize, overlay] at (110.5:4.1em) {$p_1$};
\node[font=\scriptsize, overlay] at (189.5:4.2em) {$p_2$};
\node[font=\scriptsize, overlay] at (227.5:4.2em) {$p_3$};
\node[font=\scriptsize, overlay] at (312.5:4.2em) {$p_4$};
\node[font=\scriptsize] at (0,0) {$\times$};
\draw[fill=stopcolour,color=stopcolour] (270:.5em) circle(.15em);
\end{scope}
\begin{scope}[xshift=12em]
\node[font=\scriptsize,shape=circle,scale=.6] (X) at (0,0) {};
\draw[line width=.5pt, overlay] (0,0) circle(3.5em);
\draw[fill=stopcolour,color=stopcolour] (30:3.5em) circle(.15em);
\draw[fill=stopcolour,color=stopcolour] (150:3.5em) circle(.15em);
\draw[fill=stopcolour,color=stopcolour, overlay] (270:3.5em) circle(.15em);
\path[line width=.75pt,color=arccolour] (X) edge (330:3.5em);
\path[line width=.75pt,color=arccolour] (X) edge (210:3.5em);
\path[line width=.75pt,color=arccolour, overlay] (X) edge (90:3.5em);
\path[line width=.75pt,color=arccolour, overlay] (245:3.5em) edge[bend left=55, looseness=1.1] (295:3.5em);
\path[line width=.75pt,color=arccolour, overlay] (130:3.5em) edge[bend left=55, looseness=1.1] (170:3.5em);
\node[font=\scriptsize] at (0,0) {$\times$};
\draw[fill=stopcolour,color=stopcolour] (270+120:.5em) circle(.15em);
\end{scope}
\end{tikzpicture}
\caption{An admissible dissection of $\mathbf D_3^\times$ (left), where $p_1, p_2, p_3, p_4$ are boundary paths and $q_1, q_2$ are orbifold paths. The right dissection is not admissible, since there are two stops in the upper right polygon, one boundary stop and one orbifold stop.}
\label{fig:orbifoldannulus}
\end{figure}
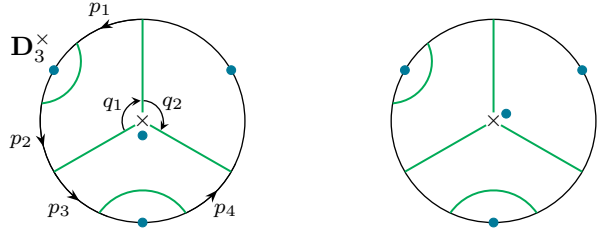

\section{A$_\infty$ categories and twisted complexes}
\label{section:ainfinitytwisted}

A fundamental tool for constructing Fukaya categories is the theory of twisted complexes \cite{bondalkapranov,kontsevich1,seidel2}. The basic idea is to consider a collection $\Delta$ of generating Lagrangians in a symplectic manifold. These Lagrangians define an A$_\infty$ category $\mathbf A_\Delta$, whose objects are given by the members in the collection $\Delta$ and the A$_\infty$ operations of $\mathbf A_\Delta$ account for their Floer theory. One then considers the category $\tw (\mathbf A_\Delta)^\natural$ of twisted complexes over $\mathbf A_\Delta$, which is the (idempotent completed) pretriangulated A$_\infty$ category generated by $\Delta$, namely the closure of $\mathbf A_\Delta$ under shifts, mapping cones and direct summands. Fukaya, Oh, Ohta and Ono proposed that the Lagrangian surgery (connected sum) operation corresponds to a mapping cone \cite[Chapter 10]{fukayaohohtaono}, see also \cite[Section 9.3]{palmerwoodward}. Twisted complexes can be viewed as iterated mapping cones, thus modelling ``new'' Lagrangians not belonging to $\Delta$.

Indeed, in the case of smooth graded surfaces, it is shown in \cite{haidenkatzarkovkontsevich} that, up to shifts, the indecomposable objects of $\W (\mathbf S) \simeq \tw (\mathbf A_\Delta)^\natural$ correspond precisely to homotopy classes of open and closed curves, possibly self-intersecting, where the closed curves have winding number $0$ with respect to the line field and furnished with an indecomposable local system of finite rank, thus modelling all (gradable) immersed Lagrangian submanifolds of $\mathbf S$. This classification result can be seen as extending the classification into string and band objects in the derived category of a zero-graded gentle algebra \cite{bekkertmerklen,burbandrozd} to the graded case. See \cite{opperplamondonschroll} for the classification of objects in the bounded derived category of a proper (i.e.\ finite-dimensional) graded gentle algebra and \cite{schrollicm} for a more detailed overview.

We give a short summary of the algebraic formalism of twisted complexes and refer to \cite{keller0,lefevre,seidel2} and to \cite[Section 2]{barmeierschrollwang} for more details.

\subsection{A$_\infty$ categories} 

Whereas DG categories only have two operations -- an associative composition $\mu_2$ and a differential $\mu_1$, together satisfying the Leibniz rule -- an A$_\infty$ category is equipped with infinitely many operations $\mu_n$ for $n \geq 1$ satisfying infinitely many equations. In the $\mathbb Z$-graded setting, there are many signs entering these equations. An immensely useful convention is to work with a shifted version of A$_\infty$ categories in which case all signs can be deduced from the Koszul sign rule.

Let $V$ be a $\mathbb Z$-graded vector space and let $V^i$ denote its degree $i$ component. We shall write $\s V$ for its shift, where $(\s V)^i = V^{i+1}$. Denoting by $\lvert v \rvert = i$ the degree of $v \in V^i$, we have that $\s v \in \s V$ has degree $\lvert \s v \rvert = \lvert v \rvert - 1 = i - 1$.

\begin{definition}
\label{definition:ainfinity}
An {\it A$_\infty$ category} $\mathbf A$ consists of a set $\mathrm{Ob} (\mathbf A)$ of objects, a graded $\Bbbk$-vector space $\mathbf A (X, Y)$ for each pair $X, Y \in \mathrm{Ob}(\mathbf A)$, and a collection of maps of degree~$1$
\[
\mu_n \colon \s \mathbf A (X_{n-1}, X_n) \otimes \s \mathbf A (X_{n-2}, X_{n-1}) \otimes \dotsb \otimes \s \mathbf A (X_0, X_1)
 \to \s \mathbf A (X_0, X_n)
\]
for each $n \geq 1$, satisfying the A$_\infty$ relations
\begin{equation}
\label{eq:ainfinityrelations}
\sum_{i=0}^{n-1}\sum_{j=1}^{n-i} (-1)^{\maltese_i} \mu_{n-j+1} (\s a_{n...i+j+1} \otimes \mu_j (\s a_{i+j...i+1}) \otimes \s a_{i...1}) = 0
\end{equation}
where here and elsewhere we set $\s a_{l \dotsc k} := \s a_l \otimes \dotsb \otimes \s a_{k+1} \otimes \s a_k$ for any indices $1 \leq k \leq l$ and we use $\maltese_i = \lvert \s a_{i \dotsc 1} \rvert = |a_i| + \dotsb + |a_1| - i$ as a shorthand for the (shifted) degree of consecutive morphisms.
\end{definition}

\begin{remark}
\label{remark:ainfinity}
We may set
\begin{equation*}
(\mu_{n-j+1} \bullet_i \mu_j) (\s a_{n \dotsc 1}) =
(-1)^{\maltese_i} \mu_{n-j+1} (\s a_{n ... i+j+1} \otimes \mu_j (\s a_{i+j...i+1}) \otimes \s a_{i ...1})
\end{equation*}
where $i$ is the number of tensor factors appearing before $\mu_j$, read from the right. If we now define
\[
\mu_{n-j+1} \bullet \mu_j = \sum_{i = 0}^{n-j} \mu_{n-j+1} \bullet_i \mu_j
\]
then \eqref{eq:ainfinityrelations} can be written simply as $\mu \bullet \mu = 0$ (see Section \ref{subsection:hhdefor}).
\end{remark}

\begin{example}
A DG category  $\mathbf A$ (in the usual unshifted sense) with differential $d$ and composition $\cdot$ is an A$_\infty$ category (in the sense of Definition \ref{definition:ainfinity}) with $\mu_n = 0$ for all $n \geq 3$ and 
\[
 \mu_1 (\s a_1) = (-1)^{|a_1|} \s  d a_1, \qquad  \mu_2 (\s a_2 \otimes \s a_1) = (-1)^{|a_1|} \s  (a_2 \cdot a_1)
\]
for any $a_1 \in \mathbf A (X_0, X_1)$ and $a_2 \in \mathbf A (X_1, X_2)$.
\end{example}

\begin{definition}
\label{definition:formal}
An A$_\infty$ category $\mathbf A$ is called {\it formal} if it is A$_\infty$ quasi-equivalent to the cohomology category $\H^\bullet (\mathbf A)$ where $\mathrm{Ob} (\H^\bullet (\mathbf A)) := \mathrm{Ob} (\mathbf A)$ and $\H^\bullet (\mathbf A) (X, Y) := \H^\bullet (\mathbf A (X, Y))$. 
\end{definition}

\begin{remark}
For any A$_\infty$ category $\mathbf A$, the homotopy transfer theorem (see for example \cite[Section 1.4]{lefevre}) implies that we may endow $\H^\bullet(\mathbf A)$ with an A$_\infty$ structure (called an {\it A$_\infty$ minimal model} of $\mathbf A$) such that the resulting A$_\infty$ category is A$_\infty$ quasi-equivalent to $\mathbf A$. Then an A$_\infty$ category $\mathbf A$ is formal if and only if its A$_\infty$ minimal model is A$_\infty$-isomorphic to the cohomology category $\H^\bullet(\mathbf A)$ with trivial higher products.
\end{remark}

The following result can be used to show the non-formality of an A$_\infty$ category $\mathbf A$. 

\begin{proposition}[{\cite[Proposition 2.8]{barmeierschrollwang}}] \label{proposition:formalcriterion}
Let $\mathbf A$ be a formal A$_\infty$ category. Then any full A$_\infty$ subcategory $\mathbf B \subset \mathbf A$ is formal.
\end{proposition}

\subsection{Twisted complexes}

We now define the A$_\infty$ category $\tw(\mathbf A)$ of twisted complexes associated to any A$_\infty$ category $\mathbf A$.

The first step is to consider the closure of $\mathbf A$ under shifts, i.e.\ we first consider the A$_\infty$ category $\mathbb Z \mathbf A$ whose objects are the pairs $(X, m)$ of an object $X$ in $\mathbf A$ and $m \in \mathbb Z$, which we shall write as $\s^m X$ and is sometimes also denoted $X [m]$. Morphism spaces are defined by 
\[
\mathbb Z \mathbf A (\s^m X, \s^l Y) = \s^{l-m} \mathbf A (X, Y). 
\]
The A$_\infty$ product of $\mathbb Z \mathbf A$ is obtained by extending the A$_\infty$ product of $\mathbf A$. 

\begin{definition}
\label{definition:twisted}
A {\it twisted complex} $X^\bullet$ is given by a sequence
\[
(X^1, X^2, \dotsc, X^r) := (\s^{m_1} X_1, \s^{m_2} X_2, \dotsc, \s^{m_r} X_r)
\]
of objects in $\mathbb Z \mathbf A$, together with a matrix $\delta = (\delta^{ij})_{1 \leq i, j \leq r}$ of degree $1$ morphisms
\[
\delta^{ij} \in \mathbb Z \mathbf A (X^j, X^i) = \s^{m_i-m_j} \mathbf A (X_j, X_i)
\]
such that $\delta^{ij} = 0$ for all $i \geq j$ and 
\[
\sum_{n \geq 1} \mu^{\mathbb Z\mathbf A}_n (\s\delta \otimes \dotsb \otimes \s \delta) = 0.
\]
We may write $\delta^{ij} = \s^{m_i-m_j} \delta_{ij}$ for $\delta_{ij} \in \mathbf A (X_j, X_i)$.
\end{definition}

Twisted complexes form an A$_\infty$ category $\tw (\mathbf A)$ and the morphism space from $X^\bullet$ to $Y^\bullet$ in $\tw (\mathbf A)$ is given by $\bigoplus_{i,j} \mathbb Z\mathbf A (X^j, Y^i)$. The A$_\infty$ product of $\tw (\mathbf A)$ is given by the higher product twisted by $\delta$
\[
\mu_n^{\tw(\mathbf A)}(\s a_{n\dotsc1})  = \!\!\sum_{i_0,\dotsc,i_n\geq0} \mu_{n+i_0+\dotsb+i_n}^{\mathbb Z \mathbf A} \! ((\s\delta)^{\otimes i_n} \otimes \s a_n  \otimes \dotsb \otimes   (\s\delta)^{\otimes i_1}  \otimes \s a_1 \otimes (\s\delta)^{\otimes i_0}).
\]

\begin{remark}
The homotopy category $\H^0 \tw (\mathbf A)$ is triangulated. There is a natural embedding $\mathbf A \to \tw(\mathbf A)$ of A$_\infty$ categories sending an object $X$ in $\mathbf A$ to the one-term twisted complex $X^\bullet = (X)$ with $\delta = 0$.

The mapping cone of a morphism $f\colon X_0 \to X_1$ of degree $0$ in $\mathbf A$ is defined to be the two-term twisted complex $(\s X_0, X_1)$ with the differential $-f \colon \s X_0 \to  X_1$ \cite[Eq.~(3.28)]{seidel2}. 
\end{remark}

\begin{remark}\label{remark:perfectderived}
For any A$_\infty$ category $\mathbf A$, we have the notion of derived category $\mathcal D (\mathbf A)$ of A$_\infty$ modules over $\mathbf A$. Denote by $\per(\mathbf A)$ the full subcategory consisting of {\it compact} objects in $\mathcal D(\mathbf A)$. Then there is a natural triangle equivalence 
\[
\per(\mathbf A) \simeq \mathrm H^0 \tw(\mathbf A)^\natural
\]
where $^\natural$ denotes idempotent completion (see e.g.\ \cite[Chapitre 7]{lefevre}).
\end{remark}

Recall that any A$_\infty$ functor $F\colon \mathbf A \to \mathbf A'$ induces an A$_\infty$ functor \[\tw(F) \colon \tw(\mathbf A) \to \tw(\mathbf A')\] and if $F$ is fully faithful then so is $\tw(F)$ (see \cite[Section I.3]{seidel2}).

\begin{definition}\label{definition:moritaequivalence}
A strictly unital A$_\infty$ functor $F\colon \mathbf A \to \mathbf A'$ is a {\it Morita equivalence} if the induced functor $\tw(F)^\natural \colon \tw(\mathbf A)^\natural \to \tw(\mathbf A')^\natural$ is an A$_\infty$ quasi-equiv\-a\-lence. More generally, two A$_\infty$ categories $\mathbf A$ and $\mathbf A'$ are called {\it Morita equivalent} if there is a zigzag of Morita equivalences connecting $\mathbf A$ and $\mathbf A'$.
\end{definition}

\begin{remark}\label{remark:quasimoritaequivalence}
It follows from \cite[Lemma 3.25]{seidel2} that an A$_\infty$ quasi-equivalence $F\colon \mathbf A \to \mathbf A'$ induces a Morita equivalence between $\mathbf A$ and $\mathbf A'$. In particular, any formal A$_\infty$ category $\mathbf A$ is Morita equivalent to the graded category $\H^\bullet(\mathbf A)$. 
\end{remark}

\subsection{A$_\infty$ orbit categories and twisted complexes}
\label{subsection:orbitcategories}

The partially wrapped Fukaya category of an orbifold surface admits several equivalent descriptions \eqref{eq:viewpoints}, one of which is as {\it orbit category} of the partially wrapped Fukaya category of a smooth surface \cite{chokim, amiotplamondon2, barmeierschrollwang}. We briefly recall the notion of orbit categories for A$_\infty$ categories. 

Let $\mathbf A$ be an A$_\infty$ category and let $G$ be a finite group. An {\it action} of $G$ on $\mathbf A$ is a group homomorphism from $G$ to the group of strict A$_\infty$-automorphisms of $\mathbf A$. Each $g \in G$ thus gives a strict A$_\infty$-automorphism $F_g \colon \mathbf A \to \mathbf A$ satisfying $F_g \circ F_h = F_{gh}$ for all $g, h \in G$ as well as $F_{1_G} = \id_{\mathbf A}$. We simply write the action on objects and morphisms as $g X = F_g (X)$ and $g \phi = F_g (\phi)$, respectively.

\begin{definition}[{\cite[Definition 5.6]{opperzvonareva}}]
Let $\mathbf A$ be an A$_\infty$ category with a $G$-action. The {\it orbit category} $\mathbf A / G$ is the A$_\infty$ category with the same objects as $\mathbf A$ and morphisms between two objects $X$ and $Y$ given by
\[
(\mathbf A / G) (X, Y) = \bigoplus_{g \in G} \mathbf A (X, g Y).
\]
Compositions \[\mu_n^{\mathbf A / G} \colon (\mathbf A / G) (X_{n-1}, X_n) \otimes \dotsb \otimes (\mathbf A / G) (X_0, X_1) \to (\mathbf A / G) (X_0, X_n)\] are defined for all $n \geq 1$ and induced componentwise by the higher multiplications of $\mathbf A$ via the formula
{\small
\begin{align*}
\mathbf A (X_{n-1}, g_n X_n) \otimes \dotsb \otimes \mathbf A (X_0, g_1 X_1) &\to \mathbf A (X_0, g_1 \dotsb g_n X_n) \\
\phi_n \otimes \dotsb \otimes \phi_2 \otimes \phi_1 &\mapsto \mu_n^{\mathbf A} (F_{g_1 \dotsb g_{n-1}} (\phi_n) \otimes \dotsb \otimes F_{g_1} (\phi_2) \otimes \phi_1).
\end{align*}
}%
\end{definition}

An action of $G$ on an A$_\infty$ category $\mathbf A$ naturally induces an action of $G$ on the A$_\infty$ category $\tw (\mathbf A)$ of twisted complexes. Concretely, given a twisted complex $(\bigoplus_i \s^{m_i} X_i, \delta)$ in $\tw (\mathbf A)$, its translate under $g \in G$ is $g (\textstyle\bigoplus_i \s^{m_i} X_i, \delta) = (\bigoplus_i \s^{m_i} g X_i, g \delta)$, 
where $g \delta = (g \delta_{ij})_{i,j}$.

A$_\infty$ orbit categories and some of their basic properties were first considered in \cite{opperzvonareva}, with further properties being shown independently in \cite{amiotplamondon2} and in \cite[Section 2.6]{barmeierschrollwang}. Some of their properties are summarized in the following proposition.

\begin{proposition} Let $\mathbf A$ be an A$_\infty$ category with a $G$-action. Then the following hold.
\begin{enumerate}
\item The compositions $\mu_n^{\mathbf A / G}$ make $\mathbf A / G$ into an A$_\infty$ category.
\item If $\mathbf A$ is an A$_\infty$ category with a $G$-action, then the A$_\infty$ category $\tw (\mathbf A)$ of twisted complexes also inherits a natural $G$-action.
\item If $\mathbf A$ is a pretriangulated A$_\infty$ category with a $G$-action, then $\H^0 \mathbf A$ is a triangulated category with $G$-action and $\H^0 (\mathbf A / G) \simeq (\H^0 \mathbf A) / G$.
\item Assume that $\mathrm{char}(\Bbbk)$ does not divide $|G|$. Then the natural functor $$(\tw (\mathbf A) / G)^\natural \to \tw (\mathbf A / G)^\natural$$ is an A$_\infty$-quasi-equivalence.
\end{enumerate}
\end{proposition}

\section{Fukaya categories of orbifold disks via dissections}
\label{section:arcsystems}

The arcs in an admissible dissection $\Delta$ of an orbifold disk (see Section \ref{subsection:arcs}) will be the objects of an A$_\infty$ category $\mathbf A_\Delta$ of generating Lagrangians. We now describe the morphisms and higher structure of $\mathbf A_\Delta$.

\subsection{Boundary paths and orbifold paths}

The morphisms in the Fukaya category of orbifold surfaces consist of two different paths between arcs which we call {\it boundary paths} and {\it orbifold paths}. The former were already used in \cite{haidenkatzarkovkontsevich} and the latter are a natural generalization for orbifold surfaces where arcs may end in orbifold points, which correspond to morphisms of intersecting arcs in the double cover. 

\begin{definition}
By a {\it boundary path} from an arc $\gamma$ to $\gamma'$ in a graded orbifold surface $\mathbf S = (S, \Sigma, \eta)$ we mean a boundary interval with the induced orientation of $\partial S$, i.e.\ an embedding $p \colon [0, 1] \to \partial S \smallsetminus \Sigma$, considered up to reparametrization, such that $p (0) \in \gamma \cap \partial S$ and $p (1) \in \gamma' \cap \partial S$. Given two arcs $\gamma, \gamma' \in \Delta$, each boundary path from a point in $\gamma \cap \partial S$ to a point in $\gamma' \cap \partial S$ will be a morphism in the category $\mathbf A_\Delta$. Here, we assume that $\gamma$ and $\gamma'$ are in minimal position.

Let $\gamma, \gamma' \in \Gamma$ be two arcs intersecting at an orbifold point $x \in \Sing (S)$. By an {\it orbifold path} from $\gamma$ to $\gamma'$ at $x$, we mean a clockwise angle locally from $\gamma$ to $\gamma'$ based at $x$, which does not pass through the angle given by the orbifold stop. Note that the existence of an orbifold stop implies that there is a unique maximal orbifold path, i.e.\ any orbifold path at $x$ is its subpath.

We refer to boundary paths and orbifold paths collectively as {\it paths}. 
\end{definition}

\subsection{Grading via line fields}
\label{subsection:grading}

Let us give a quick recall on a grading structure on $(S, \Sigma)$ given by a line field $\eta \in \Gamma (S, \mathbb P \mathrm T_S)$, so that $\mathbf S = (S, \Sigma, \eta)$ is a {\it graded} orbifold surface with stops. There are two equivalent definitions of gradings. Here we only mention the one used in \cite{lekilipolishchuk2}, namely the degree can be given by the {\it winding numbers} of $\eta$, without considering gradings of arcs.

The line field $\eta$ restricts to a line field on each polygon $P_v$. Since different polygons are glued along the arcs in $\Delta$, we may change the line field up to homotopy so that $\eta$ is parallel to the arcs in $\Delta$. Up to homotopy, every line field can be obtained by gluing such line fields on the polygons. 

Let $\eta$ be a line field such that it is parallel to the arcs in $\Delta$. Then we may assign an integer $\theta_i$ to each boundary segment and each segment of a fixed small circle $\mathrm S^1_x$ around an orbifold point $x \in \Sing (S)$ lying in any polygon $P_v$, including the segments with (boundary or orbifold) stops. Namely $\theta_i$ is the {\it winding number} of the boundary segment of $\partial S$, or circle segment of $\mathrm S^1_x$, which is the signed count of how often the tangent line of the boundary or circle segment agrees with the line given by the restriction of $\eta$ to that segment, counterclockwise rotation contributing positively. For example, for a path
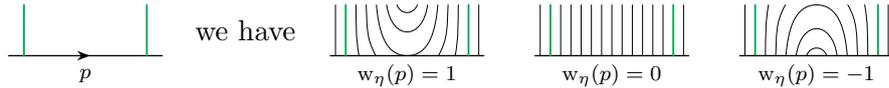
\[
\begin{tikzpicture}[x=.35em, y=.35em, decoration={markings,mark=at position 0.53 with {\arrow[black]{Stealth[length=4.8pt]}}}]
\begin{scope}
\draw[line width=.5pt, line cap=round, postaction={decorate}] (-1.5,0) -- (13.5,0);
\draw[line width=.7pt, line cap=round, color=arccolour] (0,0) -- (0,5);
\draw[line width=.7pt, line cap=round, color=arccolour] (12,0) -- (12,5);
\node[font=\scriptsize, overlay] at (6,-1.95) {$p$};
\node[right] at (5.5em, 2.5) {we have};
\end{scope}
\begin{scope}[xshift=11em]
\draw[line width=.5pt, line cap=round] (-1.5,0) -- (13.5,0);
\draw[line width=.2pt, line cap=round] (-1,0) -- (-1,5);
\draw[line width=.2pt, line cap=round] (13,0) -- (13,5);
\draw[line width=.7pt, line cap=round, color=arccolour] (0,0) -- (0,5);
\draw[line width=.7pt, line cap=round, color=arccolour] (12,0) -- (12,5);
\draw[line width=.2pt, line cap=round] (1,5) to[in=100, out=271] (1.3,0);
\draw[line width=.2pt, line cap=round] (2,5) to[in=120, out=273] (2.95,0);
\draw[line width=.2pt, line cap=round] (3,5) to[in=180, out=275] (6,0);
\draw[line width=.2pt, line cap=round] (9,5) to[in=0, out=265] (6,0);
\draw[line width=.2pt, line cap=round] (4,5) to[bend right=82, looseness=2.3] (8,5);
\draw[line width=.2pt, line cap=round] (5,5) to[bend right=68, looseness=1.4] (7,5);
\draw[line width=.2pt, line cap=round] (10,5) to[in=60, out=267] (9.05,0);
\draw[line width=.2pt, line cap=round] (11,5) to[in=80, out=269] (10.7,0);
\node[font=\scriptsize, overlay] at (6,-1.95) {$\mathrm w_\eta (p) = 1$};
\end{scope}
\begin{scope}[xshift=18em]
\draw[line width=.5pt, line cap=round] (-1.5,0) -- (13.5,0);
\draw[line width=.2pt, line cap=round] (-1,0) -- (-1,5);
\draw[line width=.2pt, line cap=round] (13,0) -- (13,5);
\foreach \c in {1,...,11} {
\draw[line width=.2pt, line cap=round] (\c,0) -- (\c,5);
}
\draw[line width=.7pt, line cap=round, color=arccolour] (0,0) -- (0,5);
\draw[line width=.7pt, line cap=round, color=arccolour] (12,0) -- (12,5);

\node[font=\scriptsize, overlay] at (6,-1.95) {$\mathrm w_\eta (p) = 0$};
\end{scope}
\begin{scope}[xshift=25em]
\begin{scope}[rotate=180, yshift=-1.75em, xshift=-4.2em]
\draw[line width=.2pt, line cap=round] (-1,0) -- (-1,5);
\draw[line width=.2pt, line cap=round] (13,0) -- (13,5);
\draw[line width=.7pt, line cap=round, color=arccolour] (0,0) -- (0,5);
\draw[line width=.7pt, line cap=round, color=arccolour] (12,0) -- (12,5);
\draw[line width=.2pt, line cap=round] (1,5) to[in=100, out=271] (1.3,0);
\draw[line width=.2pt, line cap=round] (2,5) to[in=120, out=273] (2.95,0);
\draw[line width=.2pt, line cap=round] (3,5) to[in=180, out=275] (6,0);
\draw[line width=.2pt, line cap=round] (9,5) to[in=0, out=265] (6,0);
\draw[line width=.2pt, line cap=round] (4,5) to[bend right=82, looseness=2.3] (8,5);
\draw[line width=.2pt, line cap=round] (5,5) to[bend right=68, looseness=1.4] (7,5);
\draw[line width=.2pt, line cap=round] (10,5) to[in=60, out=267] (9.05,0);
\draw[line width=.2pt, line cap=round] (11,5) to[in=80, out=269] (10.7,0);
\end{scope}
\draw[line width=.5pt, line cap=round] (-1.5,0) -- (13.5,0);
\node[font=\scriptsize, overlay] at (6,-1.95) {$\mathrm w_\eta (p) = -1$};
\end{scope}
\end{tikzpicture}
\]
for the illustrated three line fields. (Note that with this convention, the boundary of a closed disk has winding number $-2$.)

The winding numbers $\theta_1, \dotsc, \theta_n$ along the boundary or circle segments in an $n$-gon $P_v$ satisfy the topological constraint
\[
\theta_1 + \dotsb + \theta_n = n - 2
\]
deriving from the Poincaré--Hopf index formula. 
The degrees of the (boundary or orbifold) paths are then given by the winding numbers of the paths along $\eta$, i.e.\
$|p| = \mathrm w_\eta (p)$.

\subsection{Explicit A$_\infty$ categories from admissible dissections}
\label{section:ainfinity}

In this section we define the partially wrapped Fukaya category of a graded orbifold disk with stops in terms of any admissible dissection $\Delta$. That is, given  an admissible dissection $\Delta$ (see Definition \ref{definition:orbifoldstop}) we now construct an explicit $\Bbbk$-linear A$_\infty$ category $\mathbf A_\Delta$ whose associated category $\tw (\mathbf A_\Delta)^\natural$ is equivalent to $\W (\mathbf S)$. Up to Morita equivalence, the resulting category $\mathbf A_\Delta$ is independent of $\Delta$ \cite[Theorem 6.10]{barmeierschrollwang}.

The explicit nature of the higher products on $\mathbf A_\Delta$ allows us in Section \ref{section:formal} to give combinatorial conditions on $\Delta$ such that $\mathbf A_\Delta$ is a {\it formal} A$_\infty$ category, i.e.\ A$_\infty$-quasi-equivalent to its cohomology (Definition \ref{definition:formal}). This gives rise to explicit derived equivalences between associative algebras arising from what we call {\it formal dissections}.

\subsubsection{Objects and morphisms}
\label{subsection:objectsandmorphisms}

Let $\mathbf S = (S, \Sigma, \eta)$ be a graded orbifold surface with stops and let $\Delta$ be an admissible dissection. The A$_\infty$ category $\mathbf A_\Delta$ has objects given by the arcs in $\Delta$.

Given two arcs $\gamma$ and $\gamma'$ in $\Delta$, a $\Bbbk$-linear basis of morphisms from $\gamma$ to $\gamma'$ is given by (boundary and orbifold) paths from $\gamma$ to $\gamma'$ together with the identity morphism (denoted by $\id_{\gamma}$) in case $\gamma = \gamma'$. 

\subsubsection{Composition and higher products}
\label{subsection:higher}

We define three types of A$_\infty$ products on $\mathbf A_\Delta$ denoted by
\begin{itemize}
\item $\mubar_2$ for the (associative) concatenation of paths
\item $\mucirc_n$ for $n \geq 2$, where $\circ$ stands for a smooth disk sequence
\item $\mutimes_{n-1}$ for $n \geq 2$, where $\times$ stands for an orbifold disk sequence.
\end{itemize}


Let $p, q$ be two paths defining morphisms from $\gamma$ to $\gamma'$ and from $\gamma'$ to $\gamma''$, respectively. If $p$ and $q$ can be concatenated nontrivially, because they are both consecutive boundary paths or consecutive angles around an orbifold point, we define $$\mubar_2 (\s q \otimes \s p) = (-1)^{|p|} \s q p$$ where $q p$ denotes the concatenation of $p$ and $q$. In particular, we have
\[
(-1)^{|p|} \mubar_2 (\s \, \id_{\gamma'} \otimes \s p) = \s p = \mubar_2 (\s p \otimes \s \, \id_{\gamma}).
\]


In addition to the above composition, the remaining A$_\infty$ products are obtained from {\it smooth disk sequences} and {\it orbifold disk sequences}. (This is ensured by the precise condition on admissibility. For orbifold surfaces, there are further possibilities of non-admissible dissections which give rise to extra higher products.) The former were already introduced in \cite{haidenkatzarkovkontsevich} (under the name {\it disk sequence}) to define $n$-ary higher operations for $n \geq 2$, which we denote here by $\mucirc_{n}$. The latter higher products, denoted $\mutimes_{n}$ have arity $\geq 1$.

\begin{definition}
\label{definition:disksequence}
Let $\Delta$ be  an admissible dissection of $\mathbf S = (S, \Sigma, \eta)$. A {\it smooth disk sequence} of length $n$ is an $n$-gon in $S$ cut out by a subset of the arcs in $\Delta$ containing neither orbifold points or orbifold stops in its interior nor any boundary stops in its boundary. We may denote a smooth disk sequence of length $n$ as 
\[
\gamma_0 \overset{p_1}{\frown} \gamma_1 \overset{p_2}{\frown} \dotsb\overset{p_{n-1}}{\frown} \gamma_{n-1} \overset{p_{n}}{\frown} \gamma_0
\]
where $\gamma_0, p_1, \gamma_1, p_2, \dotsc, \gamma_{n-1}, p_{n}$ are the consecutive arcs and paths in the boundary of the $n$-gon. In particular, they are cyclic and the boundary paths compose to zero and thus define a quiver of type $\widetilde{\mathrm A}_{n-1}$. 

An {\it orbifold disk sequence} of length $n$ is an $n$-gon in $S$  cut out by a subset of the arcs in $\Delta$ containing exactly one orbifold stop and containing neither boundary stops nor orbifold points. We may denote an orbifold disk sequence of length $n$ as 
\[
\gamma_0 \overset{p_1}{\smile} \gamma_1 \overset{p_2}{\smile} \dotsb\overset{p_{n-1}}{\smile} \gamma_{n-1} 
\]
where $\gamma_0, p_1, \gamma_1, p_2, \dotsc, \gamma_{n-2}, p_{n-1}, \gamma_{n-1}$ are the consecutive arcs and paths in the boundary of the $n$-gon such that the orbifold stop lies between $\gamma_0$ and $\gamma_{n-1}$.
\end{definition}

Let $\gamma_0 \overset{p_1}{\frown} \gamma_1 \overset{p_2}{\frown} \dotsb\overset{p_{n-1}}{\frown} \gamma_{n-1} \overset{p_{n}}{\frown} \gamma_0$ be a smooth disk sequence of length $n$. Then we define
\begin{align}
\label{align:smoothdiskproduct}
\mucirc_n (\s p_i \otimes \dotsb \otimes \s p_1 \otimes \s p_n \otimes \dotsb \otimes \s p_{i+1}) = \s \, \id_{\gamma_i}
\end{align}
for all $0 \leq i < n$. Note that $\mucirc_n$ is precisely the higher product in the building block \eqref{eq:surfacebuildingblocksnostop} in Remark \ref{remark:homologicalsmoothness}. (Note that although \eqref{align:smoothdiskproduct} is the only higher structure on the building block \eqref{eq:surfacebuildingblocksnostop}, in a general surface, the gluing produces further higher structures, namely forcing linearity with respect to the composition $\mubar_2$ in the first and last entries of $\mucirc_n$, see \cite[Remark 7.4]{barmeierschrollwang}.)

Let $\gamma_0 \overset{p_1}{\smile} \gamma_1 \overset{p_2}{\smile} \dotsb\overset{p_{n-1}}{\smile} \gamma_{n-1} $ be an orbifold disk sequence of length $n$. Let $q$ be the orbifold path from $\gamma_0$ to $\gamma_{n-1}$ at the orbifold point $0$. Denote by $r$ the unique maximal orbifold path at $0$ so that $r = q'' q q'$, where $q', q''$ are orbifold subpaths (possibly trivial) of $p$. Then 
\begin{align}\label{align:orbifolddiskproduct}
\mutimes_{n-1} ( \s p_{n-1} \otimes \dotsb\otimes \s p_1) = (-1)^{|q''|} \s q.
\end{align}

\begin{figure}
\centering
\begin{tikzpicture}[x=1em,y=1em]
\begin{scope}[decoration={markings,mark=at position 0.63 with {\arrow[black]{Stealth[length=4.8pt]}}}]
\draw[line width=.5pt,postaction={decorate}] (1,0) -- (-1,0);
\draw[line width=.5pt,postaction={decorate}] (-1,-4) -- (1,-4);
\draw[line width=.75pt, color=arccolour, line cap=round] (-1,0) -- (-1,-4);
\draw[line width=.75pt, color=arccolour, line cap=round] (1,0) -- (1,-4);
\node[font=\scriptsize,left=-.3ex] at (-1.2,-2) {$\gamma_0$};
\node[font=\scriptsize,right=-.3ex] at (1.2,-2) {$\gamma_1$};
\node[font=\scriptsize, overlay] at (0,-4.6) {$p_1$};
\node[font=\scriptsize] at (0,.6) {$p_2$};
\node[font=\small] at (3.5,-2) {$\leftrightarrow $};
\node[font=\small] at (9.5,-1){$\mucirc_2(\s p_2 \otimes \s p_1) = \s \, \id_{\gamma_0}$};
\node[font=\small] at (9.5,-3){$\mucirc_2(\s p_1 \otimes \s p_2) = \s \, \id_{\gamma_1}$};
\end{scope}
\begin{scope}[xshift=-8em,decoration={markings,mark=at position 0.55 with {\arrow[black]{Stealth[length=4.8pt]}}}]
\draw[line width=.5pt,postaction={decorate}] ($(252:4.5)+(-1,0)$) -- ($(288:4.5)+(1,0)$);
\node[font=\scriptsize] at (0,0) {$\times$};
\draw[->, line width=.5pt] (249:1.2em) arc[start angle=249, end angle=-69, radius=1.2em];
\node[font=\scriptsize,shape=circle,scale=.6] (X) at (0,0) {};
\draw[line width=.75pt, color=arccolour, line cap=round] (X) -- (252:4.5);
\draw[line width=.75pt, color=arccolour, line cap=round] (X) -- (288:4.5);
\draw[line width=.75pt, color=arccolour, line cap=round] (X) -- (190:2em);
\draw[dash pattern=on 0pt off 1.3pt, line width=.8pt, line cap=round, color=arccolour] (190:2.1em) -- (190:2.5em);
\draw[line width=.75pt, color=arccolour, line cap=round] (X) -- (128:2em);
\draw[dash pattern=on 0pt off 1.3pt, line width=.8pt, line cap=round, color=arccolour, overlay] (128:2.1em) -- (128:2.5em);
\draw[line width=.75pt, color=arccolour, line cap=round] (X) -- (350:2em);
\draw[dash pattern=on 0pt off 1.3pt, line width=.8pt, line cap=round, color=arccolour] (350:2.1em) -- (350:2.5em);
\node[font=\scriptsize] at (79:.7) {.};
\node[font=\scriptsize] at (59:.7) {.};
\node[font=\scriptsize] at (39:.7) {.};
\node[font=\scriptsize,left=-.3ex] at (252:3) {$\gamma_0$};
\node[font=\scriptsize,right=-.3ex] at (288:3) {$\gamma_1$};
\node[font=\scriptsize,color=stopcolour] at (270:1em) {$\bullet$};
\node[font=\scriptsize, overlay] at (270:5) {$p_1$};
\node[font=\scriptsize, overlay] at (90:1.7) {$q$};
\node[font=\small] at (-7,-2) {$\mutimes_1 (\s p_1) = \s q \quad \leftrightarrow$};
\end{scope}
\end{tikzpicture}
\caption{An orbifold disk sequence of length $2$ contributing to the differential (left) and a smooth disk sequence of length $2$, where the paths $p_1$ and $p_2$ can be boundary paths or orbifold paths (right).}
\label{fig:differential}
\end{figure}
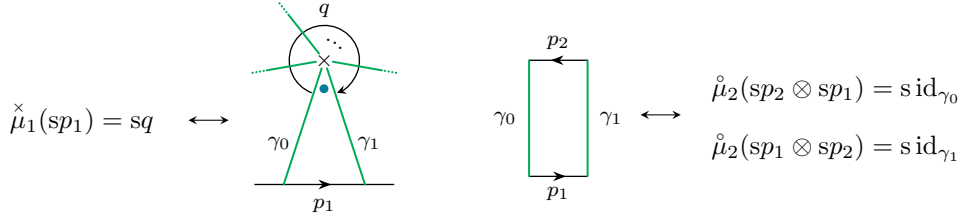

The main results on the admissible dissections we obtain in \cite{barmeierschrollwang} can be summarized as follows. 
\begin{theorem}\label{theorem:moritaequivalenceadmissible}
\begin{enumerate}
\item Let $\Delta$ be an admissible dissection of $\mathbf S = (S, \Sigma, \eta)$. Then the operations $\mubar_2$, $\mucirc_{\geq 2}$ and $\mutimes_{\geq 1}$ equip $\mathbf A_\Delta$ with the structure of a strictly unital A$_\infty$ category. 
\item Let $\mathbf S = (S, \Sigma, \eta)$ be a graded orbifold surface with stops. If $\Delta$ and $\Delta'$ are two admissible dissections of $\mathbf S$, then $\mathbf A_{\Delta}$ and $\mathbf A_{\Delta'}$ are Morita equivalent. In this way, we define the partially wrapped Fukaya category $\W (\mathbf S)$ of $\mathbf S$ as the category $\tw (\mathbf A_\Delta)^\natural$ for any admissible dissection $\Delta$. 
\end{enumerate}
\end{theorem}

\begin{remark}\label{remark:withck}
Cho and Kim \cite{chokim} use a similar but slightly different approach to construct an A$_\infty$ category $\mathcal F_\Gamma$ associated to a different kind of dissection $\Gamma$ (called {\it tagged arc system}). This category $\mathcal F_\Gamma$ turns out to be Morita equivalent to the category $\mathbf A_\Delta$ defined above. In contrast to $\mathbf A_\Delta$, the category $\mathcal F_\Gamma$ has trivial differential, since in \cite{chokim}  no orbifold disk sequences of length $2$ exist. Another main difference is that arcs connecting to orbifold points come in (isotopic) copies distinguished by a tagging in \cite{chokim} which give rise to nonisomorphic objects in $\mathcal F_\Gamma$. These nonisomorphic copies can be expressed as twisted complexes in $\mathbf A_\Delta$ in our setup (see Remark \ref{remark:tagging}). 
\end{remark}

\begin{remark}
\label{remark:tagging}
Let $\Delta$ be an admissible dissection and let $\mathbf A_\Delta$ be the associated A$_\infty$ category. To any arc $\gamma$ connecting to the orbifold point $0$ in $\Delta$ we may associate an object $\gamma'$ in $\tw (\mathbf A_\Delta)$ such that it is not isomorphic to $\gamma$ but is represented by an isotopic arc. 
\end{remark}

\newcommand{\flipdiagram}{%
\draw[line width=.5pt] (-4em,0) ++( 60:1.5em) arc[start angle= 60, end angle=-15, radius=1.5em];
\draw[line width=.5pt]  (4em,0) ++(195:1.5em) arc[start angle=195, end angle=120, radius=1.5em];
\node[font=\scriptsize,shape=circle,scale=.6] (X) at (0,4em) {};
\node[font=\scriptsize] at (0,4em) {$\times$};
\node[font=\scriptsize,color=stopcolour] at ($(0,4em)+(270:.65em)$) {$\bullet$};
\draw[line width=.75pt, color=arccolour, line cap=round] (-4em,0) ++( 45:1.5em) -- (X);
\draw[line width=.75pt, color=arccolour, line cap=round] ( 4em,0) ++(135:1.5em) -- (X);
\draw[line width=.75pt, color=arccolour, line cap=round] (X) -- ++(165:2em);
\draw[dash pattern=on 0pt off 1.3pt, line width=.8pt, line cap=round, color=arccolour] (X) ++(165:2.1em) -- ++(165:.4em);
\draw[line width=.75pt, color=arccolour, line cap=round] (X) -- ++(105:2em);
\draw[dash pattern=on 0pt off 1.3pt, line width=.8pt, line cap=round, color=arccolour] (X) ++(105:2.1em) -- ++(105:.4em);
\draw[line width=.75pt, color=arccolour, line cap=round] (X) -- ++(15:2em);
\draw[dash pattern=on 0pt off 1.3pt, line width=.8pt, line cap=round, color=arccolour] (X) ++(15:2.1em) -- ++(15:.4em);
\node[font=\scriptsize] at ($(0,4em)+(60:1em)$) {.};
\node[font=\scriptsize] at ($(0,4em)+(75:1em)$) {.};
\node[font=\scriptsize] at ($(0,4em)+(45:1em)$) {.};
}

\section{Fukaya categories of orbifold disks via orbit categories}
\label{section:disk}

In this section we show that the partially wrapped Fukaya category of an orbifold disk $\mathbf D_n^\times$ with stops and with one orbifold point is triangle equivalent to   the $\mathbb Z_2$-orbit category of the partially wrapped Fukaya category of a smooth disk with stops ($=$ the double cover of $\mathbf D_n^\times$). The orbifold disk is a local model for the neighbourhood of an orbifold point in a general orbifold surface and the orbifold disk with stops can be viewed as a new type of ``sector'' for partially wrapped Fukaya categories of general orbifold surfaces.

\subsection{An orbifold disk with stops and its double cover}
\label{subsection:orbifolddisk}

Consider the graded orbifold disk $\mathbf D^\times_n = (\mathbb D / \mathbb Z_2, \Sigma, \eta)$ with a single orbifold point $0 \in \mathbb D / \mathbb Z_2$ and $n$ boundary stops and its smooth double cover $\mathbf D_{2n}$ as in Section \ref{subsection:orbifolddisk} (see Fig.~\ref{fig:disk}).

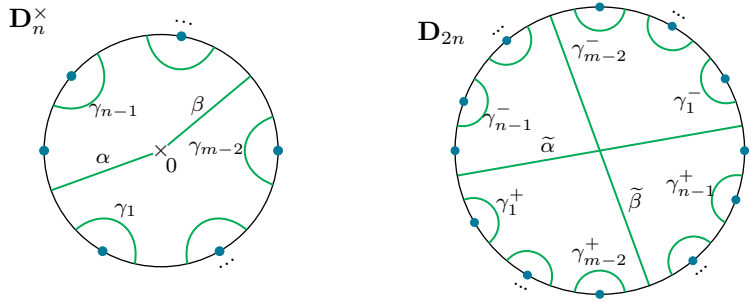
\begin{figure}[t]
\centering
\begin{tikzpicture}[x=1em,y=1em,decoration={markings,mark=at position 0.55 with {\arrow[black]{Stealth[length=4.2pt]}}}]
\draw[line width=.5pt] circle(4em);
\node[font=\small] at (-4.5em,4.5em) {$\mathbf D_n^{\times}$};
\node[font=\scriptsize,shape=circle,scale=.6,fill=white] (X) at (0,0) {};
\node[font=\scriptsize] at (0,0) {$\times$};
\draw[line width=.75pt,color=arccolour] (200:4em) to (X);
\draw[line width=.75pt,color=arccolour] (40:4em) to (X);
\node[font=\scriptsize] at (303:4.5em) {$.$};
\node[font=\scriptsize] at (297:4.5em) {$.$};
\node[font=\scriptsize] at (300:4.5em) {$.$};
\node[font=\scriptsize] at (77:4.5em) {$.$};
\node[font=\scriptsize] at (80:4.5em) {$.$};
\node[font=\scriptsize] at (83:4.5em) {$.$};
\node[font=\scriptsize] at (306:0.6em) {$0$};
\node[font=\scriptsize] at (188:2em) {$\alpha$};
\node[font=\scriptsize] at (51:2em) {$\beta$};
\node[font=\scriptsize] at (240:2.4em) {$\gamma_1$};
\node[font=\scriptsize] at (2:1.9em) {$\gamma_{m-2}$};
\node[font=\scriptsize] at (140:2.1em) {$\gamma_{n-1}$};
\foreach \a in {240,300,0,80,140} {
\draw[fill=stopcolour, color=stopcolour] (\a:4em) circle(.15em);
\path[line width=.75pt,out=\a-160,in=\a+160,looseness=1.5,color=arccolour] (\a+17:4em) edge (\a-17:4em);
}
\draw[fill=stopcolour, color=stopcolour] (180:4em) circle(.15em);
\begin{scope}[xshift=15em,scale=.9]
\node[font=\small] at (-6em,4.5em) {$\mathbf D_{2n}$};
\draw[line width=.75pt, color=arccolour] (190:5.5em) to (10:5.5em);
\draw[line width=.75pt, color=arccolour] (290:5.5em) to (110:5.5em);
\draw[line width=.5pt] circle(5.5em);
\foreach \a in {210,240,270,310,340,30,60,90,130,160} {
\draw[fill=stopcolour, color=stopcolour] (\a:5.5em) circle(.15em);
\path[line width=.75pt,out=\a-170,in=\a+170,looseness=1.5,color=arccolour] (\a+10:5.5em) edge (\a-10:5.5em);
}
\draw[fill=stopcolour, color=stopcolour] (180:5.5em) circle(.15em);
\draw[fill=stopcolour, color=stopcolour] (0:5.5em) circle(.15em);
%
\node[font=\scriptsize] at (238:6em) {$.$};
\node[font=\scriptsize] at (240:6em) {$.$};
\node[font=\scriptsize] at (242:6em) {$.$};
\node[font=\scriptsize] at (308:6em) {$.$};
\node[font=\scriptsize] at (310:6em) {$.$};
\node[font=\scriptsize] at (312:6em) {$.$};
%
\node[font=\scriptsize] at (58:6em) {$.$};
\node[font=\scriptsize] at (60:6em) {$.$};
\node[font=\scriptsize] at (62:6em) {$.$};
\node[font=\scriptsize] at (128:6em) {$.$};
\node[font=\scriptsize] at (130:6em) {$.$};
\node[font=\scriptsize] at (132:6em) {$.$};
%
\node[font=\scriptsize] at (210:3.9em) {$\gamma_1^+$};
\node[font=\scriptsize] at (270:3.9em) {$\gamma_{m-2}^+$};
\node[font=\scriptsize] at (340:3.7em) {$\gamma_{n-1}^+$};
\node[font=\scriptsize] at (30:3.9em) {$\gamma_1^-$};
\node[font=\scriptsize] at (90:3.9em) {$\gamma_{m-2}^-$};
\node[font=\scriptsize] at (160:3.7em) {$\gamma_{n-1}^-$};
\node[font=\scriptsize, above=-.2ex] at (190:2em) {$\widetilde\alpha$};
\node[font=\scriptsize, right] at (110:-2em) {$\widetilde\beta$};
\end{scope}
\end{tikzpicture}
\caption{Arcs on an orbifold disk with $n$ stops (left) and the corresponding $\mathbb Z_2$-invariant (pairs of) arcs in the double cover (right).}
\label{fig:doublecover}
\end{figure}

Consider the collection of arcs $\Gamma = \{ \alpha, \beta, \gamma_1, \dotsc, \gamma_{n-1} \}$ on $\mathbf D_n^\times$ illustrated in Fig.~\ref{fig:doublecover}.
Then each $\gamma_i$ lifts to a pair of arcs $\gamma_i^+, \gamma_i^-$ in the double cover $\mathbf D_{2n}$. The arcs $\alpha, \beta$ lift to two $\mathbb Z_2$-invariant arcs $\widetilde\alpha, \widetilde\beta$ which intersect in the fixed point of the $\mathbb Z_2$-action.

\subsubsection{Twisted complexes and morphisms in the double cover}

On graded smooth surfaces, the dissections considered in \cite{haidenkatzarkovkontsevich} are by definition comprised of {\it pairwise disjoint} arcs. Since $\widetilde\alpha$ and $\widetilde\beta$ intersect in the double cover, we consider the dissection
\[
\widetilde \Gamma = \{ \widetilde\alpha, \gamma_1^\pm, \dotsc, \gamma_{n-1}^\pm \}
\]
of the double cover $\mathbf D_{2n}$ without $\widetilde \beta$.

\begin{notation}
\label{notation:smoothdisk}
We denote by $\mathbf A_{\widetilde \Gamma}$ the graded gentle algebra, viewed as a $\Bbbk$-linear graded category, associated to the dissection $\widetilde \Gamma$ as in \cite{haidenkatzarkovkontsevich, lekilipolishchuk2}. Then
\[
\tw (\mathbf A_{\widetilde \Gamma})^\natural \simeq \W (\mathbf D_{2n})
\]
is the partially wrapped Fukaya category of the smooth disk $\mathbf D_{2n}$, where $\tw (\mathbf A_{\widetilde \Gamma})$ is the DG category of twisted complexes over $\mathbf A_{\widetilde \Gamma}$.
\end{notation}

There are morphisms $p_1^\pm \in \mathbf A_{\widetilde \Gamma} (\widetilde \alpha, \gamma_1^\pm)$ and $p_i^\pm \in \mathbf A_{\widetilde \Gamma} (\gamma_{i-1}^\pm, \gamma_i^\pm)$ for $1 < i \leq n - 1$ given by boundary paths in $\mathbf D_{2n}$.

The lift $\widetilde\beta$ may be realized as the twisted complex
\begin{equation}
\label{eq:beta}
\widetilde\beta = \widetilde\alpha \oplus \bigoplus_{\substack{1 \leq i \leq m-2 \\ \ast \in \{ +, - \}}} \s^{\maltese^*_i} \gamma_i^*.
\end{equation}
Here, $\maltese_i^* =  |p^*_1|+ \dotsb + |p^*_i| - i$ for $\ast \in \{ +, - \}$. The twisted differential $\delta_{\widetilde \beta}$ is given by the maps appearing in the diagram
\begin{equation}
\label{eq:betamaps}
\begin{tikzpicture}[baseline=-2.6pt,description/.style={fill=white,inner sep=1.75pt}]
\matrix (m) [matrix of math nodes, row sep={1.2em,between origins}, text height=1.5ex, column sep={5em,between origins}, text depth=0.25ex, ampersand replacement=\&]
{
\&[.4em] \gamma_1^+ \& \gamma_{2}^+ \& \cdots \& \gamma_{m-2}^+ \\
\widetilde\alpha \& \& \& \& \\
\&       \gamma_1^- \& \gamma_{2}^- \& \cdots \& \gamma_{m-2}^- \\
};
\path[->,line width=.4pt]
(m-2-1) edge node[above=-.2ex,font=\scriptsize, overlay] {$p_1^+$} (m-1-2)
(m-2-1) edge node[below=-.2ex,font=\scriptsize, overlay] {$p_1^-$} (m-3-2)
(m-1-2) edge node[above=-.2ex,font=\scriptsize, overlay] {$p_{2}^+$} (m-1-3)
(m-3-2) edge node[below=-.3ex,font=\scriptsize, overlay] {$p_{2}^-$} (m-3-3)
(m-1-3) edge node[above=-.2ex,font=\scriptsize, overlay] {$p_{3}^+$} (m-1-4)
(m-3-3) edge node[below=-.3ex,font=\scriptsize, overlay] {$p_{3}^-$} (m-3-4)
(m-1-4) edge node[above=-.2ex,font=\scriptsize, overlay] {$p_{m-2}^+$} (m-1-5)
(m-3-4) edge node[below=-.3ex,font=\scriptsize, overlay] {$p_{m-2}^-$} (m-3-5)
;
\end{tikzpicture}
\end{equation}
where for simplicity we have omitted the extra shifts, cf.\ \eqref{eq:beta}.

Note that the twisted complex $\widetilde\beta$ is $\mathbb Z_2$-invariant since $-1 \cdot \gamma_i^\pm = \gamma_i^\mp, \ -1 \cdot p_i^\pm = p_i^\mp$ and $-1 \cdot \widetilde\alpha = \widetilde\alpha$. The $\mathbb Z_2$-invariance of $\widetilde \beta$ is also apparent from its graphical representation in Fig.~\ref{fig:doublecover}.

\subsection{The partially wrapped Fukaya category of an orbifold disk}
\label{subsection:partiallyorbifolddisk}

Let $\mathbf D_n^\times$ be the graded orbifold disk as in Section \ref{subsection:orbifolddisk} and let $\mathbf A_{\widetilde \Gamma}$ be the graded gentle algebra in Notation \ref{notation:smoothdisk}. The $\mathbb Z_2$-action on $\mathbf A_{\widetilde\Gamma}$ induces a $\mathbb Z_2$-action on the DG category $\tw (\mathbf A_{\widetilde\Gamma})$. We now denote the idempotent completion of the A$_\infty$ orbit category by
\[
\W (\mathbf D_n^\times) := (\tw (\mathbf A_{\widetilde\Gamma})/\mathbb Z_2)^\natural \simeq (\W (\mathbf D_{2n}) / \mathbb Z_2)^\natural.
\]

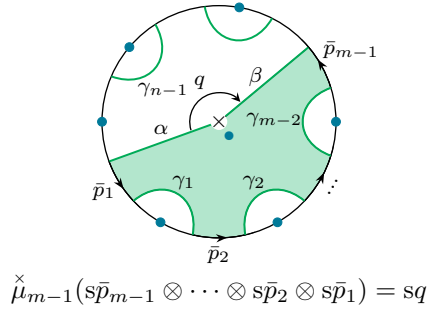
\begin{figure}
\centering
\begin{tikzpicture}[x=1em,y=1em,decoration={markings,mark=at position 0.55 with {\arrow[black]{Stealth[length=4.2pt]}}}]
\draw[fill=arccolour!30!white,line width=0pt] (200:4em) arc[start angle=200, end angle=240-17, radius=4em] to[out=240+160, in=240-160, looseness=1.5] (240+17:4em) arc[start angle=240+17, end angle=300-17, radius=4em] to[out=300+160, in=300-160, looseness=1.5] (300+17:4em) arc[start angle=300+17, end angle=360-17, radius=4em] to[out=360+160, in=360-160, looseness=1.5] (360+17:4em) arc[start angle=360+17, end angle=400, radius=4em] to (0,0);
\draw[->, line width=.5pt] (197:1em) arc[start angle=197, end angle=43, radius=1em];
\node[font=\scriptsize] at (116:1.5em) {$q$};
\draw[line width=.5pt] circle(4em);
\node[font=\scriptsize,shape=circle,scale=.6,fill=white] (X) at (0,0) {};
\node[font=\scriptsize] at (0,0) {$\times$};
\draw[line width=.75pt,color=arccolour] (200:4em) to (X);
\draw[line width=.75pt,color=arccolour] (40:4em) to (X);
\node[font=\scriptsize, color=stopcolour] at (306:0.6em) {$\bullet$};
\node[font=\scriptsize] at (188:2em) {$\alpha$};
\node[font=\scriptsize] at (51:2em) {$\beta$};
\node[font=\scriptsize] at (240:2.4em) {$\gamma_1$};
\node[font=\scriptsize] at (300:2.4em) {$\gamma_2$};
\node[font=\scriptsize] at (2:1.9em) {$\gamma_{m-2}$};
\node[font=\scriptsize] at (150:2.2em) {$\gamma_{n-1}$};
\draw[line width=0pt,postaction={decorate}] (200:4em) arc[start angle=200, end angle=240-17+5, radius=4em];
\node[font=\scriptsize] at (211.5:4.6em) {$\bar p_1$};
\draw[line width=0pt,postaction={decorate}] (240+17:4em) arc[start angle=240+17, end angle=300-17+5, radius=4em];
\node[font=\scriptsize] at (270:4.6em) {$\bar p_2$};
\draw[line width=0pt,postaction={decorate}] (300+17:4em) arc[start angle=300+17, end angle=360-17+5, radius=4em];
\node[font=\scriptsize] at (330:4.5em) {$.$};
\node[font=\scriptsize] at (327:4.5em) {$.$};
\node[font=\scriptsize] at (333:4.5em) {$.$};
\draw[line width=0pt,postaction={decorate}] (17:4em) arc[start angle=17, end angle=40+5, radius=4em];
\node[font=\scriptsize] at (28.5:5.15em) {$\bar p_{m-1}$};
\foreach \a in {240,300,0,80,140} {
\draw[fill=stopcolour, color=stopcolour] (\a:4em) circle(.15em);
\path[line width=.75pt,out=\a-160,in=\a+160,looseness=1.5,color=arccolour] (\a+17:4em) edge (\a-17:4em);
}
\draw[fill=stopcolour, color=stopcolour] (180:4em) circle(.15em);
\node[font=\small] at (0,-5.75em) {$\mutimes_{m-1} (\s \bar p_{m-1} \otimes \dotsb \otimes \s \bar p_2 \otimes \s \bar p_1) = \s q$};
\end{tikzpicture}
\caption{The unique higher product on an orbifold disk where the ``orbifold stop'' near the orbifold point indicates that there is no morphism from $\beta$ to $\alpha$.}
\label{fig:doublecoverdisk}
\end{figure}

Below, let us analyse some morphism spaces in $\tw (\mathbf A_{\widetilde\Gamma})/\mathbb Z_2$. For this, we write the morphisms spaces $(\tw (\mathbf A_{\widetilde\Gamma})/\mathbb Z_2) (X, Y)$ as column vectors, where the top entry corresponds to a morphism $X \to Y$ and the bottom entry to a morphism $X \to -1 \cdot Y$. The composition of morphisms (cf.\ Section \ref{subsection:orbitcategories}) is given by
\begin{align}
\label{align:compositionorbit}
\begin{pmatrix} a_1 \\ a_2 \end{pmatrix} \circ \begin{pmatrix} b_1 \\ b_2 \end{pmatrix} = \begin{pmatrix} a_1\circ b_1 + (-1\cdot a_2) \circ b_2 \\ a_2 \circ b_1 +  (-1\cdot a_1) \circ b_2 \end{pmatrix}
\end{align}
where $\circ$ on the right hand side of the equality is the composition in $\tw (\mathbf A_{\widetilde\Gamma})$.

Since $-1 \cdot \widetilde \alpha = \widetilde \alpha$ it follows that 
\[
(\tw (\mathbf A_{\widetilde\Gamma})/\mathbb Z_2) (\widetilde \alpha, \widetilde \alpha) := \bigoplus_{g \in \mathbb Z_2} \tw (\mathbf A_{\widetilde\Gamma}) (\widetilde \alpha, g \cdot \widetilde \alpha) = \tw (\mathbf A_{\widetilde\Gamma})(\widetilde \alpha, \widetilde \alpha)^{\oplus 2}
\]
is $2$-dimensional with trivial differential. We can therefore find two primitive orthogonal idempotents in $(\tw (\mathbf A_{\widetilde\Gamma})/\mathbb Z_2) (\widetilde \alpha, \widetilde \alpha)$ 
\[
\id_{\widetilde \alpha}^+ =  \frac{1}{2} \begin{pmatrix} \id_{\widetilde \alpha} \\ \id_{\widetilde \alpha} \end{pmatrix} \quad \text{and}  \quad  \id_{\widetilde \alpha}^- = \frac{1}{2} \begin{pmatrix} \id_{\widetilde \alpha}\\ -\id_{\widetilde \alpha} \end{pmatrix}.
\]
Similarly, since $-1 \cdot \widetilde \beta = \widetilde \beta$ we have $(\tw (\mathbf A_{\widetilde\Gamma})/\mathbb Z_2) (\widetilde \beta, \widetilde \beta) = \tw (\mathbf A_{\widetilde\Gamma})(\widetilde \beta, \widetilde \beta)^{\oplus 2}$ with two primitive orthogonal idempotents
\[
\id_{\widetilde \beta}^+ = \frac{1}{2} \begin{pmatrix} \id_{\widetilde \beta} \\ \id_{\widetilde \beta}\end{pmatrix} \quad \text{and}  \quad  \id_{\widetilde \beta}^- = \frac{1}{2} \begin{pmatrix} \id_{\widetilde \beta} \\ -\id_{\widetilde \beta}\end{pmatrix}.
\]
Here $\id_{\widetilde \beta}$ is the identity on the twisted complex $\widetilde \beta$ \eqref{eq:beta}.

The existence of primitive orthogonal idempotents implies that both $\widetilde \alpha$ and $\widetilde \beta$ decompose into two direct summands in the idempotent completion $(\tw (\mathbf A_{\widetilde\Gamma})/\mathbb Z_2)^\natural$ and we shall write
$$
\widetilde \alpha = \widetilde \alpha^+\oplus \widetilde \alpha^- \quad \text{and} \quad \widetilde \beta = \widetilde \beta^+ \oplus \widetilde \beta^-
$$
where
\begin{align*}
(\tw (\mathbf A_{\widetilde\Gamma})/\mathbb Z_2)^\natural (\widetilde \alpha^\pm, \widetilde \alpha^\pm) &= \Bbbk \, \id_{\widetilde \alpha}^\pm \\
(\tw (\mathbf A_{\widetilde\Gamma})/\mathbb Z_2)^\natural (\widetilde \beta^\pm, \widetilde \beta^\pm) &= \id_{\widetilde \beta}^\pm \, (\tw (\mathbf A_{\widetilde\Gamma})/\mathbb Z_2) (\widetilde \beta, \widetilde \beta) \, \id_{\widetilde \beta}^\pm.
\end{align*}
Let us consider the following direct sum of objects in the DG category $\W (\mathbf D_n^\times)$
\begin{equation}
\label{eq:generator}
\Gamma = \widetilde \alpha^+ \oplus \bigoplus_{i=1}^l \gamma_i^+\oplus \widetilde \beta^-.
\end{equation}

Then we have the following results. 
\begin{theorem}[\cite{barmeierschrollwang}]\label{theorem:splitgenerator}
\begin{enumerate}
\item The object $\Gamma$ \eqref{eq:generator} is a generator of the partially wrapped Fukaya category $\W (\mathbf D_n^\times) = (\tw (\mathbf A_{\widetilde\Gamma})/\mathbb Z_2)^\natural$.
\item The only nonzero higher A$_\infty$ product in the A$_\infty$ minimal model of $\End_{\W (\mathbf D_n^\times)} (\Gamma)$ is given by 
$\mutimes_{m-1}$ illustrated in Fig.~\ref{fig:doublecoverdisk}. 
\item For any admissible dissection $\Delta$ of $\mathbf D_n^\times$ the category $\W (\mathbf D_n^\times)$ is Morita equivalent to $\mathbf A_\Delta$.
\item We have an equivalence $\W (\mathbf D_n^\times) \simeq \tw (A)$ where $A = \Bbbk Q / I$ is the (ungraded) path algebra of a linearly ordered quiver of type $\mathrm D_{n+1}$ and $I$ is the ideal generated by all paths of length $2$. This induces a triangulated equivalence $\mathrm D\W (\mathbf D_n^\times) \simeq \per (A)$, where $\per (A)$ is the perfect derived category of $A$.
\end{enumerate}
\end{theorem}

\begin{remark}
Theorem \ref{theorem:splitgenerator} shows that the partially wrapped Fukaya category $\W (\mathbf D_n^\times)$ can be described by a quiver of type $\mathrm D_{n+1}$.

Although it is natural to consider the case of full quadratic monomial relations in the algebra $A = \Bbbk Q / I$, by changing the dissection of the orbifold disk one can reverse arrows in $Q$ and remove relations of $I$ without changing the derived equivalence class of $A$. This can be shown algebraically, but it is also a straightforward consequence of the geometric approach developed in \cite{barmeierschrollwang}.
\end{remark}

Theorem \ref{theorem:splitgenerator} (2) gives higher structures in terms of curves in the orbit category of a double cover. If we had worked directly on the orbifold disk, then we could have simply defined
\begin{equation}
\label{eq:mutimesunique1}
\mu_{m-1} (\s \bar p_{m-1 \dotsc 1}) = \s q
\end{equation}
as in Fig.~\ref{fig:doublecoverdisk}. This type of higher structure is a local model and can be generalized to arbitrary graded orbifold surfaces with stops.

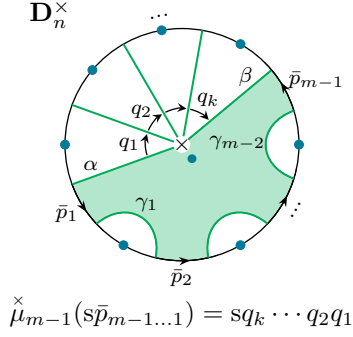
\begin{figure}
\centering
\begin{tikzpicture}[x=1em,y=1em,decoration={markings,mark=at position 0.55 with {\arrow[black]{Stealth[length=4.2pt]}}}]
\draw[fill=arccolour!30!white,line width=0pt] (200:4em) arc[start angle=200, end angle=240-17, radius=4em] to[out=240+160, in=240-160, looseness=1.5] (240+17:4em) arc[start angle=240+17, end angle=300-17, radius=4em] to[out=300+160, in=300-160, looseness=1.5] (300+17:4em) arc[start angle=300+17, end angle=360-17, radius=4em] to[out=360+160, in=360-160, looseness=1.5] (360+17:4em) arc[start angle=360+17, end angle=400, radius=4em] to (0,0);
\draw[->, line width=.5pt] (197:1.25em) arc[start angle=197, end angle=163, radius=1.25em];
\draw[->, line width=.5pt] (157:1.25em) arc[start angle=157, end angle=123, radius=1.25em];
\draw[->, line width=.5pt] (117:1.25em) arc[start angle=117, end angle=83, radius=1.25em];
\draw[->, line width=.5pt] (77:1.25em) arc[start angle=77, end angle=43, radius=1.25em];
\node[font=\scriptsize] at (180:1.8em) {$q_1$};
\node[font=\scriptsize] at (140:1.8em) {$q_2$};
\node[font=\scriptsize] at (60:1.8em) {$q_k$};
\draw[line width=.5pt] circle(4em);
\node[font=\small] at (-4.5em,4.5em) {$\mathbf D_n^\times$};
\node[font=\scriptsize,shape=circle,scale=.6,fill=white] (X) at (0,0) {};
\node[font=\scriptsize] at (0,0) {$\times$};
\draw[line width=.75pt,color=arccolour] (200:4em) to (X);
\draw[line width=.75pt,color=arccolour] (40:4em) to (X);
\draw[line width=.75pt,color=arccolour] (160:4em) to (X);
\draw[line width=.75pt,color=arccolour] (120:4em) to (X);
\draw[line width=.75pt,color=arccolour] (80:4em) to (X);
\node[font=\scriptsize] at (103:4.5em) {$.$};
\node[font=\scriptsize] at (100:4.5em) {$.$};
\node[font=\scriptsize] at (97:4.5em) {$.$};
\node[font=\scriptsize, color=stopcolour] at (306:0.6em) {$\bullet$};
\node[font=\scriptsize] at (193:3.2em) {$\alpha$};
\node[font=\scriptsize] at (47:3.2em) {$\beta$};
\node[font=\scriptsize] at (240:2.4em) {$\gamma_1$};
\node[font=\scriptsize] at (2:1.9em) {$\gamma_{m-2}$};
\draw[line width=0pt,postaction={decorate}] (200:4em) arc[start angle=200, end angle=240-17+5, radius=4em];
\node[font=\scriptsize] at (211.5:4.6em) {$\bar p_1$};
\draw[line width=0pt,postaction={decorate}] (240+17:4em) arc[start angle=240+17, end angle=300-17+5, radius=4em];
\node[font=\scriptsize] at (270:4.6em) {$\bar p_2$};
\draw[line width=0pt,postaction={decorate}] (300+17:4em) arc[start angle=300+17, end angle=360-17+5, radius=4em];
\node[font=\scriptsize] at (330:4.5em) {$.$};
\node[font=\scriptsize] at (327:4.5em) {$.$};
\node[font=\scriptsize] at (333:4.5em) {$.$};
\draw[line width=0pt,postaction={decorate}] (17:4em) arc[start angle=17, end angle=40+5, radius=4em];
\node[font=\scriptsize] at (25.5:5.1em) {$\bar p_{m-1}$};
\foreach \a in {240,300,0} {
\draw[fill=stopcolour, color=stopcolour] (\a:4em) circle(.15em);
\path[line width=.75pt,out=\a-160,in=\a+160,looseness=1.5,color=arccolour] (\a+17:4em) edge (\a-17:4em);
}
\draw[fill=stopcolour, color=stopcolour] (180:4em) circle(.15em);
\draw[fill=stopcolour, color=stopcolour] (140:4em) circle(.15em);
\draw[fill=stopcolour, color=stopcolour] (100:4em) circle(.15em);
\draw[fill=stopcolour, color=stopcolour] (60:4em) circle(.15em);
\node[font=\small] at (0,-5.75em) {$\mutimes_{m-1} (\s \bar p_{m-1 \dotsc 1}) = \s q_k \dotsb q_2 q_1$};
\end{tikzpicture}
\caption{The unique higher product on an orbifold disk where the orbifold path from $\alpha$ to $\beta$ factorizes as a composition of orbifold paths between arcs connecting to the orbifold point.}
\label{fig:orbifolddisk}
\end{figure}

Note that we may replace the arcs $\gamma_m, \dotsc, \gamma_{n-1}$ by arcs $\gamma\mathrlap{'}_m, \dotsc, \gamma\mathrlap{'}_{n-1}$, where $\gamma\mathrlap{'}_{m-1}$ is the cone of the morphism from $\beta$ to $\gamma_{m-1}$ and similarly $\gamma\mathrlap{'}_m$ is the cone of the morphism from $\gamma'_{m-1}$ to $\gamma_m$ and so on (see Fig.~\ref{fig:orbifolddisk}). One may check that $\alpha \toarg{q} \beta$ factorizes into $q = q_k \dotsb q_2 q_1$ and the unique higher product involving the arcs $\alpha, \beta, \gamma_1, \dotsc, \gamma_{m-2}, \gamma\mathrlap{'}_{m-1}, \dotsc, \gamma\mathrlap{'}_{n-1}$ is still given by
\begin{equation}
\label{eq:mutimesunique2}
\mu_{m-1} (\s \bar p_{{m-1} \dotsc 1}) = \s q = \s q_k \dotsb q_2 q_1.
\end{equation}

\section{A new type of dissection of an orbifold disk}\label{section:example}

Recall from Definition \ref{definition:orbifoldstop} that orbifold points must be connected by at least one arc in an admissible dissection. In this section, we consider a new type of dissection for orbifold surfaces, where some of the polygons are orbifold polygons without stops but with an interior orbifold point. For this type of dissection we may also obtain an A$_\infty$ category. We show that this A$_\infty$ category is usually non-formal except when all orbifold polgyons are $1$-gons in which case the corresponding A$_\infty$ category is (the category associated to) a skew-gentle algebra. In the following we only consider the case of the orbifold disk which serves as a local model for the new type of dissection for orbifold surface.

Let $\mathbf D_n^\times$ be an orbifold disk with one orbifold point and $n$ boundary stops. Let us consider the dissection $\widetilde \Gamma' =\{\gamma^{\pm 1}, \dotsc, \gamma^{\pm n}\}$ consisting of the $2n$ arcs on the double cover $\mathbf D_{2n}$ illustrated in Fig.~\ref{fig:doublecoveranotherarcs}. 

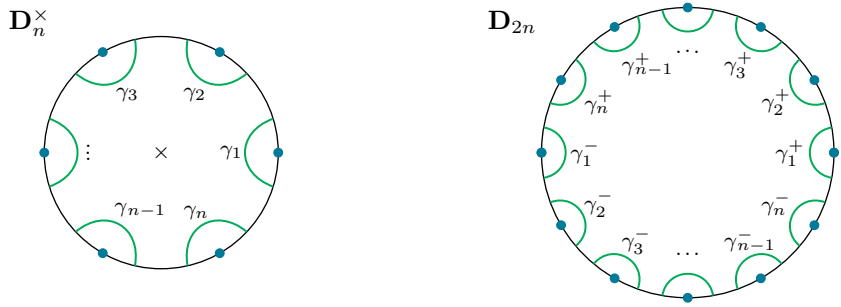
\begin{figure}
\centering
\begin{tikzpicture}[x=1em,y=1em,decoration={markings,mark=at position 0.55 with {\arrow[black]{Stealth[length=4.2pt]}}}]
\draw[line width=.5pt] circle(4em);
\node[font=\small] at (-4.5em,4.5em) {$\mathbf D_n^\times$};
\node[font=\scriptsize,shape=circle,scale=.6,fill=white] (X) at (0,0) {};
\node[font=\scriptsize] at (0,0) {$\times$};
\node[font=\scriptsize] at (175:2.5em) {$.$};
\node[font=\scriptsize] at (180:2.5em) {$.$};
\node[font=\scriptsize] at (185:2.5em) {$.$};
\node[font=\scriptsize] at (0:2.4em) {$\gamma_1$};
\node[font=\scriptsize] at (60:2.3em) {$\gamma_{2}$};
\node[font=\scriptsize] at (120:2.3em) {$\gamma_3$};
\node[font=\scriptsize] at (235:2.4em) {$\gamma\mathrlap{_{n-1}}$};
\node[font=\scriptsize] at (300:2.4em) {$\gamma_n$};

\foreach \a in {240,300,0,60,120, 180} {
\draw[fill=stopcolour, color=stopcolour] (\a:4em) circle(.15em);
\path[line width=.75pt,out=\a-160,in=\a+160,looseness=1.5,color=arccolour] (\a+17:4em) edge (\a-17:4em);
}
\begin{scope}[xshift=18em]
\node[font=\small, overlay] at (-6em,4.5em) {$\mathbf D_{2n}$};
\draw[line width=.5pt, overlay] circle(5em);
\foreach \a in {0, 30,60,90,120,150, 180, 210,240,270,300,330} {
\draw[fill=stopcolour, color=stopcolour] (\a:5em) circle(.15em);
\path[line width=.75pt,out=\a-170,in=\a+170,looseness=1.5,color=arccolour, overlay] (\a+10:5em) edge (\a-10:5em);
}
\draw[fill=stopcolour, color=stopcolour] (180:5em) circle(.15em);
\draw[fill=stopcolour, color=stopcolour] (0:5em) circle(.15em);
%
\node[font=\scriptsize] at (265:3.5em) {$.$};
\node[font=\scriptsize] at (270:3.5em) {$.$};
\node[font=\scriptsize] at (275:3.5em) {$.$};
\node[font=\scriptsize] at (85:3.5em) {$.$};
\node[font=\scriptsize] at (90:3.5em) {$.$};
\node[font=\scriptsize] at (95:3.5em) {$.$};
\node[font=\scriptsize] at (0:3.5em) {$\gamma_1^+$};
\node[font=\scriptsize] at (30:3.5em) {$\gamma_2^+$};
\node[font=\scriptsize] at (60:3.5em) {$\gamma_{3}^+$};
\node[font=\scriptsize] at (120:3.5em) {$\gamma\mathrlap{_{n-1}}^+$};
\node[font=\scriptsize] at (150:3.5em) {$\gamma_n^+$};
\node[font=\scriptsize] at (180:3.5em) {$\gamma_{1}^-$};
\node[font=\scriptsize] at (210:3.5em) {$\gamma_2^-$};
\node[font=\scriptsize] at (240:3.5em) {$\gamma_3^-$};
\node[font=\scriptsize] at (300:3.5em) {$\gamma\mathrlap{_{n-1}}^-$};
\node[font=\scriptsize] at (330:3.5em) {$\gamma_n^-$};
\end{scope}
\end{tikzpicture}
\caption{A new non-formal dissection on an orbifold disk (left) and the corresponding $\mathbb Z_2$-invariant dissection of the double cover (right).}
\label{fig:doublecoveranotherarcs}
\end{figure}

It follows from \cite{haidenkatzarkovkontsevich} that $\widetilde \Gamma'$ is a generator of the partially wrapped Fukaya category $\W (\mathbf D_{2n})$. Note that $\mathbf A_{\widetilde \Gamma'}$ is the A$_\infty$ algebra given by the building block $(\widetilde{\mathrm A}_{2n-1}, \mucirc_{2n})$ as in Remark \ref{remark:homologicalsmoothness}. 

Denoting by $p_1^+, \dotsc, p_n^+, p_1^-, \dotsc, p_n^-$ the arrows in $\widetilde{\mathrm A}_{2n-1}$, the Poincaré--Hopf index formula (see Section \ref{subsection:grading}) implies that 
\begin{equation}
\label{eq:poincarehopf}
|p_1^+|+ |p_2^+| + \dotsb +|p_n^+| + |p_1^-|+ |p_2^-| + \dotsb +|p_n^-| = 2n - 2.
\end{equation}
By \eqref{align:smoothdiskproduct}, the nontrivial higher products are given by 
\begin{align*}
\mucirc_{2n} (\s p_{i-1}^\pm \otimes \dotsb \otimes \s p_1^\pm \otimes \s p_n^\mp  \otimes \dotsb \otimes \s p_1^\mp \otimes \s p_n^\pm \otimes \dotsb \otimes \s p_{i+1}^\pm \otimes \s p_i^\pm) &= \s \, \id_{\gamma_i^\pm} 
\end{align*}
for each $1\leq i\leq n$, where the indexing is cyclic, e.g.\ $\gamma_0^- = \gamma_n^+$. By \cite{haidenkatzarkovkontsevich} we obtain a triangle equivalence $\tw(\mathbf A_{\widetilde \Gamma'})^\natural \simeq \W (\mathbf D_{2n})$.

Since the dissection $\widetilde \Gamma'$ is $\mathbb Z_2$-invariant, it descends to a collection of arcs $\Gamma' =\{\gamma_1, \dotsc, \gamma_n\}$ on the orbifold disk $\mathbf D_n^\times$. Note that $\Gamma'$ is not an admissible dissection in the sense of Definition \ref{definition:orbifoldstop} as there are no arcs in $\Gamma'$ connecting to the orbifold point as illustrated in Fig.~\ref{fig:doublecoveranotherarcs}. 

Nevertheless, we now show that the collection $\Gamma'$ still generates the partially wrapped Fukaya category $\W (\mathbf D_n^\times)$. The endomorphism algebra $\mathbf A_{\Gamma'} = \W (\mathbf D_n^\times)(\Gamma', \Gamma')$ of $\Gamma'$ is given by the graded quiver of type $\widetilde{\mathrm A}_{n-1}$ modulo the ideal of relations generated by all length $2$ paths. We denote by $\bar p_1, \dotsc, \bar p_n$ its arrows, where $\bar p_i$ lifts to $p_i^\pm$. Up to homotopy, we may choose the line field to be $\mathbb Z_2$-invariant, whence \eqref{eq:poincarehopf} implies $|\bar p_1| + \dotsb + |\bar p_n| = n-1$. The induced A$_\infty$ higher products on $\mathbf A_{\Gamma'}$, which we denote by a new symbol $\muotimes_{2n}$, are given by 
\[
\muotimes_{2n}(\s \bar p_{i-1} \otimes \dotsb \otimes \s \bar p_1 \otimes \s \bar p_{n} \otimes \dotsb \otimes \s \bar p_1 \otimes \s \bar p_n \otimes \dotsb \otimes \s \bar p_{i+1} \otimes  \s \bar p_i ) = \s \, \id_{\gamma_i}
\]
for each $1\leq i \leq n$. For $n \geq 1$, these are precisely the building blocks \eqref{eq:orbifoldbuildingblocks}.

\begin{theorem}
\label{theorem:newdissection}
For all $n \geq 1$ there is an equivalence $\tw (\mathbf A_{\Gamma'}) \simeq \W (\mathbf D_n^\times)$, in particular $\Gamma'$ is a generator of $\W(\mathbf D_n^\times)$. 

We thus have a triangle equivalence $\per (\mathbf A_{\Gamma'}) \simeq \per(A)$, giving a derived equivalence between the algebra of type $\widetilde{\mathrm A}_{n-1}$ with full quadratic relations equipped with the A$_\infty$ structure $\muotimes_{2n}$ and the path algebra $A$ of an ungraded quiver of type $\mathrm D_{n+1}$.
\end{theorem}

\begin{remark}
[Types $\widetilde{\mathrm A}_{n-1}$ and $\mathrm D_{n+1}$ for low values of $n$]
As $\mathrm A_0$ is the empty Dynkin diagram, the Dynkin diagram $\widetilde{\mathrm A}_0$ contains only one vertex and one loop, which is oriented in the corresponding quiver, as illustrated in \eqref{eq:buildingblocks} and \eqref{eq:skewgentlebuildingblock}.

As Dynkin diagrams, $\mathrm D_2$ is isomorphic to two copies of $\mathrm A_1$ and $\mathrm D_3$ isomorphic to $\mathrm A_3$. In particular, the corresponding algebras of type $\mathrm D_2$ and $\mathrm D_3$ are in fact isomorphic to gentle algebras and they belong to the few cases of gentle algebras admitting both a smooth surface model and an orbifold surface model. (See also \cite[Theorem 3.8]{labardinifragososchrollvaldivieso} and \cite[Section 8.4]{barmeierschrollwang}.)

Even though the corresponding algebras are isomorphic, they play a different role in the gluing construction and thus types $\mathrm D$ and $\mathrm A$ should be considered separately. For example, the type $\mathrm D_2$ sector \eqref{eq:skewgentlebuildingblock} has a single arc. Gluing it into a surface leaves the surface connected (unlike $\mathrm A_1 \sqcup \mathrm A_1$) and adds an orbifold point to the surface.
\end{remark}

\begin{proof}[Proof of Theorem \ref{theorem:newdissection}]
Let us consider the A$_\infty$ category of twisted complexes  $\tw(\mathbf A_{\Gamma'})$ of the A$_\infty$ category $\mathbf A_{\Gamma'}$ with objects $\Gamma'$. We construct a twisted complex $\gamma_0$ as 
\[
\gamma_1 \toarg{\bar p_1} \s^{\maltese_1} \gamma_2 \toarg{\bar p_2} \s^{\maltese_2} \gamma_3  \toarg{\bar p_{3}}  \dotsb \toarg{\bar p_{n-1}} \s^{\maltese_{n-1}} \gamma_n
\]
where $\maltese_i = |\bar p_1|+\dotsb+|\bar p_{i}|-i$ with twisted differential $\delta_0 = \bar p_1 + \dotsb + \bar p_{n-1}$. This is a well-defined twisted complex, namely we have (see Definition \ref{definition:twisted})
\[
\mu_i^{\mathbb Z \mathbf A_{\Gamma'}}( \s \delta_0 \otimes \dotsb \otimes \s \delta_0) = 0 \quad \text{for each $i\geq 1$}.
\]
The complex $\gamma_0$ can by represented by the arc illustrated as follows
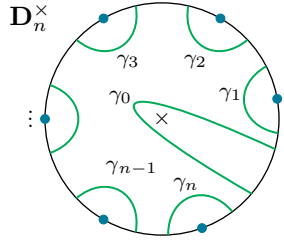
\[
\begin{tikzpicture}[x=1em,y=1em,decoration={markings,mark=at position 0.55 with {\arrow[black]{Stealth[length=4.2pt]}}}]
\draw[line width=.5pt] circle(4em);
\node[font=\small] at (-4.5em,3.5em) {$\mathbf D_n^\times$};
\node[font=\scriptsize,shape=circle,scale=.6,fill=white] (X) at (0,0) {};
\node[font=\scriptsize] at (0,0) {$\times$};
\node[font=\scriptsize] at (177:4.5em) {$.$};
\node[font=\scriptsize] at (180:4.5em) {$.$};
\node[font=\scriptsize] at (183:4.5em) {$.$};
\node[font=\scriptsize] at (152:1.6em) {$\gamma_0$};
\node[font=\scriptsize] at (20:2.5em) {$\gamma_1$};
\node[font=\scriptsize] at (60:2.3em) {$\gamma_{2}$};
\node[font=\scriptsize] at (120:2.3em) {$\gamma_3$};
\node[font=\scriptsize] at (225:2.4em) {$\gamma\mathrlap{_{n-1}}$};
\node[font=\scriptsize] at (290:2.4em) {$\gamma_n$};
\foreach \a in {240,290,10,60,120, 180} {
\draw[fill=stopcolour, color=stopcolour] (\a:4em) circle(.15em);
\path[line width=.75pt,out=\a-160,in=\a+160,looseness=1.5,color=arccolour] (\a+17:4em) edge (\a-17:4em);
}
\draw[line width=.75pt, color=arccolour, line cap=round] (-15:4em) to[out=155, in=150, looseness=9.8] (-40:4em);
\end{tikzpicture}
\]
We compute the A$_\infty$ minimal model of the A$_\infty$ endomorphism algebra $\tw(\mathbf A_{\Gamma'})(\gamma_0, \gamma_0)$. First, note that the cohomology $\H^\bullet (\tw (\mathbf A_{\Gamma'}) (\gamma_0, \gamma_0))$ is $2$-dimensional with a basis $\{\id_{\gamma_0}, \epsilon \}$, where $\id_{\gamma_0}$ corresponds to the identity of $\gamma_0$ and $\epsilon := \s^{|\bar p_n|} \bar p_n$ is given by the vertical map
\begin{equation*}
\begin{tikzpicture}[baseline=-2.6pt,description/.style={fill=white,inner sep=1.75pt}]
\matrix (m) [matrix of math nodes, row sep={3em,between origins}, text height=1.5ex, column sep=2em, text depth=0.25ex, ampersand replacement=\&]
{
\gamma_1 \& \s^{\maltese_1} \gamma_2 \& \cdots \& \s^{\maltese_{n-1}} \gamma_n \\
\& \& \& \gamma_1 \& \s^{\maltese_1} \gamma_2 \& \cdots \& \s^{\maltese_{n-1}} \gamma_n \\
};
\path[->,line width=.4pt]
(m-1-1) edge node[above=-.2ex,font=\scriptsize, overlay] {$\bar p_1$} (m-1-2)
(m-1-2) edge node[above=-.2ex,font=\scriptsize, overlay] {$\bar p_2$} (m-1-3)
(m-1-3) edge node[above=-.2ex,font=\scriptsize, overlay] {$\bar p_{n-1}$} (m-1-4)
(m-1-4) edge node[left=-.2ex,font=\scriptsize] {$\s^{|p_n|} \bar p_n$} (m-2-4)
(m-2-4) edge node[above=-.2ex,font=\scriptsize] {$\bar p_1$} (m-2-5)
(m-2-5) edge node[above=-.2ex,font=\scriptsize] {$\bar p_2$} (m-2-6)
(m-2-6) edge node[above=-.2ex,font=\scriptsize] {$\bar p_{n-1}$} (m-2-7)
;
\end{tikzpicture}
\end{equation*}
Here we use the fact that $\maltese_{n-1}+|\bar p_n|=0$. 
The self-composition of $\epsilon$ in $\tw(\mathbf A_{\Gamma'})$ is given by
\begin{multline*}
\mu_2^{\tw(\mathbf A_{\Gamma'})}(\s^{|\bar p_n|+1} \bar p_n \otimes \s^{|\bar p_n|+1} \bar p_n) \\[.2em]
\begin{aligned}
&= \sum_{i, j\geq 0} \mu_{2n} ( (\s \delta_0)^{\otimes i} \otimes \s^{|\bar p_n|+1} \bar p_n \otimes (\s \delta_0)^{\otimes j} \otimes \s^{|\bar p_n|+1} \bar p_n \otimes (\s \delta_0)^{\otimes {2n-2-i-j}}) \\
&= \sum_{i=0}^{n-1} \mu_{2n} ( \s \bar p_{i ... 1} \otimes \s^{|\bar p_n|+1} \bar p_n \otimes \s \bar p_{n-1 ... 1} \otimes \s^{|\bar p_n|+1} \bar p_n \otimes \s \bar p_{n-1 ... i+1}) \\
&= \sum_{i=0}^{n-1}\s \, \id_{\gamma_{i+1}} = \s \, \id_{\gamma_0}.
\end{aligned}
\end{multline*}

Consider the direct sum $X =\gamma_0 \oplus \gamma_1 \oplus \dotsb \oplus \gamma_{n-1}$ of objects in $\tw(\mathbf A_{\Gamma'})$. Note that $X$ is a new generator of  $\tw(\mathbf A_{\Gamma'})^\natural$.
Moreover, the minimal A$_\infty$ algebra of $\tw(\mathbf A_{\Gamma'})(X, X)$ is a graded algebra given by the quiver
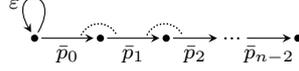
\[
\begin{tikzpicture}[x=3em, y=1em, baseline=-2.6pt,description/.style={fill=white,inner sep=1pt,outer sep=0}]
\draw[line width=.5pt, fill=black] (0,0) circle(.1em);
\draw[line width=.5pt, fill=black] (.75,0) circle(.1em);
\draw[line width=.5pt, fill=black] (1.5,0) circle(.1em);
\draw[line width=.5pt, fill=black] (3,0) circle(.1em);
\node[circle, minimum size=.1em, outer sep=2pt, inner sep=0] (0) at (0,0) {};
\node[circle, minimum size=.1em, outer sep=2pt, inner sep=0] (1) at (.75,0) {};
\node[circle, minimum size=.1em, outer sep=2pt, inner sep=0] (2) at (1.5,0) {};
\node[circle, minimum size=.1em, outer sep=2pt, inner sep=0, font=\scriptsize] (3) at (2.25,0) {...};
\node[circle, minimum size=.1em, outer sep=2pt, inner sep=0] (n-1) at (3,0) {};
\path[-stealth] (0) edge[out=60, in=120, looseness=20] node[pos=.7, font=\scriptsize, left=-.3ex, overlay] {$\epsilon$} (0) (0) edge node[font=\scriptsize, below=-.2ex, overlay] {$\bar p_0$} (1) (1) edge node[font=\scriptsize, below=-.2ex, overlay] {$\bar p_1$} (2) (2) edge node[font=\scriptsize, below=-.2ex, overlay] {$\bar p_2$} (3) (3) edge node[font=\scriptsize, below=-.2ex] {$\bar p_{n-2}$} (n-1);
\draw[line cap=round, dash pattern=on 0pt off 1.2pt, line width=.6pt] (.75,0) +(-.65em,.12em) to[bend left=62, looseness=1.22] +(.55em,.12em) (1.5,0) +(-.65em,.12em) to[bend left=62, looseness=1.22] +(.55em,.12em);
\end{tikzpicture}
\]
with relations $\epsilon^2 = \id_{\gamma_0}$ and $\bar p_i \bar p_{i-1} = 0$ for $1 \leq i \leq n-2$, which is a graded skew-gentle algebra with $\mathrm{Sp} = \{ \epsilon \}$ (see Section \ref{subsection:skewgentle} for the definition). Combining with Theorem \ref{theorem:splitgenerator} (4) and Theorem \ref{corollary:formalgenerator} we obtain that $\per (\mathbf A_{\Gamma'}) \simeq \mathrm D \W (\mathbf D_n^\times)$.

Combining this with Theorem \ref{theorem:splitgenerator} (4), we also obtain the triangle equivalence $\per(\mathbf A_{\Gamma'}) \simeq \per (A)$.
\end{proof}

We may thus revise Definition \ref{definition:orbifoldstop} as follows. 

\begin{definition}
\label{definition:newdissection}
Let $\Gamma$ be a finite collection of arcs on an orbifold disk $\mathbf D_n^\times = (\mathbb D / \mathbb Z_2, \Sigma, \eta)$. We say that $\Gamma$ is a {\it dissection} of $\mathbf D_n^\times$ if the complement $(\mathbb D / \mathbb Z_2) \backslash \Gamma$ is decomposed into its connected components as follows
\[
(\mathbb D / \mathbb Z_2) \smallsetminus \Gamma = \bigsqcup_{i} P_{v_i}
\]
where 
\begin{enumerate}
\item Each $P_{v_i}$ is either a smooth disk containing at most one boundary stop {\it or an orbifold disk containing no boundary stops}.
\item If $P_{v_{i_1}}, \dotsc, P_{v_{i_k}}$ are the connected components around the orbifold point $0$, we either have $k = 1$ such that $P_{v_1}$ contains an orbifold point, or $k \geq 2$ and at least one of the $P_{v_{i_l}}$'s does not contain any boundary stop, whence we may choose one of them to contain the orbifold stop.
\end{enumerate}
\end{definition}

As before, we refer to the $P_{v_i}$'s as {\it polygons} if they are smooth disks. An {\it orbifold polygon} is a polygon with an interior orbifold point.

This definition generalizes to a more general notion of dissection for an arbitrary graded orbifold surface $\mathbf S$. Note that the condition in (2) is empty if all orbifold points are contained in orbifold polygons.

By Theorem \ref{theorem:newdissection} we may associate an A$_\infty$ category $\mathbf A_\Gamma$ to each (new) dissection $\Gamma$ such that $\tw (\mathbf A_\Gamma)^\natural \simeq \W (\mathbf S)$. Note that if the dissection $\Gamma$ gives rise to an orbifold $m$-gon $P_{v_i}$, then $\mathbf A_\Gamma$ is non-formal except when $m = 1$, i.e.\ when the orbifold polygon is a $1$-gon.

\section{Fukaya categories of orbifold disks via deformations}

In this section, we will show that the partially wrapped Fukaya category of an orbifold disk $\mathbf D_n^\times$ can be viewed as an A$_\infty$ deformation of the partially wrapped Fukaya category of a cylinder $\widetilde{\mathbf S}$ with one stop on one boundary component.

\subsection{Hochschild cohomology and A$_\infty$ deformations}\label{subsection:hhdefor}
Let $A$ be a graded algebra. Recall that the Hochschild cochain complex of $A$ is given by 
\[
\mathrm{C}^\bullet(A, A) = \prod_{i \geq 0} \Hom((\s A)^{\otimes i},  A)
\]
equipped with the differential
\begin{equation*}
\begin{split}
\delta(f)(\s a_{n+1...1}) &= -(-1)^{{\maltese_1} |f|}  f(\s a_{n+1 ... 2}) a_1 + (-1)^{|f|+\maltese_{n}} a_{n+1} f(\s a_{n...1}) \\
&\quad-\sum_{i=2}^{n+1}(-1)^{|f| + \maltese_{i-1}} f(\s a_{n+1...i+1} \otimes \s (a_{i} a_{i-1}) \otimes \s a_{i-2...1})
\end{split}
\end{equation*}
where $f \in \Hom((\s A)^{\otimes n+1},  A)$ and $\maltese_i = |a_1| + \dotsb + |a_i| - i$ as usual. The $i$th Hochschild cohomology $\HH^i (A, A)$ of $A$ is defined as the $i$th cohomology of $\mathrm C^\bullet(A, A)$. That is, for any $i \in \mathbb Z$
\[
\HH^i(A, A) = \mathrm H^i (\mathrm C^\bullet(A, A)).
\]

The Gerstenhaber bracket $[-{,} -]$ on $\mathrm{C}^\bullet(A, A)$ plays an essential role for the deformation theory of $A$. Let $f \in \Hom((\s A)^{\otimes m},  A)$ and $g \in \Hom((\s A)^{\otimes n},  A)$. The Gerstenhaber bracket of $f$ and $g$ is given by 
\[
[f, g] = f \bullet g - (-1)^{(|f|-1)(|g|-1)} g \bullet f
\]
where 
$f \bullet g \in \Hom((\s A)^{\otimes m+n-1},  A)$ is given by 
{\small
\begin{align*}
f \bullet g (\s a_{m+n-1...1}) &= 
\sum_{i=1}^{m} (-1)^{\maltese_{i-1} (|g|-1)}  f(\s a_{m+n-1...i+n} \otimes \s g(\s a_{i+n-1...i}) \otimes \s a_{i-1...1}).
\end{align*}
}

\begin{remark}
The multiplication of $A$ induces an element $\mu \in \Hom^2((\s A)^{\otimes 2}, A)$ given by 
\[
\mu(\s a \otimes \s b) = (-1)^{|b|} ab \qquad \text{for $a, b \in A$}.
\]
Then the differential $\delta$ of $\mathrm{C}^\bullet(A, A)$ can be expressed as $\delta(f) = [\mu, f]$.
\end{remark}

\begin{definition}
\label{definition:ainfinitydeformation}
Let $(A, \mu_2)$ be a graded algebra. A (strict) {\it A$_\infty$ deformation} of $(A, \mu_2)$ is an A$_\infty$ algebra $(A, \widetilde \mu_1, \mu_2 + \widetilde \mu_2, \widetilde \mu_3, \dotsc)$ extending the original multiplication $\mu_2$. 

We say that two A$_\infty$ deformations $(A, \mu_2 + \widetilde \mu\,)$ and $(A, \mu_2 + \widetilde \mu')$ are equivalent if they are isomorphic as A$_\infty$ algebras.  Here $\widetilde \mu = \widetilde \mu_1 + \widetilde \mu_2 + \widetilde \mu_3 + \dotsb$.
\end{definition} 

\begin{remark}
If $(A, \mu_2 + \widetilde \mu\,)$ is an A$_\infty$ deformation of $A$ then we may view $\widetilde \mu = \widetilde \mu_1 + \widetilde \mu_2 + \widetilde \mu_3 + \dotsb$ as an element in $\mathrm C^2(A, A)$ (up to shifts). Then $ \mu_2 + \widetilde \mu$ satisfies the A$_\infty$ relations \eqref{eq:ainfinityrelations} if and only if the Maurer--Cartan equation holds (compare to Remark \ref{remark:ainfinity})
\[
\delta(\widetilde \mu) + \frac{1}{2} [ \widetilde \mu, \widetilde \mu] =0.
\]
In deformation theory, one often considers infinitesimal or formal deformations, i.e.\ over a local Artinian $\Bbbk$-algebra such as $\Bbbk [t] / (t^n)$ or over a complete local Noetherian $\Bbbk$-algebra such as $\Bbbk \llbracket t \rrbracket$, in the context of deformations of graded gentle algebras, it is possible to work with the stronger notion of a ``strict'' deformation (see \cite[Section~3]{barmeierschrollwang2}). This strict deformation can be viewed as being obtained from a formal deformation over $\Bbbk \llbracket t \rrbracket$, for instance by considering $(A \llbracket t \rrbracket, t \widetilde \mu_1, \mu + t \widetilde \mu_2, t \widetilde \mu_3, \dotsc)$, checking that all products restrict to $A [t]$ and evaluating $t \mapsto 1$.
\end{remark}

\subsection{A$_\infty$ deformations of graded gentle algebras}
\label{subsection:deformations}

Let $\mathbf S_{1,n} = (S, \Sigma, \eta)$ be a cylinder with one stop on one boundary component whose winding number is $1$, and $n$ stops in the other boundary component, as illustrated in Fig.~\ref{fig:cylinder}.

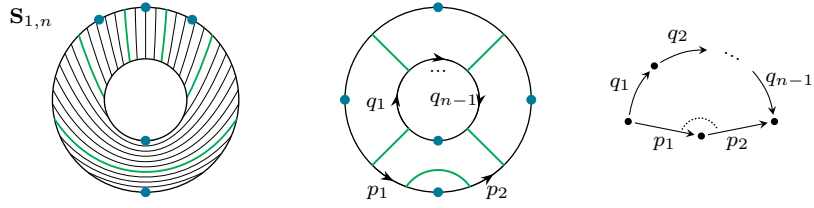
\begin{figure}
\centering
\begin{tikzpicture}[x=1em,y=1em, decoration={markings,mark=at position .66 with {\arrow[black]{Stealth[length=4.8pt]}}}]
\node[font=\scriptsize,left] at (-2.8,2.8) {$\mathbf S_{1,n}$};
\begin{scope}[rotate=-90]
\draw[line width=.5pt] (0,0) circle(3.186em);
\node[font=\scriptsize] at (0,0) {$\times$};
\draw[line width=.2pt] (180:3.18) to (0,0);
\draw[line width=.2pt] (2*19:3.18) to[bend left=26] (-2*19:3.18);
\draw[line width=.2pt] (2*27:3.18) to[bend left=40] (-2*27:3.18);
\draw[line width=.2pt] (2*33.3:3.18) to[bend left=50,looseness=1.07] (-2*33.3:3.18);
\draw[line width=.7pt, color=arccolour, line cap=round] (2*38.9:3.18) to[bend left=57,looseness=1.19] (-2*38.9:3.18);
\draw[line width=.2pt] (2*43.6:3.18) to[bend left=60,looseness=1.34] (-2*43.6:3.18);
\draw[line width=.2pt] (2*48:3.18) to[bend left=62,looseness=1.51] (-2*48:3.18);
\draw[line width=.2pt] (2*52.2:3.18) to[bend left=64,looseness=1.69] (-2*52.2:3.18);
\draw[line width=.2pt] (2*56.2:3.18) to[bend left=66,looseness=1.9] (-2*56.2:3.18);
\draw[line width=.2pt] (2*60:3.18) to[bend left=68,looseness=2.13] (-2*60:3.18);
\draw[line width=.2pt] (2*63.5:3.18) to[bend left=70,looseness=2.39] (-2*63.5:3.18);
\draw[line width=.7pt, color=arccolour, line cap=round] (2*67.1:3.18) to[bend left=72,looseness=2.73] (-2*67.1:3.18);
\draw[line width=.2pt] (2*70.5:3.18) to[bend left=74,looseness=3.14] (-2*70.5:3.18);
\draw[line width=.2pt] (2*73.9:3.18) to[bend left=76,looseness=3.72] (-2*73.9:3.18);
\draw[line width=.2pt] (2*77.1:3.18) to[bend left=78,looseness=4.52] (-2*77.1:3.18);
\draw[line width=.2pt] (2*80.3:3.18) to[bend left=82,looseness=5.78] (-2*80.3:3.18);
\draw[line width=.7pt, color=arccolour, line cap=round] (2*83.5:3.18) to[bend left=86,looseness=8.3] (-2*83.5:3.18);
\draw[line width=.2pt] (2*86.8:3.18) to[bend left=89,looseness=16.15] (-2*86.8:3.18);
\draw[line width=.5pt,fill=white] (0,0) circle(1.41em);
\draw[fill=stopcolour, color=stopcolour] (0:3.186em) circle(.15em);
\draw[fill=stopcolour, color=stopcolour] (150:3.186em) circle(.15em);
\draw[fill=stopcolour, color=stopcolour] (180:3.186em) circle(.15em);
\draw[fill=stopcolour, color=stopcolour] (210:3.186em) circle(.15em);
\draw[fill=stopcolour, color=stopcolour] (0:1.41em) circle(.15em);
\end{scope}
\begin{scope}[xshift=10em]
\node[font=\scriptsize,shape=circle,scale=.6] (X) at (0,0) {};
\draw[line width=0,postaction={decorate}] (225:3.186em) arc[start angle=225, end angle=250, radius=3.186em];
\draw[line width=0,postaction={decorate}] (290:3.186em) arc[start angle=290, end angle=315, radius=3.186em];
\draw[line width=0,postaction={decorate}] (225:1.41em) arc[start angle=225, end angle=145, radius=1.41em];
\draw[line width=0,postaction={decorate}] (100:1.41em) arc[start angle=100, end angle=70, radius=1.41em];
\draw[line width=0,postaction={decorate}] (30:1.41em) arc[start angle=30, end angle=-30, radius=1.41em];
\draw[line width=.5pt] (0,0) circle(3.186em);
\draw[line width=.5pt] (0,0) circle(1.41em);
\draw[fill=stopcolour,color=stopcolour] (-90:1.41em) circle(.15em);
\draw[fill=stopcolour,color=stopcolour] (0:3.186em) circle(.15em);
\draw[fill=stopcolour,color=stopcolour] (90:3.186em) circle(.15em);
\draw[fill=stopcolour,color=stopcolour] (180:3.186em) circle(.15em);
\draw[fill=stopcolour,color=stopcolour] (270:3.186em) circle(.15em);
\path[line width=.75pt,color=arccolour, line cap=round] (315:1.41em) edge (315:3.186em);
\path[line width=.75pt,color=arccolour, line cap=round] (225:1.41em) edge (225:3.186em);
\path[line width=.75pt,color=arccolour, line cap=round] (135:1.41em) edge (135:3.186em);
\path[line width=.75pt,color=arccolour, line cap=round] (45:1.41em) edge (45:3.186em);
\path[line width=.75pt,color=arccolour, line cap=round] (250:3.186em) edge[bend left=55, looseness=1.1] (290:3.186em);
\node[font=\scriptsize] at (184:2.1em) {$q_1$};
\node[font=\tiny] at (90:1em) {...};
\node[font=\scriptsize] at (0:.55em) {$q_{n-1}$};
\node[font=\scriptsize, overlay] at (237.5:3.8em) {$p_1$};
\node[font=\scriptsize, overlay] at (302.5:3.8em) {$p_2$};
\end{scope}
\begin{scope}[xshift=19em, yshift=-.75em]
\draw[line width=.5pt, fill=black] (180:2.5em) circle(.1em);
\draw[line width=.5pt, fill=black] (270:.5em) circle(.1em);
\draw[line width=.5pt, fill=black] (0:2.5em) circle(.1em);
\draw[line width=.5pt, fill=black] (130:2.5em) circle(.1em);
\node[circle, minimum size=.1em, outer sep=2pt, inner sep=0] (0) at (180:2.5em) {};
\node[font=\tiny] at (70:2.5em) {.};
\node[font=\tiny] at (65:2.5em) {.};
\node[font=\tiny] at (60:2.5em) {.};
\node[circle, minimum size=.1em, outer sep=2pt, inner sep=0] (1) at (270:.5em) {};
\node[circle, minimum size=.1em, outer sep=2pt, inner sep=0] (2) at (0:2.5em) {};
\node[circle, minimum size=.1em, outer sep=2pt, inner sep=0] (L) at (130:2.5em) {};
\node[circle, minimum size=.1em, outer sep=2pt, inner sep=0] (T) at (80:2.5em) {};
\node[circle, minimum size=.1em, outer sep=2pt, inner sep=0] (R) at (50:2.5em) {};
\path[-stealth] (0) edge node[font=\scriptsize, below=-.2ex] {$p_1$} (1) (1) edge node[font=\scriptsize, below=-.2ex] {$p_2$} (2) (0) edge[bend left=15] (L) (L) edge[bend left=15] (T) (R) edge[bend left=15] (2);
\draw[line cap=round, dash pattern=on 0pt off 1.2pt, line width=.6pt] (270:.5em) +(-.65em,.24em) to[bend left=62, looseness=1.22] +(.55em,.24em);
\node[font=\scriptsize] at (155:3.05em) {$q_1$};
\node[font=\scriptsize] at (105:3.05em) {$q_2$};
\node[font=\scriptsize] at (25:3.35em) {$q_{n-1}$};
\end{scope}
\end{tikzpicture}
\caption{A cylinder $\mathbf S_{1,n}$ graded by a line field with winding number $1$ around the inner boundary component (left) with a dissection $\Delta$, redrawn (middle) with labels for the boundary paths, and the corresponding category $\mathbf A_\Delta$ (right).}
\label{fig:cylinder}
\end{figure}

Following \cite{haidenkatzarkovkontsevich}, the partially wrapped Fukaya category $\W (\mathbf S)$ of the smooth surface $\mathbf S$ admits formal generators given by admissible dissections. 

In order to study the deformation theory of $\W (\mathbf S)$ we only need a single (formal) generator, since Keller \cite{keller03} shows that the Hochschild complex of $\W (\mathbf S)$ is isomorphic to that of $\mathbf A_\Gamma = \End (\Gamma)$ for any generator $\Gamma$ of $\W (\mathbf S)$. In particular, the results of \cite{keller03} yield the following isomorphism
\begin{equation}
\label{equation:isomorphismhh}
\HH^\bullet (\W (\mathbf S), \W (\mathbf S)) \simeq \HH^\bullet (\End (\Gamma), \End (\Gamma))
\end{equation}

\begin{lemma} 
Let $\mathbf S_{1,n} = (S, \Sigma, \eta)$ be a cylinder with one stop on one boundary component whose winding number is $1$, and $n$ stops on the other boundary component. Then $\HH^2 (\W (\mathbf S_{1,n}), \W (\mathbf S_{1,n}))$ is $1$-dimensional. 
\end{lemma}
\begin{proof}
We may choose the admissible dissection $\Delta$ on $\mathbf S_{1,n}$ illustrated in Fig.~\ref{fig:cylinder}. From the explicit line field given on the left of the figure, one sees that $\mathbf A_\Delta$ is concentrated in degree $0$ (cf.\ the illustrations in Section \ref{subsection:grading}) as those boundary segments where the line field is not transversal to the boundary are segments containing stops.

Using the Bardzell resolution of $\mathbf A_\Delta$, a straightforward computation yields that $\HH^2 (\mathbf A_\Delta, \mathbf A_\Delta)$ is $1$-dimensional. Its basis vector can be represented by the class of the $2$-cocycle 
\begin{align}\label{align:2cocyclemu2}
\widetilde \mu_2 (\s p_2 \otimes \s p_1) = q_n \dotsb q_2 q_1.
\end{align}
The result now follows from the isomorphism \eqref{equation:isomorphismhh}. 
\end{proof}

\begin{remark}
Note that the input of the $2$-cocycle $\widetilde \mu_2$ in \eqref{align:2cocyclemu2} comes from the relation $p_2 p_1 = 0$ in the $3$-gon (the type $\mathrm A_3$ sector graded by the line field) which contains the stop on the inner boundary component as in Fig.~\ref{fig:cylinder}. More generally, one may consider other admissible dissections $\Delta$ on $\mathbf S_{1,n}$. Necessarily, $\HH^2(A_{\Gamma}, A_{\Gamma})$ is still $1$-dimensional, but now its basis vector can be represented by a cocycle $\widetilde \mu_{m-1}$ from the special $m$-gon containing the stop on the inner boundary component.
\end{remark}

\begin{theorem}
Let $\mathbf S_{1,n}$ be a cylinder with one boundary stop $\sigma$ on one boundary component whose winding number is $1$ as above. Let $\Delta_1$ (resp.\ $\Delta_2$) be any two admissible dissections of $\mathbf S_{1,n}$ with one distinguished $m_1$-gon (resp.\ $m_2$-gon) containing $\sigma$. If $m_1, m_2 > 1$ then the algebra $\mathbf A_{\Delta_1}$ (resp.\ $\mathbf A_{\Delta_2}$) admits an A$_\infty$ deformation $(\mathbf A_{\Delta_1}, \mu_2 + \widetilde \mu_{m_1-1})$ (resp.\ $(\mathbf A_{\Delta_2}, \mu_2 + \widetilde \mu_{m_2-1})$). Moreover, we have triangle equivalences 
\[
\per((\mathbf A_{\Delta_1}, \mu_2 + \widetilde \mu_{m_1-1})) \simeq \per ((\mathbf A_{\Delta_2}, \mu_2 + \widetilde \mu_{m_2-1})) \simeq \mathrm D \W (\mathbf D_n^\times),
\]
where $\mathbf D_n^\times$ is the orbifold disk obtained from $\mathbf S_{1,n}$ by partially compactifying the inner boundary component to an orbifold point.
\end{theorem}

\begin{proof}
Any admissible dissection $\Delta$ of $\mathbf S_{1,n}$ such that the stop $\sigma$ is not contained in a $1$-gon gives an admissible dissection $\bar\Delta$ of $\mathbf D_n^\times$ by contracting (``compactifying'') the inner boundary component to an orbifold point. The resulting A$_\infty$ algebra $(\mathbf A_\Delta, \mu_2 + \widetilde\mu_{m-1})$ coincides with $\mathbf A_{\bar\Delta}$ as defined in Section \ref{section:ainfinity}. By Theorem \ref{theorem:moritaequivalenceadmissible} (2) we obtain the desired results.
\end{proof} 

Note that the line field on the cylinder in Fig.~\ref{fig:cylinder} naturally extends to the line field on the orbifold disk in Fig.~\ref{fig:disk}. This is no coincidence, but a particular instance of a general result of \cite{barmeierschrollwang2}, where all A$_\infty$ deformations of $\W (\mathbf S)$ can be described by partial orbifold compactifications of $\mathbf S$, replacing certain boundary components with winding number $1$ by orbifold points and extending the line field to the interior of the resulting orbifold surface, producing a graded orbifold surface in the sense of Definition \ref{definition:orbifoldsurface}.

The fact that algebraic deformations of partially wrapped Fukaya categories of smooth surfaces are related to partial (orbifold) compactifications confirms a general expectation formulated (in the fully wrapped case) in Seidel's ICM 2002 address \cite{seidel1} (see Section \ref{subsection:seidel} for more details and Section \ref{subsection:examples} for concrete examples).

\section{Formal generators and graded skew-gentle algebras}
\label{section:formal}

In this final section, we introduce the notion of a {\it formal} dissection $\Delta$ of an orbifold disk $\mathbf D_n^\times$. In this case the associated A$_\infty$ category $\mathbf A_\Delta$ is formal (Definition \ref{definition:formal}). This notion allows us to construct many more associative algebras which are not skew-gentle themselves, but which are derived equivalent to skew-gentle algebras, see Section \ref{subsection:examples} below for some examples. 

\subsection{DG dissections and formal dissections}
\label{subsection:DGformal}

Let us in fact introduce two further notions of admissible dissections, namely {\it DG dissections} and {\it formal dissections}.  A DG dissection $\Delta$ is such that the corresponding A$_\infty$ category $\mathbf A_\Delta$ is a DG category  and a formal dissection is  such that $\mathbf A_\Delta$ is a formal DG category (Theorem \ref{theorem:formal}).

\begin{figure}
\centering
\begin{tikzpicture}[x=1em,y=1em,decoration={markings,mark=at position .55 with {\arrow[black]{Stealth[length=4.8pt]}}}]
\begin{scope}
\draw[line width=.5pt,postaction={decorate}] ($(252:4.5)+(-1.5,0)$) -- ($(288:4.5)+(1.5,0)$);
\node[font=\small,color=stopcolour] at ($(252:4.5)+(-.75,0)$) {$\bullet$};
\node[font=\small,color=stopcolour] at ($(288:4.5)+(.75,0)$) {$\bullet$};
\node[font=\scriptsize] at (0,0) {$\times$};
\node[font=\scriptsize,shape=circle,scale=.6] (X) at (0,0) {};
\draw[line width=.75pt, color=arccolour, line cap=round] (X) -- (252:4.5);
\draw[line width=.75pt, color=arccolour, line cap=round] (X) -- (288:4.5);
\draw[line width=.75pt, color=arccolour, line cap=round, overlay] (X) -- (190:2em);
\draw[line width=.75pt, color=arccolour, line cap=round, overlay] (X) -- (128:2em);
\draw[line width=.75pt, color=arccolour, line cap=round, overlay] (X) -- (0:2em);
\node[font=\scriptsize] at (79:.5) {.};
\node[font=\scriptsize] at (59:.5) {.};
\node[font=\scriptsize] at (39:.5) {.};
\node[font=\scriptsize,left=-.3ex] at (252:3.5) {$\gamma$};
\node[font=\scriptsize,right=-.3ex] at (288:3.5) {$\gamma'$};
\node[font=\scriptsize,color=stopcolour] at (270:1em) {$\bullet$};
\node[font=\scriptsize, overlay] at (270:5) {$p$};
\draw[->, line width=.5pt] (250:.8em) arc[start angle=250, end angle=-70, radius=.8em];
\node[font=\scriptsize] at (220:1.2) {$q_1$};
\node[font=\scriptsize] at (160:1.2) {$q_2$};
\node[font=\scriptsize] at (-25:1.3) {$q_k$};
\end{scope}
\end{tikzpicture}
\caption{An illustration of the condition ($\ast$) in the definition of a formal dissection. Along the boundary segment to which $p$ belongs, there must be a stop to the left of $\gamma$ and to the right of $\gamma'$.} 
\label{fig:dgformal}
\end{figure}
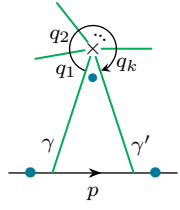

\begin{definition}
\label{definition:DGformal}
Let $\mathbf D_n^\times$ be a graded orbifold disk with $n$ stops and with one orbifold point. A {\it DG dissection} of $\mathbf D_n^\times$ is an admissible dissection $\Delta$ such that each smooth polygon contains exactly one stop. If the dissection contains an orbifold polygon, then we require this to be a $1$-gon. Else, we require the (smooth) polygon containing the orbifold stop to be an $m$-gon with $m \in \{ 2,3 \}$. 

A {\it formal dissection} is a DG dissection $\Delta$ satisfying the following additional condition:
\begin{enumerate}
\item[($\ast$)]
If the dissection cuts $\mathbf D_n^\times$ into smooth polygons and the orbifold stop is contained in a $2$-gon with one boundary path $p$ such that the orbifold path $q_k \dotsb q_2 q_1$ parallel to $p$ has length strictly greater than $1$ (i.e,\ $k > 1$), then $p$ is maximal, i.e.\ $p$ cannot be nontrivially composed with any other path. (See Fig.~\ref{fig:dgformal} for an illustration.) 
\end{enumerate}
\end{definition}

\begin{remark}\label{remark:dgdissectiondg}
Let $\Delta$ be an admissible DG dissection. Since there are neither smooth disk sequences of length $\geq 3$ nor orbifold disk sequences of length $\geq 4$, see Section \ref{subsection:higher}, it follows that there are no nonzero A$_\infty$ products $\mu_i$ with $i>2$ and thus $\mathbf A_\Delta$ is a DG category by Theorem \ref{theorem:moritaequivalenceadmissible}. 
\end{remark}

\begin{theorem}\label{theorem:formal}
Let $\Delta$ be an admissible dissection on a graded orbifold surface $\mathbf S$ with stops. Then $\Delta$ is a formal dissection if and only if the DG algebra $\mathbf A_\Delta$ is formal.
\end{theorem}

We refer to \cite[Theorem 8.6]{barmeierschrollwang} for the proof.

Note that the notion of a formal dissection drastically simplifies for smooth graded surfaces: it is simply a dissection into (possibly punctured) polygons such that each polygon contains exactly one stop. (In the homologically smooth case, these are the ``full formal arc systems'' in the terminology of \cite{haidenkatzarkovkontsevich}.) The somewhat subtle conditions in the case of orbifold surfaces allow for many more types of algebras arising from formal generators of $\W (\mathbf S)$ for graded orbifold surfaces $\mathbf S$ (see Section \ref{subsection:examples} for examples). This contrasts the case of smooth surfaces, where the simplicity of the notion of formal dissection reflects the closure under derived equivalence of gentle algebras which is a result of Schröer and Zimmermann \cite{schroeerzimmermann} in the ungraded case and a folklore conjecture in the graded case (cf.\ Conjecture \ref{conjecture:skewgentle} below).

\subsection{Graded skew-gentle algebras}
\label{subsection:skewgentle}

Let $\Bbbk Q/J$ be a graded gentle algebra and $\mathrm{Sp}$ be a subset of loops $\epsilon_i$ of degree $0$ in $Q$ with $\epsilon_i^2 = 0$. Recall that the {\it graded skew-gentle algebra} \cite{geisspena} associated to $(\Bbbk Q/J,  \mathrm{Sp})$ is defined as $A = \Bbbk Q/ I$ where the ideal $I$ is obtained from $J$ by changing the relations $\epsilon_i^2 = 0$ at each loop $\epsilon_i \in \mathrm{Sp}$ into $\epsilon_i^2 = \epsilon_i$. As we are working over a field of characteristic $0$, we may replace the relation by $\epsilon_i^2 = 1$ which fits more naturally with the geometric perspective, giving \eqref{eq:skewgentlebuildingblock} as a special case of \eqref{eq:orbifoldbuildingblocks}. Note that the relation $\epsilon_i^2 = 1$ is not admissible (i.e.\ not at least quadratic), whence the vertices at which these loops are based represent the sum of two orthogonal idempotents. 

The following result holds for any orbifold surface with stops.

\begin{theorem}[{\cite[Theorem 8.9 and Corollary 8.10]{barmeierschrollwang}}]
\label{corollary:formalgenerator}
Let $\mathbf S = (S, \Sigma, \eta)$ be any graded orbifold surface with stops.
 \begin{enumerate}
 \item There is a formal DG dissection $\Delta$ of $\mathbf S$ such that $\H^\bullet(\mathbf A_\Delta)$ is a graded skew-gentle algebra. 
\item 
For any formal generator $\Gamma$ of $\W (\mathbf S)$, the cohomology of its graded endomorphism algebra in $\W (\mathbf S)$ is a graded associative algebra which is derived equivalent to a graded skew-gentle algebra.
\end{enumerate}
\end{theorem}

The following conjecture can be viewed as a converse of Theorem \ref{corollary:formalgenerator} (2).

\begin{conjecture}[{\cite[Conjecture 8.11]{barmeierschrollwang}}]
\label{conjecture:skewgentle}
Let $A$ be a graded gentle or skew-gentle algebra and let $\mathbf S = (S, \Sigma, \eta)$ be the graded surface associated to $A$.

For any graded associative algebra $B$ which is (perfect) derived equivalent to $A$ there exists a formal dissection $\Delta$ of $\mathbf S$ such that $B \simeq \H^\bullet (\mathbf A_\Delta)$ as graded associative algebras.
\end{conjecture}

\begin{remark}
Conjecture \ref{conjecture:skewgentle} implies that the formal dissections of $\mathbf S$ would provide the complete class of graded associative algebras which are derived equivalent to $A$, giving a geometric solution to an algebraic classification problem. The conjecture has been confirmed for $\mathbf S = \mathbf D_n^\times$ \cite{amiotplamondon2}.
\end{remark}

\subsection{Examples}
\label{subsection:examples}

We conclude this survey by giving a range of concrete examples of dissections of the orbifold disk $\mathbf D_4^\times$ with one orbifold point and four boundary stops. These examples illustrate the utility and versatility of dissections of orbifold surfaces in the classification of derived skew-gentle algebras.

We illustrate the dissections $\Delta$, their associated DG categories $\mathbf A_{\Delta}$ and their minimal models $\H^\bullet (\mathbf A_{\Delta})$. We have shaded the polygon containing the orbifold stop, which gives rise to $\mutimes_1$ or $\mutimes_2$. In these examples, the grading structure can be chosen such that all straight arrows are of degree $0$ and the curved arrows are of degree $-1$.
\begin{equation}
\label{eq:dissection1}
\begin{tikzpicture}[x=1em,y=1em,decoration={markings,mark=at position 0.7 with {\arrow[black]{Stealth[length=4.8pt]}}}, baseline=-.2em]
\begin{scope}
\node[font=\small] at (0, 3.75em) {$\Delta$};
\draw[use as bounding box, draw opacity=0] (-2.5em,-2.5em) rectangle (24.5em,2.5em); 
\draw[fill=arccolour!30!white,line width=0pt] (245:2.5em) arc[start angle=245, end angle=225, radius=2.5em] to (0,0) to (-45:2.5em) arc[start angle=-45, end angle=-75, radius=2.5em] to[bend right=55, looseness=1.3] (255:2.5em) to cycle;
\draw[->, line width=.5pt, line cap=round] (222:.7em) to[out=222-90, in=137+97] (137:.7em);
\draw[->, line width=.5pt, line cap=round] (132:.7em) to[out=132-90, in=47+97] (47:.7em);
\draw[->, line width=.5pt, line cap=round] (42:.7em) to[out=42-90, in=-43+97] (-43:.7em);
\draw[line width=.5pt] circle(2.5em);
\draw[line width=0pt,postaction={decorate}] (225:2.5em) arc[start angle=225, end angle=255, radius=2.5em];
\draw[line width=0pt,postaction={decorate}] (-75:2.5em) arc[start angle=-75, end angle=-45, radius=2.5em];
\node[font=\scriptsize,shape=circle,scale=.6,fill=white] (X) at (0,0) {};
\node[font=\scriptsize] at (0,0) {$\times$};
\draw[line width=.75pt,color=arccolour, line cap=round] (225:2.5em) to (X);
\draw[line width=.75pt,color=arccolour, line cap=round] (135:2.5em) to (X);
\draw[line width=.75pt,color=arccolour, line cap=round] (45:2.5em) to (X);
\draw[line width=.75pt,color=arccolour, line cap=round] (-45:2.5em) to (X);
\draw[line width=.75pt,color=arccolour, line cap=round] (-75:2.5em) to[bend right=55, looseness=1.3] (255:2.5em);
\node[font=\scriptsize, color=stopcolour] at (270:0.6em) {$\bullet$};
\foreach \a in {90, 180, 270, 0} {
\draw[fill=stopcolour, color=stopcolour] (\a:2.5em) circle(.15em);
}
\end{scope}
\begin{scope}[xshift=5.5em, yshift=1em]
\node[font=\small] at (3, 2.75em) {$\mathbf A_\Delta$};
\draw[line width=.5pt, fill=black] (0,0) circle(.1em);
\draw[line width=.5pt, fill=black] (2,0) circle(.1em);
\draw[line width=.5pt, fill=black] (4,0) circle(.1em);
\draw[line width=.5pt, fill=black] (6,0) circle(.1em);
\draw[line width=.5pt, fill=black] (3,-2) circle(.1em);
\node[circle, minimum size=.1em, outer sep=2pt, inner sep=0] (0) at (0,0) {};
\node[circle, minimum size=.1em, outer sep=2pt, inner sep=0] (1) at (2,0) {};
\node[circle, minimum size=.1em, outer sep=2pt, inner sep=0] (2) at (4,0) {};
\node[circle, minimum size=.1em, outer sep=2pt, inner sep=0] (3) at (6,0) {};
\node[circle, minimum size=.1em, outer sep=2pt, inner sep=0] (B) at (3,-2) {};
\path[-stealth] (0) edge (1) (1) edge (2) (2) edge (3) (0) edge (B) (B) edge (3); 
\node[font=\scriptsize] at (3,-.8) {$\mutimes_2$};
\end{scope}
\begin{scope}[xshift=15em]
\node[font=\small] at (3, 3.75em) {$\H^\bullet (\mathbf A_\Delta)$};
\node[font=\footnotesize] at (3, 0) {same as $\mathbf A_\Delta$};
\end{scope}
\end{tikzpicture}
\end{equation}
\begin{equation}
\label{eq:dissection2}
\begin{tikzpicture}[x=1em,y=1em,decoration={markings,mark=at position 0.7 with {\arrow[black]{Stealth[length=4.8pt]}}}, baseline=-.2em]
\begin{scope}
\draw[use as bounding box, draw opacity=0] (-2.5em,-2.5em) rectangle (24.5em,2.5em);
\draw[fill=arccolour!30!white,line width=0pt] (210:2.5em) arc[start angle=210, end angle=240, radius=2.5em] to (0,0) to (210:2.5em) to cycle;
\draw[->, line width=.5pt, line cap=round] (207:.7em) to[out=207-90, in=137+97] (137:.7em);
\draw[->, line width=.5pt, line cap=round] (132:.7em) to[out=132-90, in=47+97] (47:.7em);
\draw[->, line width=.5pt, line cap=round] (42:.7em) to[out=42-90, in=-43+97] (-43:.7em);
\draw[->, line width=.5pt, line cap=round] (-48:.7em) to[out=-48-90, in=-118+97] (-118:.7em);
\draw[line width=.5pt] circle(2.5em);
\draw[line width=0pt,postaction={decorate}] (210:2.5em) arc[start angle=210, end angle=240, radius=2.5em];
\node[font=\scriptsize,shape=circle,scale=.6,fill=white] (X) at (0,0) {};
\node[font=\scriptsize] at (0,0) {$\times$};
\draw[line width=.75pt,color=arccolour, line cap=round] (210:2.5em) to (X);
\draw[line width=.75pt,color=arccolour, line cap=round] (240:2.5em) to (X);
\draw[line width=.75pt,color=arccolour, line cap=round] (135:2.5em) to (X);
\draw[line width=.75pt,color=arccolour, line cap=round] (45:2.5em) to (X);
\draw[line width=.75pt,color=arccolour, line cap=round] (-45:2.5em) to (X);
\node[font=\scriptsize, color=stopcolour] at (226:0.8em) {$\bullet$};
\foreach \a in {90, 180, 270, 0} {
\draw[fill=stopcolour, color=stopcolour] (\a:2.5em) circle(.15em);
}
\end{scope}
\begin{scope}[xshift=5.5em, yshift=.5em]
\draw[line width=.5pt, fill=black] (-.5,0) circle(.1em);
\draw[line width=.5pt, fill=black] (1.25,0) circle(.1em);
\draw[line width=.5pt, fill=black] (3,0) circle(.1em);
\draw[line width=.5pt, fill=black] (4.75,0) circle(.1em);
\draw[line width=.5pt, fill=black] (6.5,0) circle(.1em);
\node[circle, minimum size=.1em, outer sep=2pt, inner sep=0] (0) at (-.5,0) {};
\node[circle, minimum size=.1em, outer sep=2pt, inner sep=0] (1) at (1.25,0) {};
\node[circle, minimum size=.1em, outer sep=2pt, inner sep=0] (2) at (3,0) {};
\node[circle, minimum size=.1em, outer sep=2pt, inner sep=0] (3) at (4.75,0) {};
\node[circle, minimum size=.1em, outer sep=2pt, inner sep=0] (4) at (6.5,0) {};
\path[-stealth] (0) edge (1) (1) edge (2) (2) edge (3) (3) edge (4) (0) edge[bend right=30] (4); 
\node[font=\scriptsize] at (3,-.6) {$\mutimes_1$};
\end{scope}
\begin{scope}[xshift=15em, yshift=.5em]
\draw[line width=.5pt, fill=black] (-.5,0) circle(.1em);
\draw[line width=.5pt, fill=black] (1.25,0) circle(.1em);
\draw[line width=.5pt, fill=black] (3,0) circle(.1em);
\draw[line width=.5pt, fill=black] (4.75,0) circle(.1em);
\draw[line width=.5pt, fill=black] (6.5,0) circle(.1em);
\node[circle, minimum size=.1em, outer sep=2pt, inner sep=0] (0) at (-.5,0) {};
\node[circle, minimum size=.1em, outer sep=2pt, inner sep=0] (1) at (1.25,0) {};
\node[circle, minimum size=.1em, outer sep=2pt, inner sep=0] (2) at (3,0) {};
\node[circle, minimum size=.1em, outer sep=2pt, inner sep=0] (3) at (4.75,0) {};
\node[circle, minimum size=.1em, outer sep=2pt, inner sep=0] (4) at (6.5,0) {};
\path[-stealth] (0) edge (1) (1) edge (2) (2) edge (3) (3) edge (4); 
\draw[line cap=round, dash pattern=on 0pt off 1.2pt, line width=.6pt] (1.25,0) +(-.75em,-.12em) to[bend right=35, looseness=.97] +(4.15em,-.1em);
\end{scope}
\end{tikzpicture}
\end{equation}

\begin{equation}
\label{eq:dissection3}
\begin{tikzpicture}[x=1em,y=1em,decoration={markings,mark=at position 0.7 with {\arrow[black]{Stealth[length=4.8pt]}}}, baseline=-.2em]
\begin{scope}
\draw[use as bounding box, draw opacity=0] (-2.5em,-2.5em) rectangle (24.5em,2.5em); 
\draw[fill=arccolour!30!white,line width=0pt] (210:2.5em) arc[start angle=210, end angle=240, radius=2.5em] to (0,0) to (210:2.5em) to cycle;
\draw[->, line width=.5pt, line cap=round] (207:.7em) to[out=207-90, in=137+97] (137:.7em);
\draw[->, line width=.5pt, line cap=round] (132:.7em) to[out=132-90, in=47+97] (47:.7em);
\draw[->, line width=.5pt, line cap=round] (42:.7em) to[out=42-90, in=-118+97, looseness=1.45] (-118:.7em);
\draw[line width=.5pt] circle(2.5em);
\draw[line width=0pt,postaction={decorate}] (210:2.5em) arc[start angle=210, end angle=240, radius=2.5em];
\draw[line width=0pt,postaction={decorate}] (15:2.5em) arc[start angle=15, end angle=42, radius=2.5em];
\node[font=\scriptsize,shape=circle,scale=.6,fill=white] (X) at (0,0) {};
\node[font=\scriptsize] at (0,0) {$\times$};
\draw[line width=.75pt,color=arccolour, line cap=round] (210:2.5em) to (X);
\draw[line width=.75pt,color=arccolour, line cap=round] (240:2.5em) to (X);
\draw[line width=.75pt,color=arccolour, line cap=round] (135:2.5em) to (X);
\draw[line width=.75pt,color=arccolour, line cap=round] (45:2.5em) to (X);
\draw[line width=.75pt,color=arccolour, line cap=round] (15:2.5em) to[bend right=55, looseness=1.3] (-15:2.5em);
\node[font=\scriptsize, color=stopcolour] at (226:0.8em) {$\bullet$};
\foreach \a in {90, 180, 270, 0} {
\draw[fill=stopcolour, color=stopcolour] (\a:2.5em) circle(.15em);
}
\end{scope}
\begin{scope}[xshift=5.5em, yshift=.25em]
\draw[line width=.5pt, fill=black] (0,0) circle(.1em);
\draw[line width=.5pt, fill=black] (2,0) circle(.1em);
\draw[line width=.5pt, fill=black] (4,0) circle(.1em);
\draw[line width=.5pt, fill=black] (6,0) circle(.1em);
\draw[line width=.5pt, fill=black] (4,1.75) circle(.1em);
\node[circle, minimum size=.1em, outer sep=2pt, inner sep=0] (0) at (0,0) {};
\node[circle, minimum size=.1em, outer sep=2pt, inner sep=0] (1) at (2,0) {};
\node[circle, minimum size=.1em, outer sep=2pt, inner sep=0] (2) at (4,0) {};
\node[circle, minimum size=.1em, outer sep=2pt, inner sep=0] (3) at (6,0) {};
\node[circle, minimum size=.1em, outer sep=2pt, inner sep=0] (T) at (4,1.75) {};
\path[-stealth] (0) edge (1) (1) edge (2) (2) edge (3) (0) edge[bend right=33] (3) (T) edge (2); 
\node[font=\scriptsize] at (3,-.5) {$\mutimes_1$};
\end{scope}
\begin{scope}[xshift=15em, yshift=.25em]
\draw[line width=.5pt, fill=black] (0,0) circle(.1em);
\draw[line width=.5pt, fill=black] (2,0) circle(.1em);
\draw[line width=.5pt, fill=black] (4,0) circle(.1em);
\draw[line width=.5pt, fill=black] (6,0) circle(.1em);
\draw[line width=.5pt, fill=black] (4,1.75) circle(.1em);
\node[circle, minimum size=.1em, outer sep=2pt, inner sep=0] (0) at (0,0) {};
\node[circle, minimum size=.1em, outer sep=2pt, inner sep=0] (1) at (2,0) {};
\node[circle, minimum size=.1em, outer sep=2pt, inner sep=0] (2) at (4,0) {};
\node[circle, minimum size=.1em, outer sep=2pt, inner sep=0] (3) at (6,0) {};
\node[circle, minimum size=.1em, outer sep=2pt, inner sep=0] (T) at (4,1.75) {};
\path[-stealth] (0) edge (1) (1) edge (2) (2) edge (3) (T) edge (2);
\draw[line cap=round, dash pattern=on 0pt off 1.2pt, line width=.6pt] (2,0) +(-.75em,-.12em) to[bend right=35, looseness=.97] +(2.75em,-.1em) (4,0) +(.12em,.75em) to[bend left=25] +(.75em,.12em);
\end{scope}
\end{tikzpicture}
\end{equation}

\begin{equation}
\label{eq:dissection4}
\begin{tikzpicture}[x=1em,y=1em,decoration={markings,mark=at position 0.7 with {\arrow[black]{Stealth[length=4.8pt]}}}, baseline=-.2em]
\begin{scope}
\draw[use as bounding box, draw opacity=0] (-2.5em,-2.5em) rectangle (24.5em,2.5em);
\draw[fill=arccolour!30!white,line width=0pt] (210:2.5em) arc[start angle=210, end angle=240, radius=2.5em] to (0,0) to (210:2.5em) to cycle;
\draw[->, line width=.5pt, line cap=round] (207:.7em) to[out=207-90, in=137+97] (137:.7em);
\draw[->, line width=.5pt, line cap=round] (132:.7em) to[out=132-90, in=47+97] (47:.7em);
\draw[->, line width=.5pt, line cap=round] (42:.7em) to[out=42-90, in=-118+97, looseness=1.45] (-118:.7em);
\draw[line width=.5pt] circle(2.5em);
\draw[line width=0pt,postaction={decorate}] (210:2.5em) arc[start angle=210, end angle=240, radius=2.5em];
\draw[line width=0pt,postaction={decorate}] (240:2.5em) arc[start angle=240, end angle=259, radius=2.5em];
\node[font=\scriptsize,shape=circle,scale=.6,fill=white] (X) at (0,0) {};
\node[font=\scriptsize] at (0,0) {$\times$};
\draw[line width=.75pt,color=arccolour, line cap=round] (210:2.5em) to (X);
\draw[line width=.75pt,color=arccolour, line cap=round] (240:2.5em) to (X);
\draw[line width=.75pt,color=arccolour, line cap=round] (135:2.5em) to (X);
\draw[line width=.75pt,color=arccolour, line cap=round] (45:2.5em) to (X);
\draw[line width=.75pt,color=arccolour, line cap=round] (285:2.5em) to[bend right=55, looseness=1.3] (255:2.5em);
\node[font=\scriptsize, color=stopcolour] at (226:0.8em) {$\bullet$};
\foreach \a in {90, 180, 270, 0} {
\draw[fill=stopcolour, color=stopcolour] (\a:2.5em) circle(.15em);
}
\end{scope}
\begin{scope}[xshift=5.5em, yshift=.5em]
\draw[line width=.5pt, fill=black] (-.5,0) circle(.1em);
\draw[line width=.5pt, fill=black] (1.25,0) circle(.1em);
\draw[line width=.5pt, fill=black] (3,0) circle(.1em);
\draw[line width=.5pt, fill=black] (4.75,0) circle(.1em);
\draw[line width=.5pt, fill=black] (6.5,0) circle(.1em);
\node[circle, minimum size=.1em, outer sep=2pt, inner sep=0] (0) at (-.5,0) {};
\node[circle, minimum size=.1em, outer sep=2pt, inner sep=0] (1) at (1.25,0) {};
\node[circle, minimum size=.1em, outer sep=2pt, inner sep=0] (2) at (3,0) {};
\node[circle, minimum size=.1em, outer sep=2pt, inner sep=0] (3) at (4.75,0) {};
\node[circle, minimum size=.1em, outer sep=2pt, inner sep=0] (4) at (6.5,0) {};
\path[-stealth] (0) edge (1) (1) edge (2) (2) edge (3) (3) edge (4) (0) edge[bend right=35] (3); 
\draw[line cap=round, dash pattern=on 0pt off 1.2pt, line width=.6pt] (4.75,0) +(-.75em,.12em) to[bend left=55, looseness=1.14] +(.75em,.1em);
\node[font=\scriptsize] at (2.125,-.5) {$\mutimes_1$};
\end{scope}
\begin{scope}[xshift=15em, yshift=.5em]
\draw[line width=.5pt, fill=black] (-.5,0) circle(.1em);
\draw[line width=.5pt, fill=black] (1.25,0) circle(.1em);
\draw[line width=.5pt, fill=black] (3,0) circle(.1em);
\draw[line width=.5pt, fill=black] (4.75,0) circle(.1em);
\draw[line width=.5pt, fill=black] (6.5,0) circle(.1em);
\node[circle, minimum size=.1em, outer sep=2pt, inner sep=0] (0) at (-.5,0) {};
\node[circle, minimum size=.1em, outer sep=2pt, inner sep=0] (1) at (1.25,0) {};
\node[circle, minimum size=.1em, outer sep=2pt, inner sep=0] (2) at (3,0) {};
\node[circle, minimum size=.1em, outer sep=2pt, inner sep=0] (3) at (4.75,0) {};
\node[circle, minimum size=.1em, outer sep=2pt, inner sep=0] (4) at (6.5,0) {};
\path[-stealth] (0) edge (1) (1) edge (2) (2) edge (3) (3) edge (4) (0) edge[bend right=35, looseness=1] (4); 
\draw[line cap=round, dash pattern=on 0pt off 1.2pt, line width=.6pt] (1.25,0) +(-.75em,-.12em) to[bend right=35, looseness=.97] +(2.5em,-.1em) (4.75,0) +(-.75em,.12em) to[bend left=55, looseness=1.14] +(.75em,.1em);
\node[font=\scriptsize] at (4,-.7) {$\mu_3$};
\end{scope}
\end{tikzpicture}
\end{equation}

The dissections \eqref{eq:dissection1}--\eqref{eq:dissection3} are formal and $\H^\bullet (\mathbf A_{\Delta})$ is an associative algebra concentrated in degree $0$ which is derived equivalent to a type $\mathrm D_5$ quiver. (See also \cite[Example 6.3]{amiotplamondon2} for further examples of associative algebras in this derived equivalence class.)

The dissection \eqref{eq:dissection4} is not formal, as it violates the condition ($\ast$) in Definition \ref{definition:DGformal} (cf.\ Fig.~\ref{fig:dgformal}). Indeed, $\H^\bullet (\mathbf A_\Delta)$ has a curved arrow of degree $-1$ and a ``new'' higher structure given by $\mu_3$.

Of course the path algebra of a type $\mathrm D_5$ quiver itself can also be obtained from dissections of $\mathbf D_4^\times$:
\begin{equation}
\label{eq:dissection5}
\begin{tikzpicture}[x=1em,y=1em,decoration={markings,mark=at position 0.7 with {\arrow[black]{Stealth[length=4.8pt]}}}, baseline=-.2em]
\begin{scope}
\draw[use as bounding box, draw opacity=0] (-2.5em,-2.5em) rectangle (24.5em,2.5em);
\draw[fill=arccolour!30!white,line width=0pt] (230:2.5em) arc[start angle=230, end angle=260, radius=2.5em] to (0,0) to (230:2.5em) to cycle;
\draw[->, line width=.5pt, line cap=round] (222:.45em) arc[start angle=222, end angle=-80, radius=.45em] to +(198:.1em);
\draw[line width=.5pt] circle(2.5em);
\draw[line width=0pt,postaction={decorate}] (190:2.5em) arc[start angle=190, end angle=207, radius=2.5em];
\draw[line width=0pt,postaction={decorate}] (203:2.5em) arc[start angle=203, end angle=220, radius=2.5em];
\draw[line width=0pt,postaction={decorate}] (216:2.5em) arc[start angle=216, end angle=234, radius=2.5em];
\draw[line width=0pt,postaction={decorate}] (230:2.5em) arc[start angle=230, end angle=260, radius=2.5em];
\node[font=\scriptsize,shape=circle,scale=.6,fill=white] (X) at (0,0) {};
\node[font=\scriptsize] at (0,0) {$\times$};
\draw[line width=.75pt,color=arccolour, line cap=round] (191:2.5em) to[bend right=20] (140:2.5em);
\draw[line width=.75pt,color=arccolour, line cap=round] (203.33:2.5em) to[out=225-180, in=180] (90:1.6em) to[out=0, in=45+180] (45:2.5em);
\draw[line width=.75pt,color=arccolour, line cap=round] (216.66:2.5em) to[out=220-180, in=180] (90:.7em) to[out=0, in=-45+180] (-45:2.5em);
\draw[line width=.75pt,color=arccolour, line cap=round] (230:2.5em) to (X);
\draw[line width=.75pt,color=arccolour, line cap=round] (260:2.5em) to (X);
\node[font=\scriptsize, color=stopcolour] at (246:0.8em) {$\bullet$};
\foreach \a in {90, 180, 270, 0} {
\draw[fill=stopcolour, color=stopcolour] (\a:2.5em) circle(.15em);
}
\end{scope}
\begin{scope}[xshift=5.5em, yshift=0]
\draw[line width=.5pt, fill=black] (0,0) circle(.1em);
\draw[line width=.5pt, fill=black] (2,0) circle(.1em);
\draw[line width=.5pt, fill=black] (4,0) circle(.1em);
\draw[line width=.5pt, fill=black] (5.8,1) circle(.1em);
\draw[line width=.5pt, fill=black] (5.8,-1) circle(.1em);
\node[circle, minimum size=.1em, outer sep=2pt, inner sep=0] (0) at (0,0) {};
\node[circle, minimum size=.1em, outer sep=2pt, inner sep=0] (1) at (2,0) {};
\node[circle, minimum size=.1em, outer sep=2pt, inner sep=0] (2) at (4,0) {};
\node[circle, minimum size=.1em, outer sep=2pt, inner sep=0] (3) at (5.8,1) {};
\node[circle, minimum size=.1em, outer sep=2pt, inner sep=0] (4) at (5.8,-1) {};
\path[-stealth] (0) edge (1) (1) edge (2) (2) edge (3) (3) edge (4) (3) edge[bend left=65, looseness=1.5] (4); 
\draw[line cap=round, dash pattern=on 0pt off 1.2pt, line width=.6pt] (5.8,1) +(200:.75em) to[bend left=100, looseness=2] +(-25:.75em);
\node[font=\scriptsize] at (6.3,.1) {$\mutimes_1$};
\end{scope}
\begin{scope}[xshift=15em]
\draw[line width=.5pt, fill=black] (0,0) circle(.1em);
\draw[line width=.5pt, fill=black] (2,0) circle(.1em);
\draw[line width=.5pt, fill=black] (4,0) circle(.1em);
\draw[line width=.5pt, fill=black] (5.8,1) circle(.1em);
\draw[line width=.5pt, fill=black] (5.8,-1) circle(.1em);
\node[circle, minimum size=.1em, outer sep=2pt, inner sep=0] (0) at (0,0) {};
\node[circle, minimum size=.1em, outer sep=2pt, inner sep=0] (1) at (2,0) {};
\node[circle, minimum size=.1em, outer sep=2pt, inner sep=0] (2) at (4,0) {};
\node[circle, minimum size=.1em, outer sep=2pt, inner sep=0] (3) at (5.8,1) {};
\node[circle, minimum size=.1em, outer sep=2pt, inner sep=0] (4) at (5.8,-1) {};
\path[-stealth] (0) edge (1) (1) edge (2) (2) edge (3) (2) edge (4); 
\end{scope}
\end{tikzpicture}
\end{equation}

\begin{equation}
\label{eq:dissection6}
\begin{tikzpicture}[x=1em,y=1em,decoration={markings,mark=at position 0.7 with {\arrow[black]{Stealth[length=4.8pt]}}}, baseline=-.2em]
\begin{scope}
\draw[use as bounding box, draw opacity=0] (-2.5em,-2.5em) rectangle (24.5em,2.5em);
\draw[fill=arccolour!30!white,line width=0pt] (235:2.5em) to ++(245-180:2.3em) to[bend left=90, looseness=2.2] ($(255:2.5em)+(245-180:2.3em)$) to (255:2.5em) arc[start angle=255, end angle=235, radius=2.5em] to cycle;
\draw[line width=.5pt] circle(2.5em);
\draw[line width=0pt,postaction={decorate}] (190:2.5em) arc[start angle=190, end angle=207, radius=2.5em];
\draw[line width=0pt,postaction={decorate}] (203:2.5em) arc[start angle=203, end angle=224, radius=2.5em];
\draw[line width=0pt,postaction={decorate}] (216:2.5em) arc[start angle=216, end angle=239, radius=2.5em];
\draw[line width=0pt,postaction={decorate}] (230:2.5em) arc[start angle=230, end angle=260, radius=2.5em];
\node[font=\scriptsize,shape=circle,scale=.6] (X) at (0,0) {};
\node[font=\scriptsize] at (0,0) {$\times$};
\draw[line width=.75pt,color=arccolour, line cap=round] (191:2.5em) to[bend right=20] (140:2.5em);
\draw[line width=.75pt,color=arccolour, line cap=round] (205:2.5em) to[out=235-180, in=180] (90:1.6em) to[out=0, in=45+180] (45:2.5em);
\draw[line width=.75pt,color=arccolour, line cap=round] (221:2.5em) to ++(245-180:1.5em) to[out=245-180, in=180] (90:.9em) to[out=0, in=-45+180] (-45:2.5em);
\draw[line width=.75pt,color=arccolour, line cap=round] (235:2.5em) to ++(245-180:2.3em) to[bend left=90, looseness=2.2] ($(255:2.5em)+(245-180:2.3em)$) to (255:2.5em);
\foreach \a in {90, 180, 270, 0} {
\draw[fill=stopcolour, color=stopcolour] (\a:2.5em) circle(.15em);
}
\end{scope}
\begin{scope}[xshift=5.5em, yshift=0]
\draw[line width=.5pt, fill=black] (0,0) circle(.1em);
\draw[line width=.5pt, fill=black] (2,0) circle(.1em);
\draw[line width=.5pt, fill=black] (4,0) circle(.1em);
\draw[line width=.5pt, fill=black] (6,0) circle(.1em);
\node[circle, minimum size=.1em, outer sep=2pt, inner sep=0] (0) at (0,0) {};
\node[circle, minimum size=.1em, outer sep=2pt, inner sep=0] (1) at (2,0) {};
\node[circle, minimum size=.1em, outer sep=2pt, inner sep=0] (2) at (4,0) {};
\node[circle, minimum size=.1em, outer sep=2pt, inner sep=0] (3) at (6,0) {};
\path[-stealth] (0) edge (1) (1) edge (2) (2) edge (3) (3) edge[out=55, in=125, looseness=30] (3); 
\node[font=\scriptsize] at (6,1.4) {$\muotimes_2$};
\end{scope}
\begin{scope}[xshift=15em]
\node[font=\footnotesize] at (-1.5,0) {$\simeq$};
\draw[line width=.5pt, fill=black] (0,0) circle(.1em);
\draw[line width=.5pt, fill=black] (2,0) circle(.1em);
\draw[line width=.5pt, fill=black] (4,0) circle(.1em);
\draw[line width=.5pt, fill=black] (5.8,1) circle(.1em);
\draw[line width=.5pt, fill=black] (5.8,-1) circle(.1em);
\node[circle, minimum size=.1em, outer sep=2pt, inner sep=0] (0) at (0,0) {};
\node[circle, minimum size=.1em, outer sep=2pt, inner sep=0] (1) at (2,0) {};
\node[circle, minimum size=.1em, outer sep=2pt, inner sep=0] (2) at (4,0) {};
\node[circle, minimum size=.1em, outer sep=2pt, inner sep=0] (3) at (5.8,1) {};
\node[circle, minimum size=.1em, outer sep=2pt, inner sep=0] (4) at (5.8,-1) {};
\path[-stealth] (0) edge (1) (1) edge (2) (2) edge (3) (2) edge (4); 
\end{scope}
\end{tikzpicture}
\end{equation}
Here \eqref{eq:dissection5} is a dissection in the sense of Definition \ref{definition:dissection}, where the curved arrow has degree $+1$, and \eqref{eq:dissection6} a dissection in the sense of the (revised) Definition \ref{definition:newdissection}, featuring an orbifold $1$-gon, where the loop has degree $0$.

\begin{remark}[Deformation-theoretic interpretation]
Note that in all of these examples, $\mathbf A_\Delta$ can be viewed as a {\it DG deformation} of the graded gentle algebra obtained by setting $\mutimes_2$, $\mutimes_1$ or $\muotimes_2$ to zero.

For \eqref{eq:dissection1}--\eqref{eq:dissection5} these graded gentle algebras are homologically smooth and have the cylinder $\mathbf S_{1,4}$ (see Fig.~\ref{fig:cylinder}) as surface model, and introducing $\mutimes_2$ or $\mutimes_1$, as the case may be, captures the deformation from $\W (\mathbf S_{1,4})$ to $\W (\mathbf D_4^\times)$ obtained by partially compactifying the boundary component of winding number $1$.

The underlying graded algebra of \eqref{eq:dissection6} obtained by setting $\muotimes_2$ to $0$ also has a cylinder as surface model, but with an entire boundary stop on the inner boundary component. (Note that the corresponding gentle algebra is not homologically smooth.)

However, this difference disappears through deformation. Indeed, through deformations of Fukaya categories of surfaces, orbifold points may arise either from compactifying boundary components with one boundary stop or with an entire boundary stop. (In both cases, the winding number is $1$, thus giving rise to order $2$ orbifold points.) See \cite{barmeierschrollwang2} for a complete description of deformations of $\W (\mathbf S)$.
\end{remark}

\paragraph{\bf Acknowledgements}
It is our great pleasure to thank Sibylle Schroll for the inspiring and fruitful collaboration and the ICRA Organizing and Scientific Committees for inviting the second-named author to present this work. We would also like to thank the editors of these proceedings, and the referee for their thoughtful and constructive comments which significantly improved the presentation.

This work was partially supported by the Deutsche For\-schungs\-ge\-mein\-schaft (DFG) through the project SFB/TRR 191 Symplectic Structures in Geometry, Algebra and Dynamics (Projektnummer 281071066-TRR191) and through the grant WA 5157/1-1, by the National Key R\&D Program of China (2024YFA1013803), and by the National Natural Science Foundation of China (Grant Nos.\ 13004005, 12371043 and 12071137).


\begin{thebibliography}{ABCP10}


\bibitem[AB22]{amiotbruestle}
C.~Amiot, T.~Brüstle,
Derived equivalences between skew-gentle algebras using orbifolds,
Doc.\ Math.\ 27 (2022) 933--982.

\bibitem[AP21]{amiotplamondon}
C.~Amiot, P.-G.~Plamondon,
The cluster category of a surface with punctures via group actions,
Adv.\ Math.\ 389 (2021) 107884.

\bibitem[AP]{amiotplamondon2}
C.~Amiot, P.-G.~Plamondon,
Skew-group $A_\infty$-categories as Fukaya categories of orbifolds,
arXiv:2405.15466 (2024).

\bibitem[APS23]{amiotplamondonschroll}
C.~Amiot, P.-G.~Plamondon, S.~Schroll,
A complete derived invariant for gentle algebras via winding numbers and Arf invariants,
Selecta Math.\ 29 (2023) no.~30.

\bibitem[AL24]{annologvinenko}
B.~Anno, T.~Logvinenko,
Unbounded twisted complexes, 
J.\ Algebra 647 (2024) 794--822. 

\bibitem[AS87]{assemskowronski}
I.~Assem, A.~Skowroński,
Iterated tilted algebras of type $\widetilde A_n$,
Math.\ Z.\ 195 (1987) 269--90.

\bibitem[AG08]{avellaalaminosgeiss}
D.~Avella-Alaminos, C.~Geiss,
Combinatorial derived invariants for gentle algebras,
J.\ Pure Appl.\ Algebra 212 (2008) 228--243.

\bibitem[BSW1]{barmeierschrollwang}
S.~Barmeier, S.~Schroll, Z.~Wang,
Partially wrapped Fukaya categories of orbifold surfaces. 
arXiv:2407.16358 (2024).

\bibitem[BSW2]{barmeierschrollwang2}
S.~Barmeier, S.~Schroll, Z.~Wang,
Deformations of partially wrapped Fukaya categories of surfaces,
arXiv:2512.16354 (2025).

\bibitem[BW]{barmeierwang}
S.~Barmeier, Z.~Wang,
Deformations of path algebras of quivers with relations, forthcoming in Astérisque,
arXiv:2002.10001 (2020).

\bibitem[BM03]{bekkertmerklen}
V.~Bekkert, H.\,A.~Merklen,
Indecomposables in derived categories of gentle algebras,
Algebr.\ Represent.\ Theory 6 (2003) 285--302.

\bibitem[Boc21]{bocklandt}
R.~Bocklandt,
A gentle introduction to Homological Mirror Symmetry,
Cambridge University Press, Cambridge, 2021.

\bibitem[BK90]{bondalkapranov}
A.~Bondal, M.~Kapranov,
Enhanced triangulated categories, 
Mat.\ Sb.\ 181 (1990) 669--683.

\bibitem[BD17]{burbandrozd}
I.~Burban, Yu.~Drozd,
On the derived categories of gentle and skew-gentle algebras: homological algebra and matrix problems,
arXiv:1706.08358 (2017).

\bibitem[CK]{chokim}
C.-H.~Cho, K.~Kim,
Topological Fukaya category of tagged arcs,
arXiv:2404.10294 (2024).

\bibitem[Cho12]{choi}
S.~Choi,
Geometric structures on 2-orbifolds: Exploration of discrete symmetry,
MSJ Memoirs 27,
Mathematical Society of Japan, Tokyo, 2012.

\bibitem[DK18]{dyckerhoffkapranov}
T.~Dyckerhoff, M.~Kapranov,
Triangulated surfaces in triangulated categories,
J.\ Eur.\ Math.\ Soc.\ 20 (2018) 1473--1524.

\bibitem[FOOO09]{fukayaohohtaono}
K.~Fukaya, Y.-G.~Oh, H.~Ohta, K.~Ono,
Lagrangian intersection Floer theory: Anomaly and obstruction I, II, Chapter 10 available at \url{https://www.math.kyoto-u.ac.jp/%7Efukaya/Chapter10071117.pdf},
Studies in Advanced Mathematics 46,
American Mathematical Society, Providence, RI, 2009.

\bibitem[GPS20]{ganatrapardonshende1}
S.~Ganatra, J.~Pardon, V.~Shende,
Covariantly functorial wrapped Floer theory on Liouville sectors,
Publ.\ Math.\ Inst.\ Hautes Études Sci.\ 131 (2020) 73--200.

\bibitem[GPS24]{ganatrapardonshende2}
S.~Ganatra, J.~Pardon, V.~Shende,
Sectorial descent for wrapped Fukaya categories,
J.\ Amer.\ Math.\ Soc.\ 37 (2024) 499--635.

\bibitem[GdlP99]{geisspena}
C.~Geiss, J.\,A.~de la Pe\~{n}a, 
Auslander--Reiten components for clans, 
Boll.\ Soc.\ Mat.\ Mexicana 5 (1999) 307--326.

\bibitem[GLS16]{geisslabardinifragososchroeer}
C.~Geiß, D.~Labardini-Fragoso, J.~Schröer,
The representation type of Jacobian algebras,
Adv.\ Math.\ 290 (2016) 364--452.

\bibitem[HKK17]{haidenkatzarkovkontsevich}
F.~Haiden, L.~Katzarkov, M.~Kontsevich,
Flat surfaces and stability structures,
Publ.\ Math.\ Inst.\ Hautes Études Sci.\ 126 (2017) 247--318.

\bibitem[JSW]{jinschrollwang}
H.~Jin, S.~Schroll, Z.~Wang,
A complete derived invariant and silting theory for graded gentle algebras,
arXiv:2303.17474 (2023).

\bibitem[Kel01]{keller0}
B.~Keller,
Introduction to $A$-infinity algebras and modules, 
Homology Homotopy Appl.\ 3 (2001) 1--35.

\bibitem[Kel03]{keller03}
B.~Keller,
Derived invariance of higher structures on the Hochschild complex,
available at \url{https://webusers.imj-prg.fr/~bernhard.keller/publ/dih.pdf}, 2003.


\bibitem[Kim24]{kim}
K.~Kim,
Topological Fukaya categories of tagged arc systems and derived-tame algebras,
PhD thesis,
Seoul National University, Seoul, 2024.

\bibitem[Kon94]{kontsevich1}
M.~Kontsevich,
Homological algebra of mirror symmetry,
in Proceedings of the International Congress of Mathematicians (Zürich 1994), Vol.~I, 120--139,
Birkhäuser, Basel, 1995.

\bibitem[Kon09]{kontsevich2}
M.~Kontsevich,
Symplectic geometry of homological algebra,
Mathematische Arbeitstagung (Bonn 2009),
available at \url{https://archive.mpim-bonn.mpg.de/id/eprint/1536}, Max Planck Institute for Mathematics, Bonn, 2009.

\bibitem[Lab09]{labardini}
D.~Labardini-Fragoso,
Quivers with potentials associated to triangulated surfaces,
Proc.\ Lond.\ Math.\ Soc.\ 98 (2009) 797--839.

\bibitem[LSV22]{labardinifragososchrollvaldivieso}
D.~Labardini-Fragoso, S.~Schroll, Y.~Valdivieso,
Derived categories of skew-gentle algebras and orbifolds,
Glasg.\ Math.\ J.\ 64 (2022) 649--674.

\bibitem[Lef03]{lefevre}
K.~Lefèvre-Hasegawa,
Sur les A$_\infty$-catégories,
PhD thesis, available as arXiv:math/0310337,
Université Paris 7, Paris, 2003.

\bibitem[LP20]{lekilipolishchuk2}
Y.~Lekili, A.~Polishchuk,
Derived equivalences of gentle algebras via Fukaya categories,
Math.\ Ann.\ 376 (2020) 187--225.

\bibitem[OPS]{opperplamondonschroll}
S.~Opper, P.-G.~Plamondon, S.~Schroll,
A geometric model for the derived category of gentle algebras,
arXiv:1801.09659 (2018).

\bibitem[OZ22]{opperzvonareva}
S.~Opper, A.~Zvonareva,
Derived equivalence classification of Brauer graph algebras,
Adv.\ Math.\ 402 (2022) 108341.

\bibitem[PW]{palmerwoodward}
J.~Palmer, C.~Woodward,
Invariance of immersed Floer cohomology under Lagrangian surgery, forthcoming in Quantum Topol.,
arXiv:1903.01943 (2019).

\bibitem[QZZ]{qiuzhangzhou}
Y.~Qiu, C.~Zhang, Y.~Zhou,
Two geometric models for graded skew-gentle algebras,
arXiv:2212.10369 (2022).

\bibitem[QZ17]{qiuzhou}
Y.~Qiu, Y.~Zhou,
Cluster categories for marked surfaces: punctured case,
Compositio Math.\ 153 (2017) 1779--1819.

\bibitem[Sch08]{schiffler}
R.~Schiffler,
A geometric model for cluster categories of type $\mathrm D_n$,
J.\ Algebr.\ Comb.\ 27 (2008) 1--21.

\bibitem[SZ03]{schroeerzimmermann}
J.~Schröer, A.~Zimmermann,
Stable endomorphism algebras of modules over special biserial algebras,
Math.\ Z.\ 244 (2003) 515--530.

\bibitem[Sch15]{schroll}
S.~Schroll,
Trivial extensions of gentle algebras and Brauer graph algebras,
J.\ Algebra 444 (2015) 183--200.

\bibitem[Sch]{schrollicm}
S.~Schroll,
On geometric models in representation theory,
to appear in Proceedings of the International Congress of Mathematicians (Philadelphia 2026),
arXiv:2601.14396 (2026).

\bibitem[Sei02]{seidel1}
P.~Seidel,
Fukaya categories and deformations,
in Proceedings of the International Congress of Mathematicians (Beijing 2002), Vol.\ II, 351--360,
Higher Ed.\ Press, Beijing, 2002.

\bibitem[Sei08]{seidel2}
P.~Seidel,
Fukaya categories and Picard--Lefschetz theory,
European Mathematical Society, Zürich, 2008.

\end{thebibliography}
\end{document}